\documentclass[oneside,american]{amsart}
\usepackage[T1]{fontenc}
\usepackage{babel}
\usepackage{units}
\usepackage{amstext}
\usepackage{amsthm}
\usepackage{amssymb}
\usepackage{setspace}
\usepackage[unicode=true,pdfusetitle,
 bookmarks=true,bookmarksnumbered=false,bookmarksopen=false,
 breaklinks=false,pdfborder={0 0 1},backref=false,colorlinks=false]
 {hyperref}

\makeatletter
\numberwithin{equation}{section}
\theoremstyle{plain}
\newtheorem{thm}{\protect\theoremname}[section]
\theoremstyle{definition}
\newtheorem{defn}[thm]{\protect\definitionname}
\theoremstyle{remark}
\newtheorem{notation}[thm]{\protect\notationname}
\theoremstyle{remark}
\newtheorem{rem}[thm]{\protect\remarkname}
\theoremstyle{plain}
\newtheorem{prop}[thm]{\protect\propositionname}
\theoremstyle{plain}
\newtheorem{question}[thm]{\protect\questionname}
\theoremstyle{definition}
\newtheorem{problem}[thm]{\protect\problemname}
\theoremstyle{plain}
\newtheorem{cor}[thm]{\protect\corollaryname}
\theoremstyle{remark}
\newtheorem{claim}[thm]{\protect\claimname}
\theoremstyle{plain}
\newtheorem{lem}[thm]{\protect\lemmaname}
\theoremstyle{definition}
\newtheorem{example}[thm]{\protect\examplename}

\@ifundefined{date}{}{\date{}}
\usepackage[margin=1in]{geometry}

\makeatother

\providecommand{\claimname}{Claim}
\providecommand{\corollaryname}{Corollary}
\providecommand{\definitionname}{Definition}
\providecommand{\examplename}{Example}
\providecommand{\lemmaname}{Lemma}
\providecommand{\notationname}{Notation}
\providecommand{\problemname}{Problem}
\providecommand{\propositionname}{Proposition}
\providecommand{\questionname}{Question}
\providecommand{\remarkname}{Remark}
\providecommand{\theoremname}{Theorem}

\begin{document}
\title{Deformations of Reproducing Kernel Hilbert Spaces on Homogeneous Subvarieties
of the Unit Ball}
\author{Yasin Watted}
\email{Yasin.watted@technion.ac.il}
\curraddr{Faculty of Mathematics, Technion--Israel Institute of Technology,
Haifa, Israel.}
\date{07/18/26}
\begin{abstract}
We study the relationships between a subvariety of the open unit ball
in the complex $d$-dimensional space $\mathbb{C}^{d}$, the reproducing
kernel Hilbert space (RKHS) obtained by restricting the Drury-Arveson
space to the variety, and its multiplier algebra.

We show that if two such RKHSs are almost isometrically isomorphic
as RKHSs, their multiplier algebras are likewise almost completely
isometrically isomorphic as multiplier algebras. In such cases, the
underlying varieties are almost automorphically equivalent. For tractable
homogeneous varieties, we further show that if the corresponding multiplier
algebras are almost completely isometrically isomorphic as multiplier
algebras or one variety is almost the image of the other under a unitary
transformation, then the associated RKHSs are almost isometrically
isomorphic as RKHSs.
\end{abstract}

\maketitle
\tableofcontents{}

\section{Introduction}

We study the relationships between a subvariety $V$ of the open unit
ball $B_{d}$ in the $d$-dimensional complex vector space $\mathbb{C}^{d}$,
the reproducing kernel Hilbert space (RKHS) $H_{V}:=\mathcal{H}_{d}^{2}\mid_{V}$
obtained by restricting the Drury-Arveson space $\mathcal{H}_{d}^{2}$
to the variety $V$, and its multiplier algebra $\mathcal{M_{\mathit{V}}\mathit{:=}\textrm{Mult\ensuremath{\left(H_{V}\right)}}}$.
Here and throughout this paper, $d$ denotes an arbitrary positive
integer and $\mathbb{C}^{d}$ is equipped with the standard Hermetian
inner product.

Davidson, Ramsey and Shalit, obtained a result stating that given
two subvarieties $V$ and $W$ of $B_{d}$, $W$ is the image of $V$
under an automorphism of the ball if and only if their multiplier
algebras $\mathcal{M_{\mathit{V}}}$ and $\mathcal{M_{\mathit{W}}}$
are completely isometrically isomorphic \cite[Theorem 4.5]{davidson2015operator}.
\begin{thm}
\label{thm:(isometric Davidson,-Ramsey-and Shalit)}(Davidson, Ramsey
and Shalit, (See \cite[Theorem 4.5 and Theorem 5.9]{davidson2015operator}
and \cite[Section 9]{davidson2011isomorphism})). Let $V$ and $W$
be two varieties in $B_{d}$, with $d<\infty$. Then, the following
are equivalent:
\end{thm}

\begin{enumerate}
\item $H_{V}$ and $H_{W}$ are isometrically isomorphic as RKHSs.
\item $\mathcal{M_{\mathit{V}}}$ and $\mathcal{M_{\mathit{W}}}$ are completely
isometrically isomorphic as multiplier algebras.
\item $W$ is the image of $V$ under an automorphism of the ball.
\end{enumerate}
In this paper, the main problem addressed is inspired by the work
\cite{ofek2021distance} of Ofek, Pandey, and Shalit. They showed,
roughly, that given a finite subset $X$, then a small deformation
of $X$, in the sense of the automorphism invariant Hausdorff distance
(See \cite[Definition 3.6]{ofek2021distance}), leads to a small deformation
of the corresponding reproducing kernel Hilbert spaces (RKHS) $H_{X}$
via a RKHS isomorphism, and vice versa, \cite[Section 5]{ofek2021distance}.
Moreover, they showed that a small deformation of $H_{X}$ via a RKHS
isomorphism leads to a small deformation via a multiplier algebra
isomorphism of its multiplier algebra $\mathcal{M_{\mathit{X}}}$,
and vice versa, \cite[Section 4]{ofek2021distance}. 

We show a similar result for tractable homogeneous varieties (See
Definitions \ref{Definition 1.4} and \ref{Definition 1.5} below).
In other words, for tractable homogeneous varieties we show that Theorem
\ref{thm:(isometric Davidson,-Ramsey-and Shalit)} remains stable
under small deformations. However, while we use the same tools in
\cite{ofek2021distance} to obtain deformations of RKHSs and multiplier
algebras, we use invertible linear transformations which act isometrically
on a fixed homogeneous variety (See \cite[Theorem 7.4 and Lemma 7.5]{davidson2011isomorphism}).
For convenience, we will use the following definition throughout the
paper:
\begin{defn}
We say that a linear map $A:\mathbb{C^{\mathit{d}}}\longrightarrow\mathbb{C^{\mathit{d}}}$
acts isometrically on a subset $X\subseteq\mathbb{C^{\mathit{d}}}$
(or is isometric on $X$) if for every $x\in X$, $\left\Vert Ax\right\Vert =\left\Vert x\right\Vert $.
\end{defn}

We, roughly, state our result for tractable homogeneous varieties
as follows:
\begin{thm}
Let $d\in\mathbb{N}$. Let $V$ be a tractable homogeneous variety
in $B_{d}$, with $\textrm{span}\left(V\right)=\mathbb{C}^{d}$. Let
$A:\mathbb{C}^{d}\longrightarrow\mathbb{C}^{d}$ be an invertible
linear map that acts isometrically on $V$. Set $W:=AV$. Then, the
following are equivalent: 
\end{thm}

\begin{enumerate}
\item $H_{W}$ is obtained by a small deformation of $H_{V}$ via a RKHS
isomorphism. 
\item $\mathcal{M_{\mathit{W}}}$ is obtained by a small deformation of
$\mathcal{M_{\mathit{V}}}$ via a multiplier algebra isomorphism.
\item $W$ is obtained by a small deformation of $V$ via a linear map that
acts isometrically on $V$.
\end{enumerate}
The class of tractable homogeneous varieties is a proper subclass
of the class of homogeneous varieties. But, it contains all the homogeneous
varieties in $B_{d}$ with $d\leq3$, all irreducible homogeneous
varieties, all unions of two irreducible homogeneous varieties, and
more (see Subsection \ref{subsec:Examples-of-tractable} for examples
and details):
\begin{defn}
\label{Definition 1.4} (See \cite[Section 7]{davidson2011isomorphism}
and Notation \ref{Notation 1.6}). Let $d,k\in\mathbb{N}$ and let
$M=\cup_{i=1}^{k}M_{i}$ be a union of linear subspaces of $\mathbb{C}^{d}$
s.t. $\sum_{i=1}^{k}M_{i}=\mathbb{C}^{d}$ . We say that $M$ is a
\textbf{tractable union of linear subspaces of $\mathbb{C}^{d}$},
if one of the following holds: 
\end{defn}

\begin{enumerate}
\item $M=\mathbb{C}^{d}$.
\item There exist $m\geq2$ ($m\in\mathbb{N}$) and a finite union $N=\cup_{i=1}^{m}\cup_{r=1}^{k_{i}}N_{i,r}$
of finite unions of linear subspaces of $\mathbb{C}^{d}$ s.t. the
following hold:
\begin{enumerate}
\item $M\subseteq N$.
\item For every $i\in\left\{ 1,\ldots,m\right\} $, $N_{i}:=\cup_{r=1}^{k_{i}}N_{i,r}$
is a tractable union of linear subspaces of $\textrm{span}N_{i}\cdot$
\item For every $i,j\in\left\{ 1,\ldots,m\right\} $, if $i\neq j$ then
$\left(\textrm{span}N_{i}\right)\cap\left(\textrm{span}N_{j}\right)=\left\{ 0\right\} \cdot$
\item For every $i\in\left\{ 1,\ldots,m\right\} $, $\dim\left(\textrm{span}N_{i}\right)\leq d-1$. 
\item For every invertible linear  transformation $A:\mathbb{C}^{d}\longrightarrow\mathbb{C}^{d}$,
if $A$ is isometric on $M$ then $A$ is isometric on $N$.
\end{enumerate}
\item $M\subseteq E\oplus N$, where $N\subseteq\mathbb{C}^{d}$ is a tractable
union of linear subspaces of $\textrm{span}N$ and $\left\{ 0\right\} \neq E\subseteq\mathbb{C}^{d}$
is a subspace, which is orthogonal to $\textrm{span}N$ s.t. for every
invertible linear  transformation $A:\mathbb{C}^{d}\longrightarrow\mathbb{C}^{d}$,
if $A$ is isometric on $M$ then $A$ is isometric on $E\oplus N$.
\end{enumerate}
\begin{defn}
(See \cite[Section 7]{davidson2011isomorphism}). \label{Definition 1.5}Let
$d\in\mathbb{N}$ and $V\subseteq B_{d}$ be a homogeneous variety
s.t. $\textrm{span}\left(V\right)=\mathbb{C}^{d}$. Suppose that $V=V_{1}\cup\cdot\cdot\cdot\cup V_{k}$,
for some $k\in\mathbb{N}$, is the decomposition of $V$ into irreducible
components. Set $M:=\cup_{i=1}^{k}\textrm{span}\left(V_{i}\right)$\textbf{.}
We say that $V$ is a \textbf{tractable homogeneous variety in} $B_{d}$,
if $M$ is equal to a tractable union of linear subspaces of\textbf{
$\mathbb{C}^{d}$.}
\end{defn}

\begin{notation}
\label{Notation 1.6}Throughout this paper, by a given (finite) union
of (linear) subspaces $M=\cup_{i=1}^{k}M_{i}$ of $\mathbb{C}^{d}$,
we mean that $k\in\mathbb{N}$ and each $M_{i}$ is a linear subspace
of $\mathbb{C}^{d}$ s.t. $M_{i}\nsubseteq M_{j}$ whenever $i\neq j$.
Furthure, if $M=\cup_{i=1}^{k}M_{i}$ is a union of subspaces of $\mathbb{C}^{d}$
and $E\subseteq\mathbb{C}^{d}$ is a subspace, which is orthogonal
to $\textrm{span}M$, then $E\oplus M:=\cup_{i=1}^{k}E\oplus M_{i}$.
\end{notation}

\begin{rem}
The notion of constructing tractable homogeneous varieties was introduced
in \cite[Section 7]{davidson2011isomorphism}. Definition \ref{Definition 1.4}
draws heavily on this notion of construction. In fact, the only fundamental
difference is that the present definition allows any homogeneous variety
that has an extension to a tractable homogeneous variety, in the following
sense, to be also a tractable homogeneous variety: Let $V$ be a homogeneous
variety in $B_{d}$. We say that $V$ has an extension to a tractable
homogeneous variety, if there is a tractable homogeneous variety $W$
that contains $V$ s.t. for every invertible linear  transformation
$A:\mathbb{C}^{d}\longrightarrow\mathbb{C}^{d}$, if $A$ is isometric
on $V$ then $A$ is isometric on $W$. This allows our main results
for homogeneous varieties to cover as many varieties as possible.
\end{rem}

For the general setting of varieties, we show that given a fixed variety
$V$ that contains at least two points, then a small deformation of
$H_{V}$ via a RKHS isomorphism leads to a small deformation of its
multiplier algebra $\mathcal{M_{\mathit{V}}\mathit{:=}\textrm{Mult\ensuremath{\left(H_{V}\right)}}}$
via a multiplier algebra isomorphism. Each one of these cases leads
to a small deformation of $V$ via a multiplier biholomorphism. This
continues a long standing research program on the isomorphism problem
for irreducible complete Pick algebras \cite{agler2000complete}.

In Subsection \ref{subsec:The-Main-Problem}, we formally state the
main problem of this research for the general setting of varieties.
In Subsection \ref{subsec:Main-results}, we present the main results
in the paper. In Subsection \ref{subsec:Background-on-isomorphism problem},
we introduce a short background for the isomorphism problem. For more
background on the isomorphism problem, see the works \cite{davidson2011isomorphism}
and \cite{davidson2015operator} by Davidson, Ramsey and Shalit. Also,
see a survey \cite{salomon2016isomorphism} by Salomon and Shalit
on the isomorphism problem for complete Pick algebras . Also a work
\cite{hartz2017isomorphism} by Hartz on the isomorphism problem for
multiplier algebras of Nevanlinna-Pick spaces. Also, see \cite{deb2024deformations}
and \cite{ahmed2025multiplier}.

In Subsection \ref{subsec:Preleminaries} below, we introduce a short
background on the Drury-Arveson space $\mathcal{H}_{d}^{2}$ and its
multiplier algebra $\mathcal{M}_{d}$. Also, we present the main definitions
and notations for the isomorphism problem. For more background on
the Drury-Arveson space and its multiplier algebra, see for example
\cite{arveson1998subalgebras}, \cite{hartz2023invitation} and \cite{hartz2025operator}.
For more background on the theory of reproducing kernel Hilbert spaces,
see for example \cite{aronszajn1950theory}.

\subsection{Preliminaries \label{subsec:Preleminaries}}

We say that a Hilbert space $H$ is a reproducing kernel Hilbert space
(RKHS) on a set $X$ if the following holds.
\begin{enumerate}
\item $H$ is a linear subspace of $\mathbb{C}^{X}$.
\item For every $x\in X$, the evaluation linear functional $\rho_{x}:H\rightarrow\mathbb{C}$
at $x$, which is given by 
\[
\rho_{x}\left(f\right)=f\left(x\right),\quad f\in H
\]
 is a bounded linear functional on the Hilbert space $H$.
\end{enumerate}
By Riesz theorem it follows that for every $x\in X$ there exists
a unique $k_{x}^{H}\in H$ such that for every $f\in H,$ 
\[
f\left(x\right)=\left\langle f,k_{x}^{H}\right\rangle \cdot
\]
Therefore, 
\[
H=\overline{\textrm{span}\left\{ k_{x}^{H}:x\in X\right\} }\cdot
\]
In particular, for every $y\in X$ , we have 
\[
k_{y}^{H}\left(x\right)=\left\langle k_{y}^{H},k_{x}^{H}\right\rangle ,\quad x\in X\cdot
\]
 For every $y\in X$, $k_{y}^{H}$ is called the \textbf{kernel funcion}
of $H$ at $y$. The function 
\[
k^{H}:X\times X\rightarrow\mathbb{C},
\]
 which is defined by 
\[
k^{H}\left(x,y\right)=\left\langle k_{y}^{H},k_{x}^{H}\right\rangle ,\quad\textrm{for all }x,y\in X
\]
 is called the \textbf{reproducing kernel} of $H$ on $X$.

Let $d\in\mathbb{N}$ and let $B_{d}$ be the open unit ball in $\mathbb{C}^{d}$.
Let
\[
k:B_{d}\times B_{d}\rightarrow\mathbb{C}
\]
 be the positive definite kernel given by
\[
k\left(x,y\right)=\frac{1}{1-\langle x,y\rangle},\quad\textrm{for all }x,y\in B_{d}\cdot
\]
 For every $x\in B_{d}$, let $k_{x}:B_{d}\rightarrow\mathbb{C}$
be the kernel function at $x$, which is given by 
\[
k_{x}\left(y\right)=k\left(y,x\right),\quad\textrm{for all }y\in B_{d}\cdot
\]
 The Drury-Arveson space $\mathcal{H}_{d}^{2}$ is the RKHS \cite{aronszajn1950theory}
of analytic funcions on the open unit ball $B_{d}$ of $\mathbb{C}^{d}$,
which is determined by $k$ and given by
\[
\mathcal{H}_{d}^{2}=\overline{\textrm{span}\left\{ k_{x}:x\in B_{d}\right\} }\cdot
\]

\begin{notation}
Let $n\in\mathbb{N}_{0}$.
\end{notation}

\begin{itemize}
\item For every $x\in\mathbb{C}^{d}$, set 
\[
x^{n}:=\begin{cases}
x^{\otimes n} & n\geq2\\
x & n=1\cdot\\
1 & n=0
\end{cases}
\]
\item For every non-empty subset $X\subseteq\mathbb{C}^{d}$, set 
\[
X^{n}:=\begin{cases}
\text{span\ensuremath{\mathit{\left\{ x^{n}:x\in X\right\} }}} & n\geq2\\
\textrm{span}\left(X\right) & n=1\\
\mathbb{C} & n=0
\end{cases}\cdot
\]
\end{itemize}
\begin{notation}
For every subspace $E\subseteq\mathbb{C}^{d}$, $\mathcal{F}_{s}\left(E\right)$
will be used to denote the infinite direct sum $\mathit{\oplus_{n=0}^{\infty}E^{n}}$;
i.e., $\mathcal{F}_{s}\left(E\right):=\oplus_{n=0}^{\infty}\mathit{E}^{n}$
is the \textbf{symmetric Fock space }over\textbf{ $E$}.
\end{notation}

Note that for every $y\in B_{d}$,
\[
k_{y}\left(x\right)=\frac{1}{1-\left\langle x,y\right\rangle }=\sum_{n=0}^{\infty}\left\langle x,y\right\rangle ^{n}=\sum_{n=0}^{\infty}\left\langle x^{n},y^{n}\right\rangle ,\quad x\in B_{d}.
\]
In fact, every $f\in\mathcal{H}_{d}^{2}$ has a power series of the
form 
\[
f\left(x\right)=\sum_{n=0}^{\infty}\left\langle x^{n},\xi_{n}\right\rangle ,\:x\in B_{d}.
\]
 where $\oplus_{n=0}^{\infty}\xi_{n}\in\mathcal{F}_{s}\left(\mathbb{C}^{d}\right)$.
The map 
\[
J:\mathcal{H}_{d}^{2}\longrightarrow\mathcal{F}_{s}\left(\mathbb{C}^{d}\right)
\]
 that is given by 
\begin{equation}
J\left(\sum_{n=0}^{\infty}\left\langle \cdot,\xi_{n}\right\rangle \right)=\left(\xi_{0},\xi_{1},\ldots\right),\;\sum_{n=0}^{\infty}\left\langle \cdot,\xi_{n}\right\rangle \in\mathcal{H}_{d}^{2},\label{Equation 2.1.1}
\end{equation}
where for every $n\in\mathbb{N}_{0}$ and $z\in B_{d}$, $\left\langle \cdot,\xi_{n}\right\rangle \left(z\right)=\left\langle z^{n},\xi_{n}\right\rangle _{\left(\mathbb{C}^{d}\right)^{n}}$,
is an anti-unitary that maps $\mathcal{H}_{d}^{2}$ onto the symmetric
Fock space $\mathcal{F}_{s}\left(\mathbb{C}^{d}\right)$ \cite{arveson1998subalgebras}.

If $f=\sum_{n=0}^{\infty}\left\langle \cdot,\xi_{n}\right\rangle \in\mathcal{H}_{d}^{2}$
is given, then for every $n\in\mathbb{N}_{0}$, $f_{n}$ will mostly
be used to denote the $n$-th homogeneous component of $f$; i.e.,
$f_{n}:=\left\langle \cdot,\xi_{n}\right\rangle $. Note that, since
$J$ is anti-unitary, we have 
\[
\left\Vert f\right\Vert _{\mathcal{H}_{d}^{2}}^{2}=\sum_{n=0}^{\infty}\left\Vert f_{n}\right\Vert _{\mathcal{H}_{d}^{2}}^{2}
\]
 and for every $n\in\mathbb{N}_{0}$, $\left\Vert f_{n}\right\Vert _{\mathcal{H}_{d}^{2}}=\left\Vert \xi_{n}\right\Vert _{\left(\mathbb{C}^{d}\right)^{n}}\cdot$
\begin{defn}
We say that a subset $V\subseteq B_{d}$ is a\textbf{ }$\textrm{\textbf{subvariety}}$
of $B_{d}$ (or a \textbf{variety} in $B_{d}$), if there exists a
non-empty subset $\mathit{F}\subseteq\mathcal{H}_{d}^{2}$ such that
\[
V=Z\left(F\right):=\left\{ x\in B_{d}:\forall f\in F,\quad f\left(x\right)=0\right\} \cdot
\]
\end{defn}

For any subvariety $V$ of $B_{d}$, set
\[
H_{V}:=\overline{\textrm{span}\left\{ k_{x}:x\in V\right\} }\cdot
\]
 Then, $H_{V}$ is a RKHS subspace of $\mathcal{H}_{d}^{2}$. 
\begin{rem}
To be precise, for every $f,g\in\mathcal{H}_{d}^{2}$, define 
\[
\left\langle f\mid_{V},g\mid_{V}\right\rangle _{\mathcal{H}_{d}^{2}\mid_{V}}:=\left\langle P_{H_{V}}\left(f\right),P_{H_{V}}\left(g\right)\right\rangle _{\mathcal{H}_{d}^{2}}\cdot
\]
 Then, $H_{V}$ is isometrically isomorphic to the RKHS $\mathcal{H}_{d}^{2}\mid_{V}$
on $V$, which is determined by the kernel $k\mid_{V\times V}$. In
fact, 
\[
\mathcal{H}_{d}^{2}\mid_{V}=H_{V}\mid_{V}=\overline{\textrm{span}\left\{ k_{x}\mid_{V}:x\in V\right\} }\cdot
\]
For the above reason, despite of the fact that by definition $H_{V}$
is a Hilbert subspace of $\mathcal{H}_{d}^{2}$, we will use $H_{V}$
to denote both the Hilbert subspace $\overline{\textrm{span}\left\{ k_{x}:x\in V\right\} }$
of $\mathcal{H}_{d}^{2}$ and the RKHS $\mathcal{H}_{d}^{2}\mid_{V}$
on $V$.
\end{rem}

By a \textbf{multiplier} of $\mathcal{H}_{d}^{2}$ we mean a complex-valued
$\textrm{function}$ $f:B_{d}\longrightarrow\mathbb{C}$ with the
property $f\cdot\mathcal{H}_{d}^{2}\subseteq\mathcal{H}_{d}^{2}$.
Let $\mathcal{M}_{d}$ denote the set of all multipliers of $\mathcal{H}_{d}^{2}$,
i.e. 
\[
\mathcal{M_{\mathit{d}}\mathit{:=}\textrm{Mult\ensuremath{\left(\mathcal{H}_{d}^{2}\right)}}}=\left\{ f:B_{d}\rightarrow\mathbb{C}:fg\in\mathcal{H}_{d}^{2}\quad\text{for all \ensuremath{\mathit{g\in\mathcal{H}_{d}^{2}}}}\right\} \cdot
\]
 Since $\mathcal{H}_{d}^{2}$ contains the constant function ${\bf 1}$,
it follows that $\mathcal{M}_{d}\subseteq\mathcal{H}_{d}^{2}$ (as
sets). Every $f\in\mathcal{M}_{d}$ defines a unique bounded multiplication
operator $M_{f}$ on $\mathcal{H}_{d}^{2}$, 
\[
M_{f}\left(g\right)=fg,\,g\in\mathcal{H}_{d}^{2}.
\]
 The multiplier norm of $f\in\mathcal{M}_{d}$ is defined by 
\[
\left\Vert f\right\Vert {}_{\mathcal{M}_{d}}=\left\Vert M_{f}\right\Vert \cdot
\]
This gives $\mathcal{M}_{d}$ a structure of an operator algebra identified
with $\left\{ M_{f}:f\in\mathcal{M}_{d}\right\} $ via the isometric
isomorphism 
\[
\mathcal{M}_{d}\rightarrow\left\{ M_{f}:f\in\mathcal{M}_{d}\right\} \subseteq B\left(\mathcal{H}_{d}^{2}\right),\hfill f\rightarrow M_{f}\cdot
\]
Since $1\in\mathcal{M}_{d}$, it follows that for every $f\in\mathcal{M}_{d}$,
\[
\left\Vert M_{f}\right\Vert \geq\left\Vert M_{f}\left(1\right)\right\Vert =\left\Vert f\right\Vert _{\mathcal{H}_{d}^{2}}\cdot
\]
 By \cite[Section 2]{arveson1998subalgebras}, we have
\[
M_{f}^{*}\left(k_{x}\right)=\overline{f\left(x\right)}k_{x},x\in B_{d}\cdot
\]
It follows that $\left\Vert f\right\Vert {}_{\infty}\leq\left\Vert M_{f}\right\Vert $.
Hence, $\mathcal{M}_{d}\subseteq H^{\infty}\left(B_{d}\right)$, where
$H^{\infty}\left(B_{d}\right)$ denotes the bounded analytic functions
on $B_{d}$. Moreover, $\left\{ M_{f},f\in\mathcal{M}_{d}\right\} $
is a WOT-closed operator subalgebra of $B\left(\mathcal{H}_{d}^{2}\right)$.
In particular, $\left\{ M_{f},f\in\mathcal{M}_{d}\right\} $ is a
norm-closed operator subalgebra of $B\left(\mathcal{H}_{d}^{2}\right)$.
Thus, $\mathcal{M}_{d}$ is a unital commutative Banach algebra of
complex valued and bounded analytic functions defined on $B_{d}$.
\begin{notation}
Let $Z_{1},\ldots,Z_{d}$ be the coordinate functions on $B_{d}$;
i.e., for every $i,1\leq i\leq d$,
\[
Z_{i}\left(x_{1},\ldots,x_{d}\right)=x_{i},\quad\left(x_{1},\ldots,x_{d}\right)\in B_{d}\cdot
\]
 The $d$-tuple of operators 
\[
M_{Z}:=\left(M_{Z_{1}},\ldots,M_{Z_{d}}\right)
\]
 is called the d-shift.
\end{notation}

By the proof of \cite[Proposition 2.6]{arveson1998subalgebras}, each
$Z_{i}$ is a multiplier of $\mathcal{H}_{d}^{2}$. This yields the
fact that every polynomial $p:B_{d}\longrightarrow\mathbb{C}$ is
a multiplier of $\mathcal{H}_{d}^{2}$. That is, $\mathbb{C}\left[Z\right]\subseteq\mathcal{M}_{d}$.
The component operators of $M_{Z}$ are mutually commuting contractions
in $B\left(\mathcal{H}_{d}^{2}\right)$, which satisfy 
\[
\parallel\sum_{i=1}^{d}M_{Z_{i}}f_{i}\parallel^{2}\leq\sum_{i=1}^{d}\parallel f_{i}\parallel^{2}
\]
 for all $f_{1},\ldots f_{d}\in\mathcal{H}_{d}^{2}$. That is, 
\[
M_{Z}:\oplus_{i=1}^{d}\mathcal{H}_{d}^{2}\longrightarrow\mathcal{H}_{d}^{2}
\]
 is a d-contraction \cite{arveson1998subalgebras}. Denote by $\overline{alg\left(M_{Z}\right)}^{WOT}$
the unital WOT-closed operator subalgebra of $B\left(\mathcal{H}_{d}^{2}\right)$
generated by the d-shift. Then 
\[
\overline{alg\left(M_{Z}\right)}^{WOT}=\left\{ M_{f}:f\in\mathcal{M}_{d}\right\} .
\]
Hence, we have the identification $\mathcal{M}_{d}\cong\overline{alg\left(M_{Z}\right)}^{WOT}$
(See \cite[section 3]{hartz2025operator}). The following result from
\cite{chu2022gleason} says that every kernel function of the Drury-Arveson
space is a multiplier:
\begin{prop}
(See \cite[Proposition 2.1]{chu2022gleason}). For every $x\in B_{d}$,
$k_{x}\in\mathcal{M}_{d}$. In particular, $\mathcal{M}_{d}$ is dense
in $\mathcal{H}_{d}^{2}$.
\end{prop}

As any RKHS, $H_{V}$ comes with its multiplier algebra $\mathcal{M_{\mathit{V}}\mathit{:=}\textrm{Mult\ensuremath{\left(H_{V}\right)}}}$.
This is the set of the restrictions of the multipliers of the Drury-Arveson
space to $V$, i.e., 
\[
\mathcal{M_{\mathrm{\mathit{V}}}}=\left\{ f\mid_{V}:f\in\mathcal{M}_{d}\right\} ,
\]
see Proposition \ref{Proposition 1.14} below. Moreover, $\mathcal{M}_{V}$
is a unital semisimple commutative Banach algebra. Therefore, for
any two varieties $V,W\subseteq B_{d}$, if $\varphi:\mathcal{M_{\mathit{V}}}\rightarrow\mathcal{M_{\mathit{W}}}$
is an algebra isomorphism then, as any homomorphism from a commutative
Banach algebra into a semisimple commutative Banach algebra, $\varphi$
is bounded. Therefore, for every $n\in\mathbb{N}$, $\varphi$ induces
a bounded isomorphism
\[
\varphi^{\left(n\right)}:M_{n}\left(\mathcal{M}_{V}\right)=M_{V}\otimes M_{n}\longrightarrow M_{n}\left(\mathcal{M}_{W}\right)=M_{W}\otimes M_{n}
\]
defined by 
\[
\varphi^{\left(n\right)}\left[f_{i,j}\right]_{i,j=1}^{n}=\left[\varphi\left(f_{i,j}\right)\right]_{i,j=1}^{n},
\]
Where $M_{n}$ is the $n$-square complex matrix. Set 
\[
\left\Vert \varphi\right\Vert _{cb}:=\sup_{n\in\mathbb{N}}\left\Vert \varphi^{\left(n\right)}\right\Vert \cdot
\]
 Then, we say that
\begin{itemize}
\item $\varphi$ is completely bounded if $\left\Vert \varphi\right\Vert _{cb}<\infty$.
\item $\varphi$ is completely contractive if $\left\Vert \varphi\right\Vert _{cb}\leq1$.
\item $\varphi$ is completely isometric if for every $n\in\mathbb{N}$,
$A\in M_{n}\left(\mathcal{M}_{V}\right)$ implies 
\[
\left\Vert \varphi^{\left(n\right)}\left(A\right)\right\Vert _{M_{n}\left(\mathcal{M}_{W}\right)}=\left\Vert A\right\Vert _{M_{n}\left(\mathcal{M}_{V}\right)}\cdot
\]
\end{itemize}
\begin{defn}
We say that a function $F=\left(f_{1},\ldots,f_{d}\right)$ is a \textbf{vector
valued multiplier} \textbf{on $V$}, if $f_{i}\in\mathcal{M}_{V}$
for all $i\in\left\{ 1,\ldots,d\right\} $. The multiplier norm of
$F$ is defined to be 
\[
\parallel F\parallel_{Mult\left(V\right)}=\parallel M_{F}\parallel,
\]
where $M_{F}$ is the row operator 
\[
M_{F}:=\left(M_{f_{1}},\ldots,M_{f_{d}}\right):\oplus_{i=1}^{d}H_{V}\longrightarrow H_{V}
\]
 which is given by 
\[
M_{F}\left(g_{1},\ldots,g_{d}\right)=\sum_{i=1}^{d}f_{i}g_{i},\quad\left(g_{1},\ldots,g_{d}\right)\in\oplus_{i=1}^{d}H_{V}\ldotp
\]
\end{defn}

\begin{rem}
\label{ Remark 1.2}By the following proposition, it follows that
any vector valued multiplier $F$ on a variety $V$ has an extension
to a vector valued multiplier $\overset{\sim}{F}$ on $B_{d}$, $\overset{\sim}{F}\mid_{V}=F$,
with 
\[
\parallel\overset{\sim}{F}\parallel_{Mult\left(B_{d}\right)}=\parallel F\parallel_{Mult\left(V\right)}\cdot
\]
\end{rem}

\begin{prop}
\label{Proposition 1.14}(\cite[Proposition 2.6]{davidson2015operator}).
Let $V\subseteq B_{d}$ be a variety. Then 
\[
\mathcal{M_{\mathrm{\mathit{V}}}}=\left\{ f\mid_{V}:f\in\mathcal{M_{\mathit{d}}}\right\} .
\]
Moreover, the mapping $\varphi:\mathcal{M_{\mathit{d}}}\rightarrow\mathcal{M_{\textrm{\ensuremath{\mathit{V}} }}}$,which
is given by $\varphi\left(f\right)=f\mid_{V}$ induces a completely
isometric isomorphism and a WOT-homeomorphism of $\nicefrac{\mathcal{M}_{d}}{\ker\varphi}$
onto $\mathcal{M_{\mathit{V}}}$. For any $g\in\mathcal{M_{\mathit{V}}}$
and any $f\in\mathcal{M_{\mathit{d}}}$ such that $f|_{V}=g$, we
have $M_{g}=P_{H_{V}}M_{f}|_{\mathcal{H}_{V}}$. Given any $F\in M_{k}\left(\mathcal{M}_{V}\right)$,
one can choose $\stackrel{\sim}{F}\in M_{k}\left(\mathcal{M}_{d}\right)$
so that $\overset{\sim}{F}\mid_{V}=F$ and $\parallel\overset{\sim}{F}\parallel=\parallel F\parallel$.
\end{prop}

\begin{defn}
Let $V$ and $W$ be varieties in $B_{d}$. Then, 
\end{defn}

\begin{itemize}
\item We say that $V$ and $W$ are\textbf{ biholomorphic}, if there exist
holomorphic maps $F:B_{d}\rightarrow\mathtt{\mathbb{C}}^{d}$ and
$G:B_{d}\rightarrow\mathbb{C}^{d}$ such that $G\circ F|_{V}={\bf id}_{V}$
and $F\circ G|_{W}={\bf id}_{W}$. Such $F$ is called a \textbf{biholomorphism
of $V$ onto W.}
\item We say that $V$ and $W$ are \textbf{multiplier biholomorphic,} if
there are vector valued multipliers $F$ and $G$, respectively, on
$V$ and $W$, such that $G\circ F={\bf id}_{V}$ and $F\circ G={\bf id}_{W}$.
Such $F$ is called a \textbf{multiplier biholomorphism of $V$onto
$W$.} (Note that, by Remark \ref{ Remark 1.2} and the definition
of a biholomorphism, every multiplier biholomorphic varieties are
biholomorphic).
\item We say that $V$ and $W$are \textbf{equivalent,} if they are biholomorphic
via an automorphism, $F:B_{d}\rightarrow B_{d}$, of the open unit
ball $B_{d}$.
\end{itemize}
\begin{notation}
Let $V$ and $W$ be two varieties in $B_{d}$ and let $\varphi:\mathcal{M_{\mathit{V}}}\rightarrow\mathcal{M_{\mathit{W}}}$
be an algebra isomorphism. In such case, set 
\[
G_{\varphi}:=\left(\varphi\left(Z_{1}\mid_{V}\right),\ldots,\varphi\left(Z_{d}\mid_{V}\right)\right)
\]
 and 
\[
G_{\varphi^{-1}}:=\left(\varphi^{-1}\left(Z_{1}\mid_{W}\right),\ldots,\varphi^{-1}\left(Z_{d}\mid_{W}\right)\right)\cdot
\]
\end{notation}

\begin{rem}
\label{Remark 1.1200}Let $V$ and $W$ be two varieties in $B_{d}$,
and let $\varphi:\mathcal{M_{\mathit{V}}}\rightarrow\mathcal{M_{\mathit{W}}}$
be an algebra isomorphism. Note that by definition $G_{\varphi}$
is a vector valued multiplier on $W$.
\end{rem}

The authors of \cite{ofek2021distance} introduce the notion multiplier
algebra isomorphism as follows.
\begin{defn}
\label{Definition 1.18}(\cite{ofek2021distance}). Let $V$ and $W$
be two varieties in $B_{d}$. Let $\varphi:\mathcal{M_{\mathit{V}}}\rightarrow\mathcal{M_{\mathit{W}}}$
be a completely bounded algebra isomorphism with a completely bounded
inverse. Then, $\varphi$ is called a \textbf{multiplier algebra isomorphism}
if $\varphi$ is given by 
\[
\varphi\left(f\right)=f\circ G,\quad f\in\mathcal{M_{\mathit{V}}},
\]
where $G:W\rightarrow V$ is a bijection.
\end{defn}

\begin{rem}
\label{Remark 01.19}Note that if $\varphi:\mathcal{M_{\mathit{V}}}\rightarrow\mathcal{M_{\mathit{W}}}$
is an algebra isomorphism, which is given by
\[
\varphi\left(f\right)=f\circ G,\quad f\in\mathcal{M_{\mathit{V}}},
\]
where $G:W\rightarrow V$ is a bijection, then $\varphi^{-1}:\mathcal{M_{\mathit{V}}}\rightarrow\mathcal{M_{\mathit{W}}}$
is an algebra isomorphism that is given by 
\[
\varphi^{-1}\left(f\right)=f\circ G^{-1},\quad f\in\mathcal{M_{\mathit{W}}}\cdot
\]
 Moreover, by definition we have $G=G_{\varphi}$ and $G^{-1}=G_{\varphi^{-1}}\cdot$
Therefore, by Remark \ref{Remark 1.1200}, $G$ is a multiplier biholomorphism
of $W$onto $V$.
\end{rem}

\begin{defn}
\label{Definition 1.20}(\cite{ofek2021distance}). Let $V$ and $W$
be two varieties in $B_{d}$. A bounded invertible linear map $T:\mathcal{\mathit{H}}_{V}\rightarrow\mathcal{\mathit{H}}_{W}$
is said to be a \textbf{RKHS isomorphism}, if there exists a bijection
$F:V\longrightarrow W$, and a nowhere-vanishing function $\lambda:V\rightarrow\mathbb{C}$
such that 
\[
T\left(k_{x}\right)=\lambda\left(x\right)k_{F(x)},\quad x\in V.
\]
If $T$ is a unitary, then we say that $T$ is an \textbf{isometric
RKHS isomorphim}. 
\end{defn}

\begin{rem}
Note that if a RKHS isomorphism $T:H_{V}\longrightarrow H_{W}$ given
by
\[
T\left(k_{x}\right)=\lambda\left(x\right)k_{F(x)},\quad x\in V,
\]
then 
\[
T^{-1}:\mathcal{\mathit{H}}_{W}\rightarrow\mathcal{\mathit{H}}_{V}
\]
 is a RKHS isomrphism such that
\[
T^{-1}\left(k_{y}\right)=\frac{1}{\lambda\circ F^{-1}\left(y\right)}k_{F^{-1}\left(y\right)},\quad y\in W.
\]
Moreover, by Propostion \ref{Proposition 1.29} below and Remark \ref{Remark 01.19}
, the bijection $F:V\to W$ is a multiplier biholomorphism of $V$
onto $W$.
\end{rem}

\begin{notation}
${\bf Aut\left(\mathit{B_{d}}\right)}$ will denote the group of the
automorphisms of $B_{d}$. For every $a\in B_{d}$, $\Phi_{a}$ denotes
the elementary automorphism of $B_{d}$, that interchanges $a\text{\  and \  0}$,
which is given by
\begin{align*}
\Phi_{a}\left(x\right) & =\frac{a-P_{a}x-s_{a}\left({\bf Id}-P_{a}\right)x}{1-\langle x,a\rangle}\quad x\in B_{d},
\end{align*}
 where $P_{a}x=\begin{cases}
\frac{\langle x,a\rangle}{\langle a,a\rangle}a & a\neq0\\
0 & a=0
\end{cases}$ and $s_{a}=\left(1-\parallel a\parallel^{2}\right)^{\frac{1}{2}}$
(see, \cite[Section 2.2]{rudin2008function}). 
\end{notation}

\subsection{Statement of the Main Problem\label{subsec:The-Main-Problem}}

Inspired by the work \cite{ofek2021distance}, we state our initial
research question, roughly, as follows (See \cite[Theorem 5.4]{ofek2021distance}):
\begin{question}
\label{Question 1.23}Let $V$ and $W$ be two varieties in $B_{d}$.
Are the following equivalent in some sense?
\end{question}

\begin{enumerate}
\item $V$ and $W$ are almost the image of the other under an automorphism
of the ball.
\item $\mathcal{M}_{V}$ and $\mathcal{M}_{W}$ are almost isometrically
isomorphic as multiplier algebras.
\item $H_{V}$ and $H_{W}$ are almost isometrically isomorphic as RKHS.
\end{enumerate}
To treat Question \ref{Question 1.23}, we consider the following
two facts:
\begin{itemize}
\item If $T:N_{1}\longrightarrow N_{2}$ is an invertible linear operator
between two normed spaces, $N_{1}$ and $N_{2}$, then $T$ is an
isometric isomorphism if and only if 
\[
\left\Vert T\right\Vert =\left\Vert T^{-1}\right\Vert =1\cdot
\]
\item If $V$ and $W$ are two varieties in $B_{d}$, with $d<\infty$,
each contains at least two pints, and $F:V\longrightarrow W$ is a
multiplier biholomorphism of $V$ onto $W$, then $F$ is a restriction
of an automorphism of $B_{d}$ to $V$ if and only if $\left\Vert F\right\Vert _{Mult\left(V\right)}=1$
and $\left\Vert F^{-1}\right\Vert _{Mult\left(W\right)}=1$ (See Theorem
\ref{Theorem 7.5}).
\end{itemize}
Now, we state the main problem treated in this paper, in the general
setting of varieties, as follows:
\begin{problem}
\label{Problem 1.24}Let $V\subseteq B_{d}$ be a variety that contains
at least two points, and let $\left(W_{n}\right)_{n\in\mathbb{N}}$
be a sequence of varieties in $B_{d}$. Are the following equivalent?
\end{problem}

\begin{enumerate}
\item There exists a sequence of RKHS isomorphisms $T_{n}:H_{V}\longrightarrow H_{W_{n}}$
such that 
\[
\left\Vert T_{n}\right\Vert \left\Vert T_{n}^{-1}\right\Vert \stackrel{n\rightarrow\infty}{\longrightarrow}1\cdot
\]
\item There exists a sequence of multiplier algebra isomorphisms $\varphi_{n}:\mathcal{M_{\mathit{V}}}\rightarrow\mathcal{M_{\mathit{W_{n}}}}$
such that 
\[
\left\Vert \varphi_{n}\right\Vert _{cb}\left\Vert \varphi_{n}^{-1}\right\Vert _{cb}\stackrel{n\rightarrow\infty}{\longrightarrow}1\cdot
\]
\item There exists a sequence of multiplier biholomorphisms $F_{n}:V\longrightarrow W_{n}$
such that 
\[
\mathit{\parallel\mathit{F_{n}}\parallel_{Mult\left(V\right)},\parallel F_{n}^{-1}\parallel_{Mult\left(W_{n}\right)}\overset{n\rightarrow\infty}{\longrightarrow}1\cdot}
\]
\end{enumerate}

\subsection{Background on the Isomorphism Problem \label{subsec:Background-on-isomorphism problem}}
\begin{defn}
Let $V\subseteq B_{d}$ be a\textbf{ }subvariety of $B_{d}$.
\end{defn}

\begin{itemize}
\item We say that $V$ is an\textbf{ irreducible variety }if it does not
decompose into two non-trivial subvarieties of $V$.
\item We say that $V$ is a \textbf{discrete variety} if it has no accumulation
points in $B_{d}$.
\item If $V$ is the common vanishing locus of homogeneous polynomials in
$B_{d}$, then we say that $V$ is a \textbf{homogeneous variety}. 
\end{itemize}
\begin{rem}
It is well-known that if $V$ is a homogeneous variety, then it decomposes
into a finite union of irreducible components.
\end{rem}

\begin{thm}
\label{Theorem 1.27}(\cite{davidson2015operator}). Let $d\in\mathbb{N}$,
and let $V$ and $W$ be two varieties in $B_{d}$, each is a finite
union of irreducible varieties and a discrete variety, then an algebra
isomorphism $\varphi:\mathcal{M_{\mathit{V}}}\rightarrow\mathcal{M_{\mathit{W}}}$
between their multiplier algebras determines a\textbf{ }multiplier
biholomorphism of $W$ onto \textbf{$V$, }$G_{\varphi}:W\rightarrow V$,
and the isomorphism $\varphi$ arises as a composition with $G_{\varphi}$:
\[
\varphi\left(f\right)=f\circ G_{\varphi},\quad f\in\mathcal{M}_{V}\cdot
\]
\end{thm}

Examples by Davidson, Ramsey and Shalit, show that the converse direction
of Theorem \ref{Theorem 1.27} does not hold in the general setting
of two varieties. In fact, they showed that, in general, the existence
of a multplier biholomorphism between two varieties, namely $V$ and
$W$, does not necessarily imply the existence of an algebra isomorphism
between $\mathcal{M_{\mathit{V}}}$ and $\mathcal{M_{\mathit{W}}}$
(See \cite[Example 6.2]{davidson2015operator}). However, two significant
examples for which the converse of Theorem \ref{Theorem 1.27} does
hold are finite subsets of $B_{d}$ and homogeneous varieties.

Further instances in which the converse direction of Theorem \ref{Theorem 1.27}
holds have been studied. In \cite{alpay2003hilbert}, Alpay, Putinar
and Vinnikov treated the case of embedded discs. This was extended
to finite planar domains \cite{arcozzi2008carleson} by Arcozzi, Rochberg
and Sawyer, and to finite Riemann surfaces \cite{kerr2013isomorphism}
by Kerr, McCarthy and Shalit. In \cite{ROCHBERG20191622}, Rochberg
treated finite sets of points. See also \cite{davidson2015multipliers}
by Davidson, Hartz and Shalit, and \cite{mironov2024isomorphism}
by Mironov.

In \cite[Theorem 5.9]{davidson2015operator}, it was shown that for
$d<\infty$ and two varieties $V$ and $W$ in $B_{d}$, every isometric
isomorphism $\varphi:\mathcal{M_{\mathit{V}}}\rightarrow\mathcal{M_{\mathit{W}}}$
is a completely isometric isomorphism. Thus, in this case, $\varphi$
arises as a composition with a restriction of an automorphism $B_{d}$.
However, unlike the case of finite subsets of $B_{d}$, while every
isomorphism $\varphi:\mathcal{M_{\mathit{V}}}\rightarrow\mathcal{M_{\mathit{W}}}$
is bounded, it is not neccesary a completely bounded isomorphism (See
\cite[Example 6.1]{davidson2015operator}). Therefore, the following
problem has arisen:
\begin{problem}
\label{Problem 1.28}Let $d\in\mathbb{N}$. Let $V$ and $W$ be two
varieties in $B_{d}$. Assume that $\varphi:\mathcal{M_{\mathit{V}}}\rightarrow\mathcal{M_{\mathit{W}}}$
is an algebra isomorphism arises as a composition with a multiplier
biholomorphism of $W$ onto $V$,
\[
\varphi\left(f\right)=f\circ G,\quad f\in\mathcal{M_{\mathit{V}}}\cdot
\]
What are sufficient conditions on $\varphi$ for being a multiplier
algebra isomorphism?
\end{problem}

From the proof of \cite[Proposition 4.4]{ofek2021distance}, we obtain
the following proposition:
\begin{prop}
\label{Proposition 1.29}(See \cite[Proposition 4.4]{ofek2021distance}).
Let $d\in\mathbb{N}$. Let $V$ and $W$ be two varieties in $B_{d}$.
Suppose that there exists a RKHS isomorphism $T:H_{V}\longrightarrow H_{W}$
given by
\[
T\left(k_{x}\right)=\lambda\left(x\right)k_{F(x)},\quad x\in V.
\]
Then, $T$ induces a multiplier algebra isomorphism $\varphi:\mathcal{M_{\mathit{V}}}\rightarrow\mathcal{M_{\mathit{W}}}$,
which is given by 
\[
\varphi\left(f\right)=f\circ F^{-1},\quad f\in\mathcal{M}_{V}\cdot
\]
Moreover, For every $f\in\mathcal{M_{\mathit{V}}}$, $M_{\varphi\left(f\right)}=\left(T^{*}\right)^{-1}M_{f}T^{*}.$
In this case we have 
\[
\left\Vert \varphi\right\Vert _{cb}\leq\left\Vert T\right\Vert \left\Vert T^{-1}\right\Vert \cdot
\]
\end{prop}

\begin{rem}
From Proposition \ref{Proposition 1.29} and \cite[Example 6.2]{davidson2015operator},
we deduce that, in the general setting of two varieties, namely $V$
and $W$ , the existence of a multiplier biholomorphism $F:V\to W$
between the varieties does not imply the existence of a RKHS isomorphism
between the corresponding RKHSs. In particular, $F$ does not necessary
induce a RKHS isomorphism $T:H_{V}\longrightarrow H_{W}$. However,
one can verify that if $X$ and $Y$ are each a set of $n$ points
in $B_{d}$, then a multiplier biholomorphism $F:X\to Y$ does induce
a RKHS isomorphism $T_{F}:H_{X}\longrightarrow H_{Y}$: Assume $X:=\left\{ x_{1},\ldots,x_{n}\right\} $
for some $n\in\mathbb{N}$, and let $T_{F}:H_{X}\longrightarrow H_{Y}$
be the RKHS isomorphism determined by 
\[
T_{F}\left(k_{x_{i}}\right)=k_{F\left(x_{i}\right)},\quad i\in\left\{ 1,\ldots,n\right\} \cdot
\]
Also, it is worth to remark that if, for example, $x_{1}=0$ and $F:X\to Y$
is a restriction of an automorphism of the ball with $F\left(0\right)\neq0$.
Then, $T_{F}\left(k_{x_{i}}\right)=k_{F\left(x_{i}\right)}$ is not
an isometric isomorphism. Meaning that, if $F:V\to W$ is a restriction
of an automorphism of the ball that maps a variety $V$ onto a variety
$W$, and $T:H_{V}\to H_{W}$ is a RKHSs that is induced by $F$,
then in general, $T$ is not necessary an isometric isomorphism (See
Proposition \ref{Theorem 1.31} and Theorem \ref{Theorem 1300} below). 
\end{rem}

From the definition of the kernel of $\mathcal{H}_{d}^{2}$ and \cite[Theorem 2.2.5]{rudin2008function},
we have the following theorem:
\begin{thm}
\label{Theorem 1.31}Let $F\in{\bf Aut\left(\mathit{B_{d}}\right)}$
and $a=F^{-1}\left(0\right)$. Then, then there is a unique unitary
$U:\mathbb{C^{\mathit{n}}}\longrightarrow\mathbb{C^{\mathit{n}}}$
such that $F=U\Phi_{a}$. The identity 
\[
\langle k_{x},k_{y}\rangle=\overline{\delta_{a}\left(x\right)}\delta_{a}\left(y\right)\langle k_{F\left(x\right)},k_{F\left(y\right)}\rangle,
\]
holds for every $x,y\in B_{d}$, where for every $x\in B_{d}$, $\delta_{a}\left(x\right):=\frac{k_{a}\left(x\right)}{\left\Vert k_{a}\right\Vert }$.
\end{thm}

Theorem \ref{Theorem 1.31} leads to the following:
\begin{thm}
(See \cite[Theorem 9.2]{davidson2011isomorphism} and \cite[Proposition 4.1]{davidson2015operator}).
\label{Theorem 1300}Let $F$ be an automorphism of $B_{d}$. Set
$a=F^{-1}\left(0\right)$. For every two varieties $V$ and $W$such
that $F\left(V\right)=W$, the map 
\begin{align*}
\left\{ k_{x}:x\in V\right\}  & \longrightarrow\left\{ k_{y}:y\in W\right\} \\
k_{x} & \longrightarrow\overline{\delta_{a}\left(x\right)}k_{F\left(x\right)}
\end{align*}
 extends to an isometric RKHS isomorphism $T_{F}:H_{V}\longrightarrow H_{W}$,
where for every $x\in V$, $\delta_{a}\left(x\right):=\frac{k_{a}\left(x\right)}{\left\Vert k_{a}\right\Vert }$.
\end{thm}

From Proposition \ref{Proposition 1.29} we obtain the following proposition
as an answer for Problem \ref{Problem 1.28}:
\begin{prop}
\label{Proposition 1.33}Let $d\in\mathbb{N}$. Let $V$ and $W$
be two varieties in $B_{d}$. Assume that $\varphi:\mathcal{M_{\mathit{V}}}\rightarrow\mathcal{M_{\mathit{W}}}$
is an \textbf{algebra isomorphism} that arises as a composition with
a multiplier biholomorphism $G_{\varphi}:W\to V$. That is,
\[
\varphi\left(f\right)=f\circ G_{\varphi},\quad f\in\mathcal{M_{\mathit{V}}}.
\]
Assume that there is a RKHS isomorphism $T:H_{V}\longrightarrow H_{W}$
given by
\[
T\left(k_{x}\right)=\lambda\left(x\right)k_{G_{\varphi}^{-1}},\quad x\in V\cdot
\]
 Then, $\varphi$ is a multiplier algebra isomorphism.
\end{prop}

Let $V$ and $W$ be two homogeneous varieties in $B_{d}$. Assume
that $V=\cup_{i=1}^{k}V_{i}$, for some $k\in\mathbb{N}$, is the
decomposition of $V$ into its irreducible components. Let\textbf{
$F:V\longrightarrow W$} be a biholomorphism. By \cite[Theorem 7.4]{davidson2011isomorphism}
and that $0\in\cap_{i=1}^{k}F\left(V_{i}\right)=F\left(\cap_{i=1}^{k}V_{i}\right)$,
it follows that the following are equivalent:
\begin{enumerate}
\item $\cap_{i=1}^{k}V_{i}=\left\{ 0\right\} \cdot$
\item $F$ is a restriction of an invertible linear map on $\mathbb{C}^{d}$
to $V$.
\end{enumerate}
By the following theorem, there always exists an invertible linear
map $A:\mathbb{C}^{d}\longrightarrow\mathbb{C}^{d}$ such that $A\left(V\right)=W$.
This implies that $\mathcal{M_{\mathit{V}}}$ and $\mathcal{M_{\mathit{W}}}$
are always isomorphic via a multiplier algebra isomorphism (See Theorem
\ref{Theorem 1.35} and Proposition \ref{Proposition 1.29}). Moreover,
from the fact that $\cup_{i=1}^{k}A\left(V_{i}\right)$ and $\cup_{i=1}^{k}F\left(V_{i}\right)$
are decompositions of $W$ into its irreducible components and the
fact that such a decomposition is unique, it follows that for every
$i\in\left\{ 1,\ldots,k\right\} $, there is $j\in\left\{ 1,\ldots,k\right\} $
such that $A\left(V_{i}\right)=F\left(V_{j}\right)\cdot$ This implies
that

\[
F\left(\cap_{i=1}^{k}V_{i}\right)=A\left(\cap_{i=1}^{k}V_{i}\right),
\]

\begin{thm}
\label{Theorem 1.34}(See \cite[Proposition 4.7 and Theorem 7.4 ]{davidson2011isomorphism}
and \cite[Lemma 9.6]{hartz2017isomorphism}). Let $V$ and $W$ be
homogeneous varieties. Suppose that $F:V\longrightarrow W$ is a bihlomorphism.
Then there exists a biholomorphism $F_{0}:V\rightarrow W$ s.t. $F_{0}\left(0\right)=0$.
In this case, there exists an invertible linear map $A:\mathbb{C}^{d}\longrightarrow\mathbb{C}^{d}$
such that $F_{0}=A\mid_{V}$.
\end{thm}

Let $A:\mathbb{C}^{d}\longrightarrow\mathbb{C}^{d}$ be such an invertible
linear map, with $A\left(V\right)=W$. Due to the work \cite{hartz2012topological}
by Hartz, the map 
\begin{align*}
\left\{ k_{x}:x\in V\right\}  & \longrightarrow\left\{ k_{y}:y\in W\right\} \\
k_{x} & \longrightarrow k_{Ax}
\end{align*}
 extends to a RKHS isomorphism $T_{A}:H_{V}\longrightarrow H_{W}$:
\begin{thm}
\label{Theorem 1.35}(\cite{hartz2012topological}). Let $V$ and
$W$ be homogeneous varieties in $B_{d}$. Assume that there exists
an invertible linear map $A:C^{d}\longrightarrow C^{d}$ that maps
$V$onto $W$. Then, $A$ induces a RKHS isomorphism $T_{A}:H_{V}\longrightarrow H_{W}$
such that for every $x\in V$, 
\[
T_{A}\left(k_{x}\right)=k_{Ax}\cdot
\]
\end{thm}

The existence of such a RKHS isomorphism implies that for every $f\in\mathcal{M_{\mathit{\mathit{V}}}}$,
$f\circ A^{-1}\in\mathcal{M_{\mathit{W}}}$, and that $A$ (or $T_{A}$)
induces a completely bounded algebra isomorphism (with a completely
bounded inverse) 
\[
\varphi_{A}:\mathcal{M_{\mathit{V}}}\rightarrow\mathcal{M_{\mathit{W}}},
\]
given by
\[
\varphi_{A}\left(f\right)=f\circ A^{-1},\;f\in\mathcal{M}_{V}\cdot
\]
That is, $A$ induces a multiplier algebra isomorphism (See Proposition
\ref{Proposition 1.29} above). From Theorems \ref{Theorem 1.34}
and \ref{Theorem 1.35}, it follows that if $V$ and $W$ are two
biholomorphic homogeneous varieties then $H_{V}$ and $H_{W}$ are
isomorphic as RKHSs. 

\subsection{Main Results\label{subsec:Main-results}}

\subsubsection{Main Results for General Varieties}

From Theorem \ref{Theorem 7.3}, we obtain the following result:
\begin{thm}
\label{Theorem 1.36}Let $V\subseteq B_{d}$ be a variety that contains
at least two points, and let $\left(W_{n}\right)_{n\in\mathbb{N}}$
be a sequence of varieties in $B_{d}$. Assume that for every $n\in\mathbb{N}$,
there exists a multiplier algebra isomorphism $\varphi_{n}:\mathcal{M_{\mathit{V}}}\rightarrow\mathcal{M_{\mathit{W_{n}}}}$,
which is defined by
\[
\varphi_{n}\left(f\right)=f\circ G_{\varphi_{n}},\quad f\in\mathcal{M_{\mathit{V}}}\cdot
\]
Assume that $\left\Vert \varphi_{n}^{\left(d\right)}\right\Vert ,\left\Vert \left(\varphi_{n}^{-1}\right)^{\left(d\right)}\right\Vert \stackrel{n\rightarrow\infty}{\longrightarrow}1$.
Then, 
\[
\mathit{\parallel\mathit{G_{\varphi_{n}}}\parallel_{Mult\left(W_{n}\right)},\parallel G_{\varphi_{n}}^{-1}\parallel_{Mult\left(V\right)}\overset{n\rightarrow\infty}{\longrightarrow}1}\cdot
\]
\end{thm}

From Theorem \ref{Theorem 1.36} and Corollary \ref{Corollary 7.4}
, we obtain the following result:
\begin{thm}
\label{Theorem 1.37}Let $V\subseteq B_{d}$ be a variety that contains
at least two points, and let $\left(W_{n}\right)_{n\in\mathbb{N}}$
be a sequence of varieties in $B_{d}$. Assume that for every $n\in\mathbb{N}$,
there exists a RKHS isomrphism $T_{n}:H_{V}\longrightarrow W_{n}$
such that
\[
T_{n}\left(k_{x}\right)=\delta_{n}\left(x\right)k_{F_{n}\left(x\right)},\quad x\in V,
\]
where $F_{n}:V\to W_{n}$ is a multiplier biholomorphism and $\delta_{n}:V\longrightarrow\mathbb{C}$
is a no-where vanishing function. For every $n\in\mathbb{N}$, let
$\varphi_{n}:\mathcal{M_{\mathit{V}}}\rightarrow\mathcal{M_{\mathit{W_{n}}}}$
be the multiplier algebra isomorphism, which is induced by $T_{n}$,
and defined by
\[
\varphi_{n}\left(f\right)=f\circ F_{n}^{-1},\quad f\in\mathcal{M_{\mathit{V}}}\cdot
\]
Assume that $\left\Vert T_{n}\right\Vert \left\Vert T_{n}^{-1}\right\Vert \stackrel{n\rightarrow\infty}{\longrightarrow}1$.
Then, 
\end{thm}

\begin{enumerate}
\item $\left\Vert \varphi_{n}\right\Vert _{cb}\left\Vert \varphi_{n}^{-1}\right\Vert _{cb}\stackrel{n\rightarrow\infty}{\longrightarrow}1$. 
\item $\mathit{\parallel F_{n}\parallel_{Mult\left(V\right)},\parallel F_{n}^{-1}\parallel_{Mult\left(W_{n}\right)}\overset{n\rightarrow\infty}{\longrightarrow}1}\cdot$
\end{enumerate}

\subsubsection{Main Results for Homogeneous Varieties}
\begin{notation}
Let $A:\mathbb{C}^{d}\to\mathbb{C}^{d}$ be a linear map. Then, $\left\Vert A\right\Vert _{\mathrm{OP}}$
denotes the operator norm of $A$.
\end{notation}

From Theorems \ref{Theorem 5.14} , \ref{Theorem 1.37} , \ref{Theorem 1.36}
and \ref{Theorem 6.9}, we obtain the following result, which says,
roughly, that if $V\subseteq B_{d}$ is a tractable homogeneous variety
(See Definitions \ref{Definition 1.4} and \ref{Definition 1.5}),
then small linear deformations of $V$ result in small deformations
of the corresponding RKHS and multiplier algebra, and vice versa.
\begin{thm}
\label{Theorem 1.39} Let $d\in\mathbb{N}$ . Let $V$ be a tractable
homogeneous variety in $B_{d}$ s.t. $\textrm{span}\left(V\right)=\mathbb{C}^{d}$.
Let $\left(W_{n}\right)_{n\in\mathbb{N}}$ be a sequence of homogeneous
varieties. Suppose that for every $n\in\mathbb{N}$, $A_{n}:\mathbb{C}^{d}\longrightarrow\mathbb{C}^{d}$
is an invertible linear transformation such that $W_{n}:=A_{n}V$.
For every $n\in\mathbb{N}$, let 
\[
T_{A_{n}}:H_{V}\rightarrow H_{W_{n}}
\]
 be the RKHS isomorphism determined by 
\[
T_{A_{n}}\left(k_{x}\right)=k_{A_{n}x},\;x\in V,
\]
and let $\varphi_{A_{n}}:\mathcal{M}_{W}\longrightarrow\mathcal{M}_{V}$
be the induced multiplier algebra isomorphism given by
\[
\varphi_{A_{n}}\left(f\right)=f\circ A_{n}^{-1},\quad f\in\mathcal{M}_{V}\cdot
\]
 Then, the following are equivalent.
\end{thm}

\begin{enumerate}
\item $\mathit{\left\Vert \mathit{A_{n}}\right\Vert _{\mathrm{OP}},\left\Vert A_{n}^{-1}\right\Vert _{\mathrm{OP}}\overset{n\rightarrow\infty}{\longrightarrow}1}$.
\item $\left\Vert T_{A_{n}}\right\Vert \left\Vert T_{A_{n}}^{-1}\right\Vert \stackrel{n\rightarrow\infty}{\longrightarrow}1\cdot$
\item $\left\Vert \varphi_{A_{n}}\right\Vert _{cb}\left\Vert \varphi_{A_{n}}^{-1}\right\Vert _{cb}\stackrel{n\rightarrow\infty}{\longrightarrow}1\cdot$
\item $\mathit{\left\Vert \mathit{A_{n}}\right\Vert {}_{Mult\left(V\right)},\left\Vert A_{n}^{-1}\right\Vert {}_{Mult\left(W_{n}\right)}\overset{n\rightarrow\infty}{\longrightarrow}1}$.
\end{enumerate}
From the Theorems \ref{Theorem 1.37}, \ref{Theorem 1.36}, \ref{Theorem 7.7},
\ref{Theorem 6.9} and \ref{Theorem 5.14}, we obtain the following
result:
\begin{thm}
Let $d\in\mathbb{N}$ . Let $V$ be a tractable homogeneous variety
in $B_{d}$, with $\textrm{span}\left(V\right)=\mathbb{C}^{d}$. Let
$\left(W_{n}\right)_{n\in\mathbb{N}}$ be a sequence of homogeneous
varieties. Then, the following are equivalent.
\end{thm}

\begin{enumerate}
\item There is a sequence of RKHS isomorphisms $T_{n}:H_{V}\longrightarrow H_{W_{n}}$
such that 
\[
\left\Vert T_{n}\right\Vert \left\Vert T_{n}^{-1}\right\Vert \stackrel{n\rightarrow\infty}{\longrightarrow}1\cdot
\]
\item There is a sequence of multiplier algebra isomorphisms $\varphi_{n}:\mathcal{M_{\mathit{V}}}\rightarrow\mathcal{M_{\mathit{W_{n}}}}$
such that 
\[
\left\Vert \varphi_{n}\right\Vert _{cb}\left\Vert \varphi_{n}^{-1}\right\Vert _{cb}\stackrel{n\rightarrow\infty}{\longrightarrow}1\cdot
\]
\item There is a sequence of algebra isomorphisms $\varphi_{n}:\mathcal{M_{\mathit{V}}}\rightarrow\mathcal{M_{\mathit{W_{n}}}}$
such that 
\[
\left\Vert \varphi_{n}^{\left(d\right)}\right\Vert \left\Vert \left(\varphi_{n}^{-1}\right)^{\left(d\right)}\right\Vert \stackrel{n\rightarrow\infty}{\longrightarrow}1\cdot
\]
\item There is a sequence of invertible linear transformations $A_{n}:\mathbb{C}^{d}\longrightarrow\mathbb{C}^{d}$
such that for every $n\in\mathbb{N}$, $W_{n}=A_{n}V$, and 
\[
\mathit{\left\Vert \mathit{A_{n}}\right\Vert {}_{Mult\left(V\right)},\left\Vert A_{n}^{-1}\right\Vert {}_{Mult\left(W_{n}\right)}\overset{n\rightarrow\infty}{\longrightarrow}1}\cdot
\]
\end{enumerate}

\section{RKHSs Obtained by Restricting $\mathcal{H}_{d}^{2}$ to Subsets of
the Open Unit Ball\label{sec:On-the-RKHS}}
\begin{notation}
\label{Notation 21}For any non-empty subsets $X\subseteq B_{d}$,
$\mathit{F}\subseteq\mathcal{H}_{d}^{2}$ and $S\subseteq\mathcal{M}_{d}$,
we will use the following notations:
\end{notation}

\begin{itemize}
\item $H_{X}:=\overline{\textrm{span}\left\{ k_{x}:x\in X\right\} }$
\item $Z\left(F\right):=\left\{ x\in B_{d}:f\left(x\right)=0,\text{for all }f\in F\right\} $. 
\item $I_{X}:=\left\{ f\in\mathcal{H}_{d}^{2}:f\left(x\right)=0,\:\text{for all }x\in X\right\} $.
\item $J_{X}:=\left\{ f\in\mathcal{M}_{d}:f\left(x\right)=0,\:\text{for all }x\in X\right\} $.
\item $\left\langle S\right\rangle =\left\{ \sum_{i=1}^{n}f_{i}g_{i}:f_{i}\in\mathcal{M}_{d},g_{i}\in S,n\in\mathbb{N}\right\} $;
i.e., $\left\langle S\right\rangle $ will denote the ideal in $\mathcal{M}_{d}$,
which is generated by the set $S$.
\item If $S$ is a vector subspace of $\mathcal{H}_{d}^{2}$ , then $\left[S\right]$
will be used to denote the closure of $S$ in $\mathcal{H}_{d}^{2}$. 
\end{itemize}
\begin{prop}
\label{Proposition 22}Let $F\subseteq\mathcal{H}_{d}^{2}$ be a subset.
Then,
\end{prop}

\begin{enumerate}
\item If $F_{1}\subseteq\mathcal{H}_{d}^{2}$ is a subset s.t. $F\subseteq F_{1}\subseteq I_{Z\left(F\right)}$,
then $Z\left(I_{Z\left(F\right)}\right)=Z\left(F\right)=Z\left(F_{1}\right)\cdot$
\item ( \cite[Proposition 2.1]{davidson2015operator}). $Z\left(F\right)=Z\left(J_{Z\left(F\right)}\right)=Z\left(I_{Z\left(F\right)}\right)$. 
\end{enumerate}
\begin{proof}
Let $F\subseteq\mathcal{H}_{d}^{2}$ be a subset. Then,
\begin{enumerate}
\item Since $F\subseteq F_{1}\subseteq I_{Z\left(F\right)}$, it follows
that 
\[
Z\left(I_{Z\left(F\right)}\right)\subseteq Z\left(F_{1}\right)\subseteq Z\left(F\right)\cdot
\]
 Now, note that $Z\left(F\right)\subseteq Z\left(I_{Z\left(F\right)}\right)$
to obtain 
\[
Z\left(I_{Z\left(F\right)}\right)=Z\left(F_{1}\right)=Z\left(F\right)\cdot
\]
\item See \cite{davidson2015operator}and $\left(1\right)$.
\end{enumerate}
\end{proof}
\begin{prop}
\label{Proposition 23}Let $X\subseteq B_{d}$ be a subset. Then,
\end{prop}

\begin{enumerate}
\item \label{Theorem 2.11}$I_{X}=\left(H_{X}\right)^{\perp}$. Hence, $I_{X}$
is a closed subspace of $\mathcal{H}_{d}^{2}$ and $\mathcal{H}_{d}^{2}=H_{X}\oplus I_{X}$.
\item (\cite[Proposition 2.2]{davidson2015operator}). $\mathit{H_{X}}=H_{Z\left(I_{X}\right)}=H_{Z\left(J_{X}\right)}$.
\item $I_{X}=I_{Z\left(I_{X}\right)}=I_{Z\left(J_{X}\right)}$.
\item $Z\left(I_{X}\right)=Z\left(J_{X}\right)$.
\item $J_{X}=I_{X}\cap\mathcal{M}_{d}=\left[J_{X}\right]\cap\mathcal{M}_{d}$.
\end{enumerate}
\begin{proof}
\begin{enumerate}
\item Let $f\in\mathcal{H}_{d}^{2}$. Then,
\[
f\in I_{X}\Leftrightarrow\forall x\in X,\;f\left(x\right)=0\Leftrightarrow\forall x\in X,\:\left\langle f,k_{x}\right\rangle =0\Leftrightarrow f\in\left(H_{X}\right)^{\perp}.
\]
\item See \cite{davidson2015operator}.
\item The assertion follows from $\left(2\right)+\left(1\right)$.
\item By $\left(3\right)$ and Proposition \ref{Proposition 22} , we have
$I_{X}=I_{Z\left(J_{X}\right)}$ and $Z\left(I_{Z\left(J_{X}\right)}\right)=Z\left(J_{X}\right)$.
It follows that 
\[
Z\left(I_{X}\right)=Z\left(I_{Z\left(J_{X}\right)}\right)=Z\left(J_{X}\right).
\]
\item Note that $J_{X}\subseteq\left[J_{X}\right]\subseteq I_{X}$ and $I_{X}\cap\mathcal{M}_{d}=J_{X}$.
Therefore, $J_{X}\subseteq\left[J_{X}\right]\cap\mathcal{M}_{d}\subseteq J_{X}$.
Hence, $J_{X}=\left[J_{X}\right]\cap\mathcal{M}_{d}$.
\end{enumerate}
\end{proof}
\begin{prop}
\label{Proposition 24}Let $X\subseteq B_{d}$ be a subset. Then,
for every $y\in B_{d}$, $\mathit{y}\in Z\left(I_{X}\right)$ if and
only if $k_{\mathit{y}}\in H_{X}$ if and only if $k_{y}\in\left(J_{X}\right)^{\perp}$.
\end{prop}

\begin{proof}
Let $y\in B_{d}$. Then,
\[
y\in Z\left(I_{X}\right)\Leftrightarrow\forall f\in I_{X},\;f\left(y\right)=0\Leftrightarrow\forall f\in I_{X},\:\left\langle f,k_{y}\right\rangle =0\Leftrightarrow k_{y}\in\left(I_{X}\right)^{\perp}=H_{X}
\]
 and 
\[
y\in Z\left(J_{X}\right)\Longleftrightarrow\forall f\in J_{X},\;f\left(y\right)=0\Leftrightarrow\forall f\in J_{X},\:\left\langle f,k_{y}\right\rangle =0\Leftrightarrow k_{y}\in\left(J_{X}\right)^{\perp}.
\]
Since, by Proposition \ref{Proposition 23}, $Z\left(I_{X}\right)=Z\left(J_{X}\right)$,
the assertion follows.
\end{proof}
We have seen, above, that for any subset $X\subseteq B_{d}$, $H_{X}=H_{Z\left(I_{X}\right)}$.
Now, we will show that $Z\left(I_{X}\right)$ is the unique variety
with this property:
\begin{thm}
\label{Theorem 25} Let $X\subseteq B_{d}$ be a subset, and let $V\subseteq B_{d}$
be a variety. Then, $H_{X}=H_{V}$ if and only if $Z\left(I_{X}\right)=V$.
\end{thm}

\begin{proof}
By Proposition \ref{Proposition 23}$,$ $H_{X}=H_{V}$ if and only
$I_{X}=I_{V}$ if and only if $Z\left(I_{X}\right)=Z\left(I_{V}\right)$
. Since, by Proposition \ref{Proposition 22}, $Z\left(I_{V}\right)=V$,
the assertion follows.
\end{proof}
\begin{thm}
Let $V\subseteq B_{d}$ be a variety, and let $y\in B_{d}$. Then,
$y\in V$ if and only $k_{y}\in H_{V}$ if and only if $\mathit{k_{y}\in}\left(J_{V}\right)^{\perp}$.
\end{thm}

\begin{proof}
By Propositions \ref{Proposition 22}, we have $Z\left(I_{V}\right)=V$.
Now, apply \ref{Proposition 24}.
\end{proof}
\begin{notation}
For any non empty subsets $X\subseteq B_{d}$, set $\mathcal{H}_{d}^{2}\mid_{X}:=\left\{ f\mid_{X}:f\in\mathcal{H}_{d}^{2}\right\} $.
\end{notation}

\begin{prop}
\label{proposition 28}Let $X\subseteq B_{d}$ be a non empty set.
Then, 
\end{prop}

\begin{enumerate}
\item For every $f\in\mathcal{H}_{d}^{2}$, $P_{H_{X}}\left(f\right)\mid_{X}=f\mid_{X}\cdot$
\item $\left(H_{X}\right)\mid_{X}=\mathcal{H}_{d}^{2}\mid_{X}$.
\item The map $T:H_{X}\longrightarrow\mathcal{H}_{d}^{2}\mid_{X}$, which
is given by $T\left(f\right)=f\mid_{X}$ is a bijective linear map.
\end{enumerate}
\begin{proof}
Let $X\subseteq B_{d}$ be a non empty set. 
\end{proof}
\begin{enumerate}
\item Let $f\in\mathcal{H}_{d}^{2}$. then, $P_{H_{X}}\left(f\right)\mid_{X}=P_{H_{X}}\left(f\right)\mid_{X}+P_{I_{X}}\left(f\right)\mid_{X}=\left(P_{H_{X}}\left(f\right)+P_{I_{X}}\left(f\right)\right)\mid_{X}=f\mid_{X}\cdot$
\item Since $H_{X}\subseteq\mathcal{H}_{d}^{2}$ it follows that $\left(H_{X}\right)\mid_{X}\subseteq\mathcal{H}_{d}^{2}\mid_{X}$.
Now, if $f\in\mathcal{H}_{d}^{2}$ then $P_{H_{X}}\left(f\right)\in H_{X}$.
By $\left(1\right)$ we have $P_{H_{X}}\left(f\right)\mid_{X}=f\mid_{X}\cdot$
Therefore , $\left(H_{X}\right)\mid_{X}\supseteq\mathcal{H}_{d}^{2}\mid_{X}$.
Hence, $\left(H_{X}\right)\mid_{X}=\mathcal{H}_{d}^{2}\mid_{X}$.
\item It is clear that $T$ is linear and $\ker T=H_{X}\cap I_{X}=\left\{ 0\right\} $.
Therefore, $T$ is an injective linear map. Now, if $f\in\mathcal{H}_{d}^{2}$
then $P_{H_{X}}\left(f\right)\in H_{X}$, and by $\left(1\right)$
we have 
\[
T\left(P_{H_{X}}\left(f\right)\right)=P_{H_{X}}\left(f\right)\mid_{X}=f\mid_{X}\cdot
\]
 Therefore, $T$ is surjective.
\end{enumerate}
\begin{cor}
\label{Corollary 29}Let $X\subseteq B_{d}$ be a non empty set. For
every $f,g\in\mathcal{H}_{d}^{2}$, define 
\[
\left\langle f\mid_{X},g\mid_{X}\right\rangle _{\mathcal{H}_{d}^{2}\mid_{X}}:=\left\langle P_{H_{X}}\left(f\right),P_{H_{X}}\left(g\right)\right\rangle _{\mathcal{H}_{d}^{2}}\cdot
\]
Then, $\mathcal{H}_{d}^{2}\mid_{X}$ is isometrically isomorphic to
$H_{X}$. Moreover, $\mathcal{H}_{d}^{2}\mid_{X}$ is a RKHS on $X$
that determined by the reproducing kernel 
\[
k:X\times X\rightarrow\mathbb{C}
\]
 given by 
\[
k\left(x,y\right)=\frac{1}{1-\langle x,y\rangle},\quad\textrm{for all }x,y\in X\cdot
\]
 In fact, $\mathcal{H}_{d}^{2}\mid_{X}:=\overline{\text{span}\left\{ k_{x}\mid_{X}:x\in X\right\} }$.
\end{cor}

\begin{proof}
By Proposition \ref{proposition 28}, it follows that $T$ is an isometric
isomorphism. Moreover, for every $x\in X$ and $f\in\mathcal{H}_{d}^{2}$,
we have
\[
f\mid_{X}=P_{H_{X}}\left(f\right)\mid_{X}\cdot
\]
 Therefore, $\rho_{x}\left(f\mid_{X}\right)=\rho_{x}\left(P_{H_{X}}\left(f\right)\right)$.
Note that since $H_{X}\subseteq\mathcal{H}_{d}^{2}$, $\rho_{x}$
is bounded on $\mathcal{H}_{d}^{2}$ and $\left\Vert f\mid_{X}\right\Vert _{H_{d}^{2}\mid\mid_{X}}=\left\Vert P_{H_{X}}\left(f\right)\right\Vert _{H_{X}}$,
it follows that it is bounded on $H_{X}$. Therefore, $\rho_{x}$
is bounded on $\mathcal{H}_{d}^{2}\mid_{X}$. Hence, $\mathcal{H}_{d}^{2}\mid_{X}$
is a RKHS on $X$. Finally, recall that for every $x\in X$, $k_{x}\in H_{X}$
and that for every $f\in\mathcal{H}_{d}^{2}$ we have
\[
f\mid_{X}\left(x\right)=P_{H_{X}}\left(f\right)\left(x\right)=\left\langle P_{H_{X}}\left(f\right),k_{x}\right\rangle _{\mathcal{H}_{d}^{2}}=\left\langle f\mid_{X},k_{x}\mid_{X}\right\rangle _{\mathcal{H}_{d}^{2}\mid_{X}}\cdot
\]
Therefore, $\mathcal{H}_{d}^{2}\mid_{X}$ is determined by the reproducing
kernel 
\[
k:X\times X\rightarrow\mathbb{C}
\]
 given by 
\[
k\left(x,y\right)=\frac{1}{1-\langle x,y\rangle},\quad\textrm{for all }x,y\in X\cdot
\]
 Moreover, $\mathcal{H}_{d}^{2}\mid_{X}:=\overline{\text{span}\left\{ k_{x}\mid_{X}:x\in X\right\} }$.
\end{proof}
\begin{claim}
Let $X\subseteq B_{d}$ be a subset. Let $J:\mathcal{H}_{d}^{2}\longrightarrow F_{s}\left(\mathbb{C}^{d}\right)$
be the anti-unitary operator defined in \ref{Equation 2.1.1}. Then,
\end{claim}

\[
J\left(H_{X}\right)=\overline{\text{span\ensuremath{\left\{  \oplus_{n=0}^{\infty}x^{n}:x\in X\right\} } }}\subseteq\oplus_{n=0}^{\infty}X^{n}=\oplus_{n=0}^{\infty}\text{\ensuremath{\mathbb{\left(C\mathit{X}\right)^{\mathit{n}}}}}\subseteq\oplus_{n=0}^{\infty}\left(\text{span\ensuremath{\mathit{X}}}\right)^{n}\cdot
\]
(Here $\mathbb{C}X:=\left\{ \alpha x:\alpha\in\mathbb{C},x\in X\right\} $). 
\begin{proof}
Recall that $H_{X}=\overline{\text{span\ensuremath{\left\{  k_{x}:x\in X\right\} } }}$and
note that $J\left(H_{X}\right)=\overline{\text{span\ensuremath{\left\{  J\left(k_{x}\right):x\in X\right\} } }}$,
because $J$ is an anti-unitary. The rest is clear.
\end{proof}

\section{RKHSs on Homogeneous Varieties\label{Sec:On-the-RKHS-Homogeneous}}

This section is mainly devoted to showing useful equivalent conditions
for a variety to be homogeneous. See Theorem \ref{Theorem 36} below.
\begin{claim}
\label{Claim 31}Let $f=\sum_{n=0}^{\infty}\left\langle \cdot,\xi_{n}\right\rangle \in\mathcal{H}_{d}^{2}$
and $n\in\mathbb{N}_{0}$. Then, For every $x\in B_{d}$, 
\[
f_{n}\left(x\right)=\left\langle f,\left\langle \cdot,x^{n}\right\rangle \right\rangle \cdot
\]
\end{claim}

\begin{proof}
Let $f=\sum_{k=0}^{\infty}\left\langle \cdot,\xi_{k}\right\rangle \in\mathcal{H}_{d}^{2}$
and $n\in\mathbb{N}_{0}$. Since $J:\mathcal{H}_{d}^{2}\longrightarrow\mathcal{F}_{s}\left(\mathbb{C}^{d}\right)$
is an anti-unitary, it follows that homogeneous polynomials of different
degree are orthogonal in $\mathcal{H}_{d}^{2}$. Therefore, for every
$x\in B_{d}$,
\[
f_{n}\left(x\right)=\left\langle f_{n},k_{x}\right\rangle =\left\langle f_{n},\left\langle \cdot,x^{n}\right\rangle \right\rangle =\left\langle f,\left\langle \cdot,x^{n}\right\rangle \right\rangle .
\]
\end{proof}
\begin{prop}
\label{Proposition 32}Let $V\subseteq B_{d}$ be a variety. Let $n\in\mathbb{N}$.
Then, 
\end{prop}

\begin{enumerate}
\item For every $x\in B_{d}$, $\left\langle \cdot,x^{n}\right\rangle \in H_{V}\Longleftrightarrow f_{n}\left(x\right)=0\:\text{for all }f\in I_{V}$.
\item $\text{span}\left\{ \left\langle \cdot,x^{n}\right\rangle :x\in V\right\} \subseteq H_{V}\Longleftrightarrow f_{n}\in J_{V}\:\text{for all }f\in I_{V}$.
\end{enumerate}
\begin{proof}
Let $n\in\mathbb{N}$.
\begin{enumerate}
\item Let $x\in B_{d}$. Then, 
\begin{eqnarray*}
\left\langle \cdot,x^{n}\right\rangle  & \in & H_{V}\Longleftrightarrow\left\langle \cdot,x^{n}\right\rangle \in\left(I_{V}\right)^{\perp}\\
 &  & \Longleftrightarrow f_{n}\left(x\right)=\left\langle f,\left\langle \cdot,x^{n}\right\rangle \right\rangle =0,\;\text{for all }f\in I_{V}\cdot
\end{eqnarray*}
\item 
\begin{eqnarray*}
\text{span}\left\{ \left\langle \cdot,x^{n}\right\rangle :x\in V\right\}  & \subseteq & H_{V}\Longleftrightarrow\left\langle \cdot,x^{n}\right\rangle \in H_{V},\:\text{for all }x\in V\\
 &  & \overset{\left(1\right)}{\Longleftrightarrow}f_{n}\left(x\right)=0,\:\text{for all }x\in V,f\in I_{V}\\
 &  & \Longleftrightarrow f_{n}\in J_{V},\:\text{for all }f\in I_{V}.
\end{eqnarray*}
\end{enumerate}
\end{proof}
\begin{notation}
$\mathbb{C}\left[Z\right]$ denotes the ring of complex polynomials
in $\mathbb{C}^{d}$.
\end{notation}

\begin{prop}
\label{Proposition 34}Let $V\subseteq B_{d}$ be a variety. Then, 
\begin{enumerate}
\item $f_{n}\in J_{V}$ for all $f\in J_{V}$ and $n\in\mathbb{N}$, if
and only if $f_{n}\in J_{V}$ for all $f\in\left[J_{V}\right]$ and
$n\in\mathbb{N}$.
\item Let $I\subseteq\mathcal{H}_{d}^{2}$ s.t. $f_{n}\in J_{V}$ for all
$f\in I$ and $n\in\mathbb{N}_{0}$, then $\overline{I}\subseteq\left[I_{V}\cap\mathbb{C}\left[Z\right]\right]\subseteq\left[J_{V}\right]\subseteq I_{V}$.
\end{enumerate}
\begin{proof}
Let $V\subseteq B_{d}$ be a variety.
\begin{enumerate}
\item $\left(\Rightarrow\right)$: Suppose that $g_{n}\in J_{V}$ for all
$g\in J_{V}$ and $n\in\mathbb{N}$. Let $f\in\left[J_{V}\right]$
and $n\in\mathbb{N}$. Then, there is $\left(f^{\left(m\right)}\right)_{m=1}^{\infty}\subseteq J_{V}$
s.t. that $f^{\left(m\right)}\stackrel{m\rightarrow\infty}{\longrightarrow}f$
in $\mathcal{H}_{d}^{2}$. Now, since for every $m\in\mathbb{N}$,
we have
\[
\left\Vert f_{n}^{\left(m\right)}-f_{n}\right\Vert ^{2}\leq\sum_{k=0}^{\infty}\left\Vert f_{k}^{\left(m\right)}-f_{k}\right\Vert ^{2}=\left\Vert f^{\left(m\right)}-f\right\Vert _{\mathcal{H}_{d}^{2}}^{2},
\]
 it follows that $\left\Vert f_{n}^{\left(m\right)}-f_{n}\right\Vert ^{2}\stackrel{m\rightarrow\infty}{\longrightarrow}0$.
Note that, for every $x\in V$,
\[
\left|f_{n}\left(x\right)\right|=\left|f_{n}^{\left(m\right)}\left(x\right)-f_{n}\left(x\right)\right|=\left|\left\langle f_{n}^{\left(m\right)}-f_{n},k_{x}\right\rangle \right|\leq\left\Vert f_{n}^{\left(m\right)}-f_{n}\right\Vert \left\Vert k_{x}\right\Vert .
\]
Therefore, $f_{n}\in J_{V}$. The converse; i.e., $\left(\Leftarrow\right)$,
follows from the fact that $J_{V}\subseteq\left[J_{V}\right]$ .
\item Note that $I\subseteq\left[I_{V}\cap\mathbb{C}\left[Z\right]\right]\subseteq\left[J_{V}\right]\subseteq I_{V}$.
In fact, for every $f\in I$ and $n\in\mathbb{N}$, $\sum_{k=0}^{n}f_{k}\in I_{V}\cap\mathbb{C}\left[Z\right]$.
Hence, $f\in\left[I_{V}\cap\mathbb{C}\left[Z\right]\right]$.
\end{enumerate}
\end{proof}
\end{prop}

\begin{thm}
\label{Theorem 35}Let $V\subseteq B_{d}$ be a variety such that
for every $f\in I_{V}$ and $n\in\mathbb{N}_{0}$, $f_{n}\in J_{V}$.
Then, 
\[
\left[I_{V}\cap\mathbb{C}\left[Z\right]\right]=\left[J_{V}\right]=I_{V}\cdot
\]
\end{thm}

\begin{proof}
The proof follows from Proposition \ref{Proposition 34}.
\end{proof}
\begin{thm}
\label{Theorem 36}Let $V\subseteq B_{d}$ be a variety. The following
are equivalent.
\end{thm}

\begin{enumerate}
\item $V$ is homogeneous.
\item $f_{n}\in J_{V}$ for all $f\in I_{V}$ and $n\in\mathbb{N}_{0}$.
\item $J\left(H_{V}\right)=\oplus_{n=0}^{\infty}V^{n}$.
\item $f_{n}\in H_{V}$ for all $f\in H_{V}$ and $n\in\mathbb{N}_{0}$.
\item $H_{V}=\overline{\textrm{span}\left\{ \left\langle \cdot,x^{n}\right\rangle :x\in V,n\in\mathbb{N}_{0}\right\} }$
. 
\end{enumerate}
\begin{proof}
Let $V\subseteq B_{d}$ be a variety. 
\begin{itemize}
\item $\left(\left(1\right)\Rightarrow\left(2\right)\right)$: Suppose that
$V$ is homogeneous, and let $f\in I_{V}$. We will show, by induction,
that $f_{n}\in J_{V}$ for all $n\in\mathbb{N}_{0}$: Since $f\in I_{V}$,
it follows that $f_{0}=0$. That is, $f_{0}\in J_{V}$. Now, suppose
that for some $k\in\mathbb{N}_{0}$, $f_{n}\in J_{V}$ for all $n=0,....,k$.
Then, for every $x\in V$ and $t\in\mathbb{C}$ s.t. $0<\left|t\right|\leq1$,
we obtain
\[
0=\frac{f\left(tx\right)}{t^{k+1}}=f_{k+1}\left(x\right)+t\sum_{n=k+2}^{\infty}t^{n-\left(k+2\right)}f_{n}\left(x\right)\stackrel{t\longrightarrow0}{\longrightarrow}f_{k+1}\left(x\right).
\]
Hence, $f_{k+1}\in J_{V}$.
\item $\left(\left(2\right)\Rightarrow\left(1\right)\right)$: suppose that
$f_{n}\in J_{V}$ for all $f\in I_{V}$ and $n\in\mathbb{N}_{0}$,
Then $I_{V}\cap\mathbb{C}\left[Z\right]$ is a homogeneous ideal in
$\mathbb{C}\left[Z\right]$. By Proposition \ref{Proposition 34},
we have $I_{V}=\left[I_{V}\cap\mathbb{C}\left[Z\right]\right]$ .
Therefore, $Z\left(I_{V}\cap\mathbb{C}\left[Z\right]\right)=Z\left(I_{V}\right)=V$.
Hence, $V$ is homogeneous.
\item $\left(\left(2\right)\Leftrightarrow\left(3\right)\right)$: Since
always we have 
\[
J\left(H_{V}\right)\subseteq\oplus_{n=0}^{\infty}\text{span}\left\{ x^{n}:x\in V\right\} ,
\]
 it follows that $J\left(H_{V}\right)=\oplus_{n=0}^{\infty}\text{span}\left\{ x^{n}:x\in V\right\} $
if and only if for every $n\in\mathbb{N}$, 
\[
\text{span}\left\{ \left\langle \cdot,x^{n}\right\rangle :x\in V\right\} \subseteq H_{V}
\]
 if and only if $f_{n}\in J_{V}$ for all $f\in I_{V}$ and $n\in\mathbb{N}$,
by Proposition \ref{Proposition 32}.
\item $\left(\left(2\right)\Rightarrow\left(4\right)\right)$: Suppose that
$f_{n}\in J_{V}$ for all $f\in I_{V}$ and $n\in\mathbb{N}_{0}$.
Let $g\in H_{V}$, $f\in I_{V}$ and $n\in\mathbb{N}_{0}$. Then,
\[
\left\langle g_{n},f\right\rangle =\left\langle g_{n},f_{n}\right\rangle =\left\langle g,f_{n}\right\rangle =0\cdot
\]
 Hence, $g_{n}\in\left(I_{V}\right)^{\perp}=H_{V}$.
\item $\left(\left(4\right)\Rightarrow\left(2\right)\right)$:Suppose that
$f_{n}\in H_{V}$ for all $f\in H_{V}$ and $n\in\mathbb{N}_{0}$.
Since for every $x\in V$, $k_{x}\in H_{V}$, it follows that for
every $x\in V$ and $n\in\mathbb{N}\cup\left\{ 0\right\} $, $\left\langle \cdot,x^{n}\right\rangle \in H_{V}$.
Now, let $f\in I_{V}$ and $n\in\mathbb{N}\cup\left\{ 0\right\} $,
then $f_{n}\left(x\right)=\left\langle f,\left\langle \cdot,x^{n}\right\rangle \right\rangle =0$.
Hence, $f_{n}\in J_{V}$.
\item $\left(\left(5\right)\Rightarrow\left(2\right)\right)$:Suppose that
$H_{V}=\overline{\textrm{span}\left\{ \left\langle \cdot,x^{n}\right\rangle :x\in V,n\in\mathbb{N}_{0}\right\} }$.
Then, in particular, for every $x\in V$ and $n\in\mathbb{N}_{0}$,
$\left\langle \cdot,x^{n}\right\rangle \in H_{V}$. Hence, for every
$f\in I_{V}$ and $n\in\mathbb{N}_{0}$, $f_{n}\in J_{V}$.
\item $\left(\left(4\right)\Rightarrow\left(5\right)\right)$:Suppose that
$f_{n}\in H_{V}$ for all $f\in H_{V}$ and $n\in\mathbb{N}_{0}$.
Then, in particular, for every $x\in V$ and $n\in\mathbb{N}_{0}$,
$\left\langle \cdot,x^{n}\right\rangle \in H_{V}$. It follows that,
\[
H_{V}=\overline{\textrm{span}\left\{ k_{x}:x\in V\right\} }\subseteq\overline{\textrm{span}\left\{ \left\langle \cdot,x^{n}\right\rangle :x\in V,n\in\mathbb{N}_{0}\right\} }\subseteq H_{V}\cdot
\]
\end{itemize}
\end{proof}
\begin{thm}
\label{Theorem 37}Let $V\subseteq B_{d}$ be a homogeneous variety.
Then,
\end{thm}

\begin{enumerate}
\item $\left[I_{V}\cap\mathbb{C}\left[Z\right]\right]=\left[J_{V}\right]=I_{V}$.
\item $\left[H_{V}\cap\mathbb{C}\left[Z\right]\right]=H_{V}$.
\end{enumerate}
\begin{proof}
The proof follows by Theorem \ref{Theorem 36} and Proposition \ref{Proposition 34}.
\end{proof}
\begin{notation}
For any variety $V\subseteq B_{d}$, if $V$ is decomposed into finite
union of irreducible varieties; say, $V=V_{1}\cup\cdot\cdot\cdot\cup V_{k}$
(for some $k\in\mathbb{N}$) , where for every $i\in\left\{ 1,\ldots,k\right\} $,
$V_{i}$ is an irreducible variety in $B_{d}$, then, 
\end{notation}

\begin{itemize}
\item For every $i\in\left\{ 1,\ldots,k\right\} $, Set $S_{i}\left(V\right):=B_{d}\cap\textrm{span}\left(V_{i}\right)$.
\item $S\left(V\right):=\bigcup_{i=1}^{k}S_{i}\left(V\right)$.
\end{itemize}
\begin{rem}
Note that if $V\subseteq B_{d}$ is a variety, which is decomposed
into finite union of irreducible varieties, then $S\left(V\right)$
is a homogeneous variety in $B_{d}$.
\end{rem}

\begin{prop}
\label{Proposition 310}Let $V\subseteq B_{d}$ be a variety. Suppose
that $V$ is decomposed into finite union of irreducible varieties,
and let $V=V_{1}\cup\cdot\cdot\cdot\cup V_{k}$  be the decomposition
of $V$ into a finite union of irreducible varieties. Then,
\end{prop}

\begin{enumerate}
\item $\mathcal{\mathit{H}_{\mathit{V}}\mathit{=}\overline{\sum_{\mathit{i=1}}^{\mathit{k}}\mathcal{\mathit{H}_{\mathit{V_{i}}}}}}$.
\item $H_{S\left(V\right)}=\sum_{i=1}^{k}\mathcal{\mathit{H}_{\mathit{S_{i}\left(V\right)}}}$.
\end{enumerate}
\begin{proof}
\begin{enumerate}
\item Note that $I_{V}=\bigcap_{i=1}^{k}I_{V_{i}}$. Therefore, $H_{V}=\left(I_{V}\right)^{\perp}=\overline{\sum_{i=1}^{k}\left(I_{V_{i}}\right)^{\perp}}$.
Hence, $\mathcal{\mathit{H}_{\mathit{V}}\mathit{=}\overline{\sum_{\mathit{i=1}}^{\mathit{k}}\mathcal{\mathit{H}_{\mathit{V_{i}}}}}}$.
\item Apply $\left(1\right)$ for $S\left(V\right)$ to obtain $H_{S\left(V\right)}=\overline{\sum_{i=1}^{k}\mathcal{\mathit{H}_{\mathit{S_{i}\left(V\right)}}}}$.
Now, by \cite[Lemma 2.7+Corollary 5.8]{hartz2012topological} and
Proposition \ref{Theorem 37}, we obtain that $\sum_{i=1}^{k}\mathcal{\mathit{H}_{\mathit{S_{i}\left(V\right)}}\subseteq\mathit{\mathcal{H}_{d}^{2}}}$
is closed.
\end{enumerate}
\end{proof}

\section{Tractable unions of linear subspaces of $\mathbb{C}^{d}$}

In this section, we study the structure of finite unions of linear
subspaces of $\mathbb{C}^{d}$. Recall that, by Theorem \ref{Theorem 1.34},
if $V$ and $W$ are biholomorphic homogeneous varieties in $B_{d}$,
then there is an invertible linear map $A:\mathbb{C^{\mathit{d}}}\longrightarrow\mathbb{C^{\mathit{d}}}$
s.t. $A\left(V\right)=W$. By Proposition \ref{Proposition 4.1} below,
$A$ is isometric on each linear space spanned by each irreducible
component of $V$. We aim to simplify the evaluation of the norm of
the induced RKHS isomorphism $T_{A}$ by reducing the problem to a
lower dimensional case, whenever it is possible (see Corollary \ref{Corollary 5.6}).
We also aim to deepen our understanding of the structure of such unions
and to make their analysis easier for this purpose. In Subsection
\ref{subsec:Examples-of-tractable}, we present examples of tractable
unions of linear subspaces of $\mathbb{C}^{d}$. 
\begin{prop}
\label{Proposition 4.1}(\cite[Proposition 7.6]{davidson2011isomorphism})
Let $V$ be a homogeneous variety in $\mathbb{C}^{d}$, and let $A$
be a linear map on $\mathbb{C}^{d}$ such that $\left\Vert A\left(x\right)\right\Vert =\left\Vert x\right\Vert $
for all $x\in V$. If $V=\cup_{i=1}^{k}V_{i}$ is the decomposition
of  $V$ into irreducible components, then $A$ is isometric on $\textrm{span}\left(V_{i}\right)$
for $1\leq i\leq k$.
\end{prop}

\subsection{Preliminaries }
\begin{notation}
If $E\subseteq\mathbb{C}^{d}$ is a subspace, then $E^{\perp}$ denotes
the orthogonal complement of $E$ in $\mathbb{C}^{d}$.
\end{notation}

\begin{notation}
For any linear  transformation $A:\mathbb{C^{\mathit{d}}}\longrightarrow\mathbb{C^{\mathit{d}}}$
and a (linear) subspace $E\subseteq\mathbb{C^{\mathit{d}}}$, $A\mid_{E}:E\longrightarrow A\left(E\right)$
is the linear  transformation that is given by 
\[
A\mid_{E}\left(x\right)=A\left(x\right)\:,x\in E\cdot
\]
\end{notation}

\begin{lem}
\label{Lemma 4.4}Let $A:\mathbb{C^{\mathit{d}}}\longrightarrow\mathbb{C^{\mathit{d}}}$
be a linear  transformation and let $E\subseteq\mathbb{C^{\mathit{d}}}$
be a subspace. Then, the following are equivalent.
\end{lem}

\begin{enumerate}
\item $A\left(E\right)\perp A\left(E^{\perp}\right)$.
\item $A^{*}A\left(E\right)\subseteq E$.
\item $\left(A\mid_{E}\right)^{*}=A^{*}\mid_{A\left(E\right)}$ .
\end{enumerate}
\begin{proof}
Let $A:\mathbb{C^{\mathit{d}}}\longrightarrow\mathbb{C^{\mathit{d}}}$
be a linear  transformation and let $E\subseteq\mathbb{C^{\mathit{d}}}$
be a subspace. 
\begin{itemize}
\item $\left(1\right)\Rightarrow\left(2\right)$: Suppose that $A\left(E\right)\perp A\left(E^{\perp}\right)$.
Let $x\in E$ and $y\in E^{\perp}$. Then, 
\[
\left\langle A^{*}Ax,y\right\rangle =\left\langle Ax,Ay\right\rangle =0\cdot
\]
 Therefore, $\mathit{A^{*}A\left(E\right)\subseteq\left(E^{\perp}\right)^{\perp}=E}$.
\item $\left(2\right)\Rightarrow\left(1\right)$: Suppose that $\mathit{A^{*}A\left(E\right)\subseteq E}$.
Then, for every $x\in E$ and $y\in E^{\perp}$,$\left\langle Ax,Ay\right\rangle =\left\langle A^{\ast}Ax,y\right\rangle =0$.
Therefore, $A\left(E\right)\perp A\left(E^{\perp}\right)$.
\item $\left(2\right)\Rightarrow\left(3\right)$: Suppose that $A^{*}A\left(E\right)\subseteq E$.
Since $\left(A\mid_{E}\right)^{*}=P_{E}A^{*}\mid_{A\left(E\right)}$,
it follows that $\left(A\mid_{E}\right)^{*}=A^{*}\mid_{A\left(E\right)}$. 
\item $\left(3\right)\Rightarrow\left(2\right)$: Suppose that $\left(A\mid_{E}\right)^{*}=A^{*}\mid_{A\left(E\right)}$.
Then, 
\[
A^{*}A\left(E\right)={\bf \mathit{\left(A\mid_{E}\right)^{\ast}A\left(E\right)={\bf \textrm{image}\left(\mathit{A\mid_{E}}\right)^{*}}\subseteq E\cdot}}
\]
\end{itemize}
\end{proof}
\begin{notation}
For any linear  transformation $A:\mathbb{C^{\mathit{d}}}\longrightarrow\mathbb{C^{\mathit{d}}}$
and a (linear) subspace $E\subseteq\mathbb{C^{\mathit{d}}}$,
\end{notation}

\[
E_{1}\left(A\right):=ker\left(A^{*}A-{\bf Id_{\mathbb{C}^{\mathit{d}}}}\right)\cdot
\]

\begin{prop}
\label{Proposition 4.6}Let $A:\mathbb{C^{\mathit{d}}}\longrightarrow\mathbb{C^{\mathit{d}}}$
be a linear  transformation and let $E\subseteq\mathbb{C^{\mathit{d}}}$
be a subspace. Then, the following are equivalent.
\end{prop}

\begin{enumerate}
\item $E\subseteq E_{1}\left(A\right)$.
\item $A$ is isometric on $E$ and $A\left(E\right)\perp A\left(E^{\perp}\right)$.
\end{enumerate}
\begin{proof}
Assume that $E\subseteq E_{1}\left(A\right)$. Then, for every $x\in E$,
$A^{*}Ax=x$. Therefore, $\mathit{A^{*}A\left(E\right)\subseteq E}$.
By Lemma \ref{Lemma 4.4}, it follows that $A\left(E\right)\perp A\left(E^{\perp}\right)$.
Note that for every $x\in E$, 
\[
\left\Vert Ax\right\Vert ^{2}=\left\langle Ax,Ax\right\rangle =\left\langle A^{*}Ax,x\right\rangle =\left\langle x,x\right\rangle =\left\Vert x\right\Vert ^{2}\cdot
\]
 Hence, $A$ is isometric on $E$. Conversely, assume that $A$ is
isometric on $E$ and $A\left(E\right)\perp A\left(E^{\perp}\right)$.
Since $A\left(E\right)\perp A\left(E^{\perp}\right)$, it follows,
by Lemma \ref{Lemma 4.4}, that 
\[
\left(A^{*}A\right)\mid_{E}=A^{*}\mid_{A\left(E\right)}A\mid_{E}=\left(A\mid_{E}\right)^{*}A\mid_{E}\cdot
\]
Since, $A\mid_{E}$ is isometric on $E$, it follows that $\left(A\mid_{E}\right)^{*}A\mid_{E}={\bf Id_{\mathit{}}}$.
Therefore, $A^{*}A\mid_{E}={\bf Id_{\mathit{}}}$. Hence, $E\subseteq E_{1}\left(A\right)$.
\end{proof}
\begin{notation}
If $S\subseteq\mathbb{C}^{d}$ is a linear subspace, then $P_{S}:\mathbb{C}^{d}\to\mathbb{C}^{d}$
denotes the orthogonal projection onto $S$.
\end{notation}

\begin{prop}
\label{Proposition 4.8}Let $A:\mathbb{C^{\mathit{d}}}\longrightarrow\mathbb{C^{\mathit{d}}}$
be a linear  transformation. Let $E\subseteq\mathbb{C^{\mathit{d}}}$
be a subspace s.t. $E\subseteq E_{1}\left(A\right)$. Then, for every
$x\in\mathbb{C^{\mathit{d}}}$,
\[
\left\Vert Ax\right\Vert =\left\Vert x\right\Vert 
\]
 if and only if 
\[
\left\Vert AP_{E^{\perp}}\left(x\right)\right\Vert =\left\Vert P_{E^{\perp}}\left(x\right)\right\Vert \cdot
\]
\end{prop}

\begin{proof}
Let $x\in\mathbb{C^{\mathit{d}}}$. Then, $x=P_{E}\left(x\right)+P_{E^{\perp}}\left(x\right)$
and $Ax=AP_{E}\left(x\right)+AP_{E^{\perp}}\left(x\right)$. Since
$E\subseteq E_{1}\left(A\right)$, it follows, by Proposition \ref{Proposition 4.6},
that $\left\Vert AP_{E}\left(x\right)\right\Vert =\left\Vert P_{E}\left(x\right)\right\Vert $
and $AP_{E}\left(x\right)\perp AP_{E^{\perp}}\left(x\right)$. Therefore,
\[
\left\Vert x\right\Vert ^{2}=\left\Vert P_{E}\left(x\right)\right\Vert ^{2}+\left\Vert P_{E^{\perp}}\left(x\right)\right\Vert ^{2}
\]
 and

\[
\left\Vert Ax\right\Vert ^{2}=\left\Vert P_{E}\left(x\right)\right\Vert ^{2}+\left\Vert AP_{E^{\perp}}\left(x\right)\right\Vert ^{2}\cdot
\]
 This yields that 
\[
\left\Vert Ax\right\Vert ^{2}-\left\Vert x\right\Vert ^{2}=\left\Vert AP_{E^{\perp}}\left(x\right)\right\Vert ^{2}-\left\Vert P_{E^{\perp}}\left(x\right)\right\Vert ^{2}\cdot
\]
It follows that, $\left\Vert Ax\right\Vert =\left\Vert x\right\Vert $
if and only if $\left\Vert AP_{E^{\perp}}\left(x\right)\right\Vert =\left\Vert P_{E^{\perp}}\left(x\right)\right\Vert $.
\end{proof}
\begin{prop}
\label{Proposition 4.9}Let $A:\mathbb{C^{\mathit{d}}}\to\mathbb{C^{\mathit{d}}}$
be a linear  transformation and let $S\subseteq\mathbb{C^{\mathit{d}}}$
a linear subspace. If $E\subseteq E_{1}\left(A\right)$ is a subspace
then the following are equivalent:
\end{prop}

\begin{enumerate}
\item $A$ is isometric on $S$.
\item $A$ is isometric on $E\oplus P_{E^{\perp}}\left(S\right)$.
\end{enumerate}
\begin{proof}
Let $S\subseteq\mathbb{C^{\mathit{d}}}$ be a subspace. If $A$ is
isometric on $S$ then, by Proposition \ref{Proposition 4.8}, $A$
is isometric on $P_{E^{\perp}}\left(S\right)$. Since $E\subseteq E_{1}\left(A\right)$,
it follows, by Proposition \ref{Proposition 4.6}, that $A$ is isometric
on $E$ and $AP_{E^{\perp}}\left(S\right)\perp AE$. Therefore, $A$
is isometric on $E\oplus P_{E^{\perp}}\left(S\right)\cdot$ Conversely,
note that $S\subseteq E\oplus P_{E^{\perp}}\left(S\right)$. Therefore,
if $A$ is isometric on $E\oplus P_{E^{\perp}}\left(S\right)$ then
it is isometric on $S$. 
\end{proof}
\begin{thm}
\label{Theorem 4.10.}Let $M=\cup_{i=1}^{k}M_{i}$ be a union of subspaces
of $\mathbb{C}^{d}$ s.t. $\sum_{i=1}^{k}M_{i}=\mathbb{C}^{d}$. Suppose
that $A:\mathbb{C}^{d}\longrightarrow\mathbb{C}^{d}$ is a linear
transformation, which is isometric on $M$. Then, 
\[
\cap_{i=1}^{k}M_{i}\subseteq E_{1}\left(A\right)\cdot
\]
\end{thm}

\begin{proof}
Set $E:=\cap_{i=1}^{k}M_{i}$. Since $E\subseteq M$, it follows that
$A$ is isometric on $E$. By Proposition \ref{Proposition 4.6},
it remains to show that $AE^{\perp}\perp AE$: Note that 
\[
\mathbb{C}^{d}=\sum_{i=1}^{k}M_{i}=\sum_{i=1}^{k}\left(E\oplus P_{E^{\perp}}M_{i}\right)=E\oplus\sum_{i=1}^{k}P_{E^{\perp}}M_{i}.
\]
 It follows that $E^{\perp}=\sum_{i=1}^{k}P_{E^{\perp}}M_{i}$. Now,
since for every $i\in\left\{ 1,\ldots,k\right\} $, $A$ is isometric
on $M_{i}$ and $E\subseteq M_{i}$, it follows that $AP_{E^{\perp}}M_{i}\perp AE$
for all $i\in\left\{ 1,\ldots,k\right\} $. Therefore, $AE^{\perp}\perp AE$.
\end{proof}

\subsection{Maximal linear subspaces on which a given invertible linear  transformation
acts isometrically}
\begin{notation}
For any linear  transformation $A:\mathbb{C^{\mathit{d}}}\longrightarrow\mathbb{C^{\mathit{d}}}$
and a subspace $E\subseteq\mathbb{C^{\mathit{d}}}$, we will use the
following notations.
\end{notation}

\begin{itemize}
\item $E_{\lambda}\left(A\right):=ker\left(A^{*}A-\lambda^{2}{\bf Id_{\mathbb{C}^{\mathit{d}}}}\right)$,
for any $\lambda\in\mathbb{C}$.
\item $E_{+}\left(A\right):=\sum_{\lambda\in\sigma\left(\sqrt{A^{*}A}\right):\lambda>1}E_{\lambda}\left(A\right)$.
\item $E_{-}\left(A\right):=\sum_{\lambda\in\sigma\left(\sqrt{A^{*}A}\right):\lambda<1}E_{\lambda}\left(A\right)$.
\end{itemize}
\begin{prop}
\label{Proposition 3.2}Let $A:\mathbb{C^{\mathit{d}}}\longrightarrow\mathbb{C^{\mathit{d}}}$
be a linear  transformation. Then,
\end{prop}

\begin{enumerate}
\item $\left\Vert Ax\right\Vert >\left\Vert x\right\Vert $ for all $0\neq x\in E_{+}\left(A\right)$.
\item $\left\Vert Ax\right\Vert <\left\Vert x\right\Vert $ for all $0\neq x\in E_{-}\left(A\right)$.
\end{enumerate}
\begin{proof}
Let $A:\mathbb{C^{\mathit{d}}}\longrightarrow\mathbb{C^{\mathit{d}}}$
be a linear  transformation.
\begin{enumerate}
\item Let $0\neq x\in E_{+}\left(A\right)$. Then, 
\[
x=\sum_{\lambda\in\sigma\left(\sqrt{A^{*}A}\right):\lambda>1}P_{E_{\lambda}\left(A\right)}\left(x\right)
\]
 and 
\[
A^{*}Ax=\sum_{\lambda\in\sigma\left(\sqrt{A^{*}A}\right):\lambda>1}\lambda^{2}P_{E_{\lambda}\left(A\right)}\left(x\right)
\]
Therefore, 
\begin{flalign*}
\left\Vert Ax\right\Vert ^{2}-\left\Vert x\right\Vert ^{2} & =\left\langle A^{*}Ax-x,x\right\rangle \\
 & =\left\langle \sum_{\lambda\in\sigma\left(\sqrt{A^{*}A}\right):\lambda>1}\left(\lambda^{2}-1\right)P_{E_{\lambda}}\left(x\right),\sum_{\lambda\in\sigma\left(\sqrt{A^{*}A}\right):\lambda>1}P_{E_{\lambda}}\left(x\right)\right\rangle \\
 & =\sum_{\lambda\in\sigma\left(\sqrt{A^{*}A}\right):\lambda>1}\left(\lambda^{2}-1\right)\left\Vert P_{E_{\lambda}}\left(x\right)\right\Vert ^{2}>0\cdot\\
\end{flalign*}
\item The proof is the same as that of $\left(1\right)$.
\end{enumerate}
\end{proof}

\begin{cor}
\label{Corollary 3.3}Let $A:\mathbb{C^{\mathit{d}}}\longrightarrow\mathbb{C^{\mathit{d}}}$
be an\textbf{ }invertible linear  transformation. Suppose that $A$
is isometric on a subspace $S\subseteq\mathbb{C^{\mathit{d}}}$. Then,
\[
\dim S\leq\dim E_{1}\left(A\right)+\min\left\{ \dim E_{-}\left(A\right),\dim E_{+}\left(A\right)\right\} \leq\frac{d+\dim E_{1}\left(A\right)}{2}\cdot
\]
\end{cor}

\begin{proof}
Since $A$ is isometric on $S$, it follows, by Proposition \ref{Proposition 3.2},
that $S\cap E_{+}\left(A\right)=\left\{ 0\right\} $ and $S\cap E_{-}\left(A\right)=\left\{ 0\right\} $.
Therefore, 
\[
\dim\left(S\right)\leq d-\dim E_{+}\left(A\right)
\]
 and 
\[
\dim\left(S\right)\leq d-\dim E_{-}\left(A\right)\cdot
\]
Now, note that 
\[
d=\dim E_{-}\left(A\right)+\dim E_{1}\left(A\right)+\dim E_{+}\left(A\right)\cdot
\]
 Therefore,
\[
\dim\left(S\right)\leq\dim E_{1}\left(A\right)+\dim E_{-}\left(A\right)
\]
 and 
\[
\dim\left(S\right)\leq\dim E_{1}\left(A\right)+\dim E_{+}\left(A\right)\cdot
\]
 It follows that 
\[
\dim S\leq\dim E_{1}\left(A\right)+\min\left\{ \dim E_{-}\left(A\right),\dim E_{+}\left(A\right)\right\} \cdot
\]
 Finally, note that 
\[
\min\left\{ \dim E_{-}\left(A\right),\dim E_{+}\left(A\right)\right\} \leq\frac{d-\dim E_{1}\left(A\right)}{2}\cdot
\]
 Therefore, 
\[
\dim S\leq\dim E_{1}\left(A\right)+\min\left\{ \dim E_{-}\left(A\right),\dim E_{+}\left(A\right)\right\} \leq\frac{d+\dim E_{1}\left(A\right)}{2}\cdot
\]
\end{proof}
\begin{cor}
\label{Corollary 3.4}Let $A:\mathbb{C^{\mathit{d}}}\longrightarrow\mathbb{C^{\mathit{d}}}$
be an invertible linear  transformation. Suppose that $A$ is isometric
on a subspace $S\subseteq\mathbb{C^{\mathit{d}}}$ with 
\[
\dim S>\min\left\{ \dim E_{-}\left(A\right),\dim E_{+}\left(A\right)\right\} \cdot
\]
Then, $E_{1}\left(A\right)\neq\left\{ 0\right\} $. In particular,
if $\dim S>\frac{d}{2}$, then $E_{1}\left(A\right)\neq\left\{ 0\right\} $. 
\end{cor}

\begin{proof}
By Corollary \ref{Corollary 3.3}, it follows that $0<\dim E_{1}\left(A\right)$.
\end{proof}
\begin{prop}
\label{Proposition 3.5}Let $A:\mathbb{C^{\mathit{d}}}\longrightarrow\mathbb{C^{\mathit{d}}}$
be an invertible linear  transformation. Suppose that $A$ is isometric
on a subspace $S\subseteq\mathbb{C^{\mathit{d}}}$. Then,
\[
\left\Vert A\right\Vert =1\Longrightarrow S\subseteq E_{1}\left(A\right)\cdot
\]
\end{prop}

\begin{proof}
Since $A$ is isometric on $S$, it follows, by Theorem \ref{Proposition 4.9},
that $A$ is isometric on $M:=E_{1}\left(A\right)\oplus P_{E_{1}^{\perp}\left(A\right)}\left(S\right)$.
Therefore, by Proposition \ref{Proposition 3.2}, $M\cap E_{-}\left(A\right)=\left\{ 0\right\} $.
Therefore, 
\[
\dim E_{1}\left(A\right)\leq\dim M\leq d-\dim E_{-}\left(A\right)=\dim E_{+}+\dim E_{1}\left(A\right)\cdot
\]
Since $\left\Vert A\right\Vert =1$, we have $\left\Vert A^{*}A\right\Vert =\left\Vert A\right\Vert ^{2}=1$.
Therefore, $\dim E_{+}\left(A\right)=0$. It follows that $\dim M=\dim E_{1}\left(A\right)$.
Now, since $E_{1}\left(A\right)\subseteq M$, it follows that $E_{1}\left(A\right)=M$.
Finally, note that $S\subseteq M$.
\end{proof}
\begin{thm}
Let $M=\cup_{i=1}^{k}M_{i}$ be a union of subspaces of $\mathbb{C}^{d}$
s.t. $\sum_{i=1}^{k}M_{i}=\mathbb{C}^{d}$. Suppose that $A:\mathbb{C}^{d}\longrightarrow\mathbb{C}^{d}$
is an invertible linear transformation, which is isometric on $M$.
If $\left\Vert A\right\Vert =1$ then $A$ is a unitary. 
\end{thm}

\begin{proof}
By Proposition \ref{Proposition 3.5}, it follows that $M\subseteq E_{1}\left(A\right)$.
Therefore, $\mathbb{C}^{d}=\sum_{i=1}^{k}M_{i}\subseteq E_{1}\left(A\right)$.Therefore,
$E_{1}\left(A\right)=\mathbb{C^{\mathit{d}}}$ . Hence, $A$ is a
unitary.
\end{proof}
\begin{prop}
\label{Proposition 3.8}Let $A:\mathbb{C^{\mathit{d}}}\longrightarrow\mathbb{C^{\mathit{d}}}$
be a linear  transformation. Let $E\subseteq\mathbb{C^{\mathit{d}}}$
be a subspace s.t. $E\subseteq E_{1}\left(A\right)$. Then,
\end{prop}

\[
E_{1}\left(A\mid_{E^{\perp}}\right):={\bf \textrm{ker}\left(\mathit{\left(A\mid_{E^{\perp}}\right)^{*}A\mid_{E^{\perp}}-{\bf Id_{\mathit{E^{\perp}}}}}\right)\mathit{=E^{\perp}\cap E_{1}\left(A\right)}}\cdot
\]
 In particular, $E_{1}\left(A\mid_{E_{1}^{\perp}\left(A\right)}\right)=\left\{ 0\right\} $.
\begin{proof}
Since $E\subseteq E_{1}\left(A\right)$, it follows, by Proposition
\ref{Proposition 4.6}, that $A\left(E\right)\perp A\left(E^{\perp}\right)$.
It follows, by Lemma \ref{Lemma 4.4}, that 
\[
\left(A\mid_{E^{\perp}}\right)^{*}=A^{*}\mid_{A\left(E^{\perp}\right)}\cdot
\]
 Therefore, {\small{}
\[
{\bf \textrm{ker}\left(\mathit{\left(A\mid_{E^{\perp}}\right)^{*}A\mid_{E^{\perp}}-{\bf Id_{\mathit{E^{\perp}}}}}\right)=\ker\mathit{\left(\left(A^{*}A-{\bf Id}\right)\mid_{E^{\perp}}\right)}\mathit{=}}E^{\perp}\cap E_{1}\left(A\right)\cdot
\]
}{\small\par}
\end{proof}
\begin{lem}
\label{lemma 3.8}Let $M=M_{1}\cup M_{2}$ be a union of subspaces
of $\mathbb{C^{\mathit{d}}}$ s.t. $M_{1}+M_{2}=\mathbb{C}^{\mathit{d}}$.
Suppose that $A:\mathbb{C}^{d}\longrightarrow\mathbb{C}^{d}$ is an
invertible linear transformation, which is isometric on $M_{1}\cup M_{2}$.
Then, 
\end{lem}

\[
\dim P_{E_{1}^{\perp}\left(A\right)}\left(M_{1}\right)=\dim P_{E_{1}^{\perp}\left(A\right)}\left(M_{2}\right)=\frac{\dim E_{1}^{\perp}\left(A\right)}{2}\cdot
\]

\begin{proof}
First, by Proposition \ref{Proposition 3.8}, we have $E_{1}\left(A\mid_{E_{1}^{\perp}\left(A\right)}\right)=\left\{ 0\right\} $.
Second, by Proposition \ref{Proposition 4.9}, $A\mid_{E_{1}^{\perp}\left(A\right)}$
is isometric on 
\[
P_{E_{1}^{\perp}\left(A\right)}\left(M_{1}\right)\cup P_{E_{1}^{\perp}\left(A\right)}\left(M_{2}\right)\cdot
\]
Also, since $P_{E_{1}^{\perp}\left(A\right)}\left(M_{1}\right)+P_{E_{1}^{\perp}\left(A\right)}\left(M_{2}\right)=E_{1}^{\perp}\left(A\right)$,
it follows, by Theorem \ref{Theorem 4.10.}, that
\[
P_{E_{1}^{\perp}\left(A\right)}\left(M_{1}\right)\cap P_{E_{1}^{\perp}\left(A\right)}\left(M_{2}\right)\subseteq E_{1}\left(A\mid_{E_{1}^{\perp}\left(A\right)}\right)=\left\{ 0\right\} \cdot
\]
Therefore,
\[
\dim P_{E_{1}^{\perp}\left(A\right)}\left(M_{1}\right)+\dim P_{E_{1}^{\perp}\left(A\right)}\left(M_{2}\right)=\dim E_{1}^{\perp}\left(A\right)\cdot
\]
 Without loss of generality, assume that $\dim P_{E_{1}^{\perp}\left(A\right)}\left(M_{1}\right)<\dim P_{E_{1}^{\perp}\left(A\right)}\left(M_{2}\right)$,
then 
\[
2\dim P_{E_{1}^{\perp}\left(A\right)}\left(M_{2}\right)>\dim E_{1}^{\perp}\left(A\right)
\]
or
\[
\dim P_{E_{1}^{\perp}\left(A\right)}\left(M_{2}\right)>\frac{\dim E_{1}^{\perp}\left(A\right)}{2}\cdot
\]
Since $A\mid_{E_{1}^{\perp}\left(A\right)}$ is isometric on $P_{E_{1}^{\perp}\left(A\right)}\left(M_{2}\right)$,
it follows, by Corollary \ref{Corollary 3.4}, that $E_{1}\left(A\mid_{E_{1}^{\perp}\left(A\right)}\right)\neq\left\{ 0\right\} $.
But, $E_{1}\left(A\mid_{E_{1}^{\perp}\left(A\right)}\right)=\left\{ 0\right\} $.
Therefore, 
\[
\dim P_{E_{1}^{\perp}\left(A\right)}\left(M_{1}\right)=\dim P_{E_{1}^{\perp}\left(A\right)}\left(M_{2}\right)=\frac{\dim E_{1}^{\perp}\left(A\right)}{2}\cdot
\]
 
\end{proof}
\begin{lem}
\label{lemma 3.9}Let $r,s,d\in\mathbb{N}$ s.t. $r+s\leq d$. Let
$\lambda_{1},\ldots,\lambda_{d}\in\mathbb{R}$ s.t. 
\[
\lambda_{1}\geq\ldots\geq\lambda_{r}>1>\lambda_{r+1}\geq\ldots\geq\lambda_{r+s}>0
\]
and 
\[
\lambda_{i}=1,\textrm{for all }i\in\left\{ r+s+1,\ldots,d\right\} \cdot
\]
Let $A:\mathbb{C}^{d}\longrightarrow\mathbb{C}^{d}$ be the linear
transformation, which is defined by 
\[
A\left(e_{i}\right)=\lambda_{i}e_{i},\quad i\in\left\{ 1,\ldots,d\right\} \cdot
\]
Then, there exists a subspace $S\subseteq\mathbb{C}^{d}$ of $\dim S=\left(d-\left(r+s\right)\right)+\min\left\{ r,s\right\} $
s.t. $A$ is isometric on $S$.
\end{lem}

\begin{proof}
Assume that $r=\min\left\{ r,s\right\} $ (the proof of the case $s=\min\left\{ r,s\right\} $
will be the same). Let $\alpha_{1},\ldots,\alpha_{r}\in[0,2\pi)$.
Let $S\subseteq\mathbb{C}^{d}$ be the set of all $\left(x_{1},\ldots,x_{d}\right)\in\mathbb{C}^{d}$
s.t. for every $i\in\left\{ 1,\ldots,r\right\} $, 
\[
x_{r+i}=\frac{\sqrt{\lambda_{i}^{2}-1}}{\sqrt{1-\lambda_{r+i}^{2}}}e^{i\alpha_{i}}x_{i},\text{for all {\it i\ensuremath{\in}\ensuremath{\left\{  1,\ldots,r\right\} } }}
\]
and 
\[
x_{i}=0\text{, for all {\it i\ensuremath{\in}\ensuremath{\left\{  2r+1,\ldots,r+s\right\} } }\ensuremath{\cdot}}
\]
Note that for every $\left(x_{1},\ldots,x_{d}\right)\in\mathbb{C}^{d}$,
\[
\left\Vert Ax\right\Vert ^{2}=\left\Vert x\right\Vert ^{2}
\]
 if and only if 
\[
\sum_{i=1}^{d}\left(\lambda_{i}^{2}-1\right)\left|x_{i}\right|^{2}=0\cdot
\]
Also, note that for every $i\in\left\{ 1,\ldots,r\right\} $, 
\[
\left|x_{r+i}\right|^{2}=\frac{\lambda_{i}^{2}-1}{1-\lambda_{r+i}^{2}}\left|x_{i}\right|^{2}\cdot
\]
 Now, let $\left(x_{1},\ldots,x_{d}\right)\in S$, then 
\begin{eqnarray*}
\sum_{i=1}^{d}\left(\lambda_{i}^{2}-1\right)\left|x_{i}\right|^{2} & = & \sum_{i=1}^{2r}\left(\lambda_{i}^{2}-1\right)\left|x_{i}\right|^{2}\\
 & = & \sum_{i=1}^{r}\left(\lambda_{i}^{2}-1\right)\left|x_{i}\right|^{2}+\sum_{i=1}^{r}\left(\lambda_{r+i}^{2}-1\right)\left|x_{r+i}\right|^{2}\\
 & = & \sum_{i=1}^{r}\left(\lambda_{i}^{2}-1\right)\left|x_{i}\right|^{2}-\sum_{i=1}^{r}\left(\lambda_{i}^{2}-1\right)\left|x_{i}\right|^{2}=0\cdot
\end{eqnarray*}
Therefore, $A$ is isometric on $S$ and $\dim S=\left(d-\left(r+s\right)\right)+\min\left\{ r,s\right\} $.
\end{proof}
\begin{cor}
\label{Corollary 3.10}Let $A:\mathbb{C}^{d}\longrightarrow\mathbb{C}^{d}$
be an invertible linear transformation. Then, $A$ is isometric on
a subspace $E\subseteq\mathbb{C}^{d}$ of 
\[
\dim E=\dim E_{1}\left(A\right)+\min\left\{ \dim E_{+}\left(A\right),\dim E_{-}\left(A\right)\right\} \cdot
\]
\end{cor}

\begin{proof}
Let $U:=A\left(\sqrt{A^{*}A}\right)^{-1}$. Then, $U$ is a unitary
and $A=U\sqrt{A^{*}A}$. Since $\sqrt{A^{*}A}$ is positive, it follows
that there exists a diagonal linear transformation $D:\mathbb{C}^{d}\longrightarrow\mathbb{C}^{d}$
as in Lemma \ref{lemma 3.9} and a unitary $U_{1}:\mathbb{C}^{d}\longrightarrow\mathbb{C}^{d}$
s.t. $\sqrt{A^{*}A}=U_{1}^{*}DU_{1}$. By Lemma \ref{lemma 3.9},
there exists a subspace $S\subseteq\mathbb{C}^{d}$ of 
\[
\dim S=\dim E_{1}\left(A\right)+\min\left\{ \dim E_{+}\left(A\right),\dim E_{-}\left(A\right)\right\} \cdot
\]
s.t. $D$ is isometric on $S$. It follows that $\sqrt{A^{*}A}$ is
isometric on $E:=U_{1}^{*}\left(S\right)$. Therefore, $A$ is isometric
on $E$.
\end{proof}
\begin{thm}
\label{Theorem 3.11}Let $M\subseteq\mathbb{C}^{d}$ be a linear subspace.
Suppose that $A:\mathbb{C}^{d}\longrightarrow\mathbb{C}^{d}$ is an
invertible linear transformation, which is isometric on $M$. Then,
there exists a subspace $S\left(M\right)\subseteq\mathbb{C}^{d}$
s.t. 
\end{thm}

\begin{enumerate}
\item $M\subseteq S\left(M\right)$.
\item $A$ is isometric on $S\left(M\right)$.
\item $\dim S\left(M\right)=\dim E_{1}\left(A\right)+\min\left\{ \dim E_{+}\left(A\right),\dim E_{-}\left(A\right)\right\} $.
\end{enumerate}
\begin{proof}
By Corollary \ref{Corollary 3.10}, $A$ is isometric on a subspace
$S$ with 
\[
\dim S=\dim E_{1}\left(A\right)+\min\left\{ \dim E_{+}\left(A\right),\dim E_{-}\left(A\right)\right\} \cdot
\]
Note that $A\mid_{S+M}$ is isometric on $S\cup M$. By Theorem \ref{Proposition 4.9},
it follows that $A\mid_{S+M}$ is isometric on 
\[
P_{\left(S+M\right)\ominus E_{1}\left(A\mid_{S+M}\right)}\left(S\right)\oplus E_{1}\left(A\mid_{S+M}\right)
\]
 and on 
\[
S\left(M\right):=P_{\left(S+M\right)\ominus E_{1}\left(A\mid_{S+M}\right)}\left(M\right)\oplus E_{1}\left(A\mid_{S+M}\right)\cdot
\]
Note that 
\[
S\subseteq P_{\left(S+M\right)\ominus E_{1}\left(A\mid_{S+M}\right)}\left(S\right)\oplus E_{1}\left(A\mid_{S+M}\right)\cdot
\]
 It follows, by Corollary \ref{Corollary 3.3}, that
\[
S=P_{\left(S+M\right)\ominus E_{1}\left(A\mid_{S+M}\right)}\left(S\right)\oplus E_{1}\left(A\mid_{S+M}\right)\cdot
\]
By Lemma \ref{lemma 3.8}, it follows that $\dim S\left(M\right)=\dim S$.
Therefore, 
\[
\dim S\left(M\right)=\dim E_{1}\left(A\right)+\min\left\{ \dim E_{+}\left(A\right),\dim E_{-}\left(A\right)\right\} \cdot
\]
\end{proof}
\begin{defn}
Let $M\subseteq\mathbb{C}^{d}$ be a linear subspace. Suppose that
$A:\mathbb{C}^{d}\longrightarrow\mathbb{C}^{d}$ is an invertible
linear transformation, which is isometric on $M$. We say that $M$
is \textbf{maximal with respect to} $A$, if for every subspace $S\subseteq\mathbb{C}^{d}$
s.t. $M\subseteq S$, the following holds: $A$ is isometric on $S$
implies $S=M$.
\end{defn}

\begin{thm}
Let $M\subseteq\mathbb{C}^{d}$ be a linear subspace. Suppose that
$A:\mathbb{C}^{d}\longrightarrow\mathbb{C}^{d}$ is an invertible
linear transformation, which is isometric on $M$. Then $M$ is maximal
with respect to $A$ if and only if 
\[
\dim M=\dim E_{1}\left(A\right)+\min\left\{ \dim E_{+}\left(A\right),\dim E_{-}\left(A\right)\right\} \cdot
\]
\end{thm}

\begin{proof}
If $\dim M=\dim E_{1}\left(A\right)+\min\left\{ \dim E_{+}\left(A\right),\dim E_{-}\left(A\right)\right\} $
then, by Corollary \ref{Corollary 3.3}, it follows that $S$ is maximal
with respect to $A$. Conversely, suppose that $M$ is maximal with
respect to $A$. Then, by Theorem \ref{Theorem 3.11}, there exists
a subspace $S\left(M\right)\subseteq\mathbb{C}^{d}$ s.t. $M\subseteq S\left(M\right)$,
$A$ is isometric on $S\left(M\right)$ and 
\[
\dim S\left(M\right)=\dim E_{1}\left(A\right)+\min\left\{ \dim E_{+}\left(A\right),\dim E_{-}\left(A\right)\right\} \cdot
\]
Since $M$ is maximal, it follows that $M=S\left(M\right)$. Therefore,
\[
\dim M=\dim E_{1}\left(A\right)+\min\left\{ \dim E_{+}\left(A\right),\dim E_{-}\left(A\right)\right\} \cdot
\]
\end{proof}
\begin{prop}
\label{Proposition 4.7}Let $M=\cup_{i=1}^{k}M_{i}$ be a union of
linear subspaces of $\mathbb{C}^{d}$ s.t. $\sum_{i=1}^{k}M_{i}=\mathbb{C}^{d}$.
Suppose that $A:\mathbb{C}^{d}\longrightarrow\mathbb{C}^{d}$ is an
invertible linear transformation, which is isometric on $M$. Suppose
that exists $m\in\mathbb{N}$ s.t. for every $i\in\left\{ 1,\ldots,k\right\} $,
$\dim M_{i}=m$. If $\dim M_{i}\cap M_{j}=m-1$ for all $i\neq j$
then
\end{prop}

\[
\dim E_{1}\left(A\right)\geq m-1\cdot
\]

\begin{proof}
Choose $r\in\left\{ 2,\ldots,k\right\} $ s.t. $\sum_{i=1}^{r}M_{i}=\mathbb{C}^{d}$
with $M_{r}\nsubseteq\sum_{i=1}^{r-1}M_{i}$. Then, 
\[
\dim M_{r}\cap\sum_{i=1}^{r-1}M_{i}\leq m-1\cdot
\]
Now, let $j\in\left\{ 1,\ldots,r-1\right\} $. Then, 
\[
m-1=\dim M_{r}\cap M_{j}\leq\dim M_{r}\cap\sum_{i=1}^{r-1}M_{i}\leq m-1\cdot
\]
It follows that for every $j\in\left\{ 1,\ldots,r-1\right\} $,
\[
M_{r}\cap M_{j}=M_{r}\cap\sum_{i=1}^{r-1}M_{i}\cdot
\]
Therefore, 
\[
\cap_{i=1}^{r}M_{i}=\cap_{j=1}^{r-1}\left(M_{r}\cap M_{j}\right)=M_{r}\cap\sum_{i=1}^{r-1}M_{i}\cdot
\]
Since $\dim M_{r}\cap\sum_{i=1}^{r-1}M_{i}=m-1$, it follows that
$\dim\cap_{i=1}^{r}M_{i}=m-1\cdot$ Since $\sum_{i=1}^{r}M_{i}=\mathbb{C}^{d}$,
it follows, by Theorem \ref{Theorem 4.10.}, that $\cap_{i=1}^{r}M_{i}\subseteq E_{1}\left(A\right)$.
Therefore, $\dim E_{1}\left(A\right)\geq m-1$.
\end{proof}
\begin{prop}
Let $k,d\in\mathbb{N}$ s.t. $k\geq2$. Let $M=\cup_{i=1}^{k}M_{i}$
be a union of subspaces of $\mathbb{C}^{d}$ with $\sum_{i=1}^{k}M_{i}=\mathbb{C}^{d}$
s.t. for every $i\in\left\{ 1,\ldots,k\right\} $, $M_{i}\nsubseteq\sum_{j\neq i}M_{j}$.
Suppose that exists $m\in\mathbb{N}$ s.t. for every $i\in\left\{ 1,\ldots,k\right\} $,
$\dim M_{i}=m$. Then, the following are equivalent:
\end{prop}

\begin{enumerate}
\item $\dim\cap_{i=1}^{k}M_{i}=m-1\cdot$
\item For every $i,j\in\left\{ 1,\ldots,k\right\} $, if $i\neq j$ then
$\dim M_{i}\cap M_{j}=m-1$.
\end{enumerate}
\begin{proof}
The implication $\left(1\right)\Longrightarrow\left(2\right)$ is
clear (note that for every $i,j\in\left\{ 1,\ldots,k\right\} $, if
$i\neq j$ then $\cap_{i=1}^{k}M_{i}\subseteq M_{i}\cap M_{j}$).
For the converse, see the proof of Proposition \ref{Proposition 4.7}.
\end{proof}
\begin{cor}
\label{Corollary 3.16} Let $k,d\in\mathbb{N}$ s.t. $k\geq2$. Let
$M=\cup_{i=1}^{k}M_{i}$ be a union of subspaces of $\mathbb{C}^{d}$
with $\sum_{i=1}^{k}M_{i}=\mathbb{C}^{d}$ s.t. for every $i\in\left\{ 1,\ldots,k\right\} $,
$M_{i}\nsubseteq\sum_{j\neq i}M_{j}$. Suppose that for every $i\in\left\{ 1,\ldots,k\right\} $,
$\dim M_{i}=2$. Then, the following are equivalent:
\end{cor}

\begin{enumerate}
\item $\cap_{i=1}^{k}M_{i}\neq\left\{ 0\right\} \cdot$
\item For every $i,j\in\left\{ 1,\ldots,k\right\} $, if $i\neq j$ then
$M_{i}\cap M_{j}\neq\left\{ 0\right\} $.
\end{enumerate}

\subsection{\label{subsec:Examples-of-tractable}Examples of tractable unions
of linear subspaces of $\mathbb{C}^{d}$}
\begin{example}
Let $d,k\geq2$. Let $M=\cup_{i=1}^{k}M_{i}$ be a union of subspaces
of $\mathbb{C}^{d}$ s.t. $\sum_{i=1}^{k}M_{i}=\mathbb{C}^{d}$. Suppose
that for every $i,j\in\left\{ 1,\ldots,k\right\} $, 
\[
i\neq j\Longrightarrow M_{i}\cap M_{j}=\left\{ 0\right\} \cdot
\]
 Then, $M$ is a tractable union of linear subspaces of \textbf{$\mathbb{C}^{d}$.}
\end{example}

\begin{example}
Let $d,k\geq2$. Let $M=\cup_{i=1}^{k}M_{i}$ be a union of subspaces
of $\mathbb{C}^{d}$ s.t. $\sum_{i=1}^{k}M_{i}=\mathbb{C}^{d}$. Suppose
that there exists a subspace $\left\{ 0\right\} \neq E\subseteq\mathbb{C}^{d}$
such that for every $i,j\in\left\{ 1,\ldots,k\right\} $, 
\[
i\neq j\Longrightarrow M_{i}\cap M_{j}=E\cdot
\]
 Then, $M$ is a tractable union of linear subspaces of \textbf{$\mathbb{C}^{d}$.}
\end{example}

\begin{proof}
Note that $M=E\oplus\left(\cup_{i=1}^{k}P_{E^{\perp}}M_{i}\right)$
and $\sum_{i=1}^{k}P_{E^{\perp}}M_{i}=E^{\perp}$. Moreover, for every
$i,j\in\left\{ 1,\ldots,k\right\} $, if $i\neq j$ then 
\[
P_{E^{\perp}}M_{i}\cap P_{E^{\perp}}M_{j}=\left\{ 0\right\} \cdot
\]
 Therefore, $\cup_{i=1}^{k}P_{E^{\perp}}M_{i}$ is a tractable union
of linear subspaces of $E^{\perp}$. It follows that $M$ is a tractable
union of linear subspaces of \textbf{$\mathbb{C}^{d}$.}
\end{proof}
\begin{example}
Let $d,k\geq2$. Let $M=\cup_{i=1}^{k}M_{i}$ be a union of subspaces
of $\mathbb{C}^{d}$ s.t. $\sum_{i=1}^{k}M_{i}=\mathbb{C}^{d}$. Suppose
that for every for every $i,j\in\left\{ 1,\ldots,k\right\} $, 
\[
i\neq j\Longrightarrow M_{i}+M_{j}=\mathbb{C}^{d}\cdot
\]
 Then, $M$ is a tractable union of linear subspaces of \textbf{$\mathbb{C}^{d}$.}
\end{example}

\begin{proof}
We will prove the assertion by induction on $d\geq2$: If $d=2$ then
for every $i\in\left\{ 1,\ldots,k\right\} $, $\dim M_{i}=1$. Therefore,
for every $i,j\in\left\{ 1,\ldots,k\right\} $, $i\neq j\Longrightarrow M_{i}\cap M_{j}=\left\{ 0\right\} $.
That is, $M$ is a tractable union of linear subspaces of $\mathbb{C}^{d}$.
Now, assume that the assertion holds for all $d=2,\ldots,n$ for some
$n\in\mathbb{N}$ s.t. $n\geq2$. We will show that the assertion
holds for $d=n+1$: Let $k\geq2$ and let $M=\cup_{i=1}^{k}M_{i}$
be a union of subspaces of $\mathbb{C}^{n+1}$ s.t. $\sum_{i=1}^{k}M_{i}=\mathbb{C}^{n+1}$.
Assume that for every for every $i,j\in\left\{ 1,\ldots,k\right\} $,
\[
i\neq j\Longrightarrow M_{i}+M_{j}=\mathbb{C}^{n+1}\cdot
\]
If for every $i,j\in\left\{ 1,\ldots,k\right\} $, $i\neq j\Longrightarrow M_{i}\cap M_{j}=\left\{ 0\right\} $
then $M$ is a tractable union of subspaces of $\mathbb{C}^{n+1}$.
Else, for every $i,j\in\left\{ 1,\ldots,k\right\} $, set $E_{ij}=M_{i}\cap M_{j}$.
Also, set $E:=\sum_{i=1}^{k}\sum_{j\neq i}E_{ij}$ and $N:=\cup_{i=1}^{k}P_{E^{\perp}}M_{i}$.
Then, $E\neq\left\{ 0\right\} $ and $M\subseteq E\oplus N$. Note
that for every $i,j\in\left\{ 1,\ldots,k\right\} $, if $i\neq j$
then 
\[
P_{E^{\perp}}\left(M_{i}\right)+P_{E^{\perp}}\left(M_{j}\right)=E^{\perp}\cdot
\]
Since $E\neq\left\{ 0\right\} $, it follows that $\dim E^{\perp}\leq n$.
Therefore, by the induction assumption, $N$ is a tractable union
of linear subspaces of \textbf{$E^{\perp}$}. Now, let $A:\mathbb{C}^{n+1}\longrightarrow\mathbb{C}^{n+1}$
be a linear transformation, which is isometric on $M$. By Theorem
\ref{Theorem 4.10.}, it follows that for every $i,j\in\left\{ 1,\ldots,k\right\} $,
if $i\neq j$ then $E_{ij}\subseteq E_{1}\left(A\right)$. Therefore,
$E\subseteq E_{1}\left(A\right)$. By Proposition \ref{Proposition 4.9},
it follows that $A$ is isometric on $E\oplus N$. Therefore, $M$
is a tractable union of subspaces of $\mathbb{C}^{n+1}$.
\end{proof}
\begin{example}
\label{Example 4.4}Let $k,d\in\mathbb{N}$ s.t. $d,k\geq2$ and let
$M=\cup_{i=1}^{k}M_{i}$ be a union of linear subspaces of $\mathbb{C}^{d}$
s.t. $\sum_{i=1}^{k}M_{i}=\mathbb{C}^{d}$. If $\dim M_{1}=d-1$ then
$M$ is a tractable union of linear subspaces of \textbf{$\mathbb{C}^{d}$.}
\end{example}

\begin{proof}
We will prove the assertion by induction on $d\geq2$: If $d=2$ then
for every $i\in\left\{ 1,\ldots,k\right\} $, $\dim M_{i}=1$. Therefore,
for every $i,j\in\left\{ 1,\ldots,k\right\} $, $i\neq j\Longrightarrow M_{i}\cap M_{j}=\left\{ 0\right\} $.
That is, $M$ is a tractable union of linear subspaces of $\mathbb{C}^{d}$.
Now, assume that the assertion holds for all $d=2,\ldots,n$ for some
$n\in\mathbb{N}$ s.t. $n\geq2$. We will show that the assertion
holds for $d=n+1$: Let $k\geq2$ and let $M=\cup_{i=1}^{k}M_{i}$
be a union of subspaces of $\mathbb{C}^{n+1}$ s.t. $\sum_{i=1}^{k}M_{i}=\mathbb{C}^{n+1}$
and $\dim M_{1}=n$. If for every $j\in\left\{ 2,\ldots,k\right\} $,
$M_{1}\cap M_{j}=\left\{ 0\right\} $ then for every $j\in\left\{ 2,\ldots,k\right\} $,
$\dim M_{j}=1$. In this case, for every $i,j\in\left\{ 1,\ldots,k\right\} $,
if $i\neq j$ then $M_{i}\cap M_{j}=\left\{ 0\right\} $. Hence, $M$
is tractable union of subspaces of $\mathbb{C}^{n+1}$. Else, if there
is $j\in\left\{ 2,\ldots,k\right\} $ s.t. $M_{1}\cap M_{j}\neq\left\{ 0\right\} $
then for every $i,j\in\left\{ 1,\ldots,k\right\} $, set $E_{ij}=M_{i}\cap M_{j}$.
Also, set $E:=\sum_{j=2}^{k}E_{1j}$. Then, $\left\{ 0\right\} \neq E\subseteq M_{1}$.
Therefore, 
\[
\dim\sum_{i=1}^{k}P_{E^{\perp}}M_{i}=\dim E^{\perp}\leq n
\]
 and 
\[
\dim P_{E^{\perp}}M_{1}=n-\dim E=\left(\left(n+1\right)-\dim E\right)-1=\dim E^{\perp}-1\cdot
\]
Set $N:=\cup_{i=1}^{k}P_{E^{\perp}}M_{i}$ . Then, by the induction
assumption, $N$ is tractable union of linear subspaces of \textbf{$E^{\perp}$}.
Now, let $A:\mathbb{C}^{d}\longrightarrow\mathbb{C}^{d}$ be an invertible
linear  transformation, which is isometric on $M$. Note that if $j\in\left\{ 2,\ldots,k\right\} $
then $M_{1}+M_{j}=\mathbb{C}^{d}$. By Theorem \ref{Theorem 4.10.},
it follows that for every $j\in\left\{ 2,\ldots,k\right\} $, $E_{1j}\subseteq E_{1}\left(A\right)$.
Therefore, $E\subseteq E_{1}\left(A\right)$. Since $A$ is isometric
on $M$, it follows, by Theorem \ref{Proposition 4.9}, that $A$
is isometric on $E\oplus N$. It is clear that $M\subseteq E\oplus N$.\textbf{
}Therefore, $M$ is tractable union of subspaces of $\mathbb{C}^{n+1}$.
\end{proof}
\begin{cor}
Let $k,d\in\mathbb{N}$ s.t $d\leq3$ and let $M=\cup_{i=1}^{k}M_{i}$
be a union of linear subspaces of $\mathbb{C}^{d}$ s.t. $\sum_{i=1}^{k}M_{i}=\mathbb{C}^{d}$.
Then, $M$ is a tractable union of linear subspaces of \textbf{$\mathbb{C}^{d}$.}
\end{cor}

\begin{proof}
If $d=1$ or $k=1$ then the assertion is clear. If $d\in\left\{ 2,3\right\} $
and $k\geq2$ then one of the following hold:
\end{proof}
\begin{enumerate}
\item For every $i,j\in\left\{ 1,\ldots,k\right\} $, $i\neq j\Longrightarrow M_{i}\cap M_{j}=\left\{ 0\right\} $.
\item $d=3$ and there is $i\in\left\{ 1,\ldots,k\right\} $ s.t. $\dim M_{i}=2$.
In this case, the proof follows from Example \ref{Example 4.4}.
\end{enumerate}
\begin{rem}
For $d\geq4$, one can construct an example of a union of linear subspaces
of $\mathbb{C}^{d}$ to show that the class of tractable unions of
linear subspaces is a proper subclass of all finite unions of linear
subspaces of $\mathbb{C}^{d}$. However,\textbf{ }below we introduce
an additional example of a tractable union of subspaces of $\mathbb{C}^{d}$
for $d\geq4$ (see Theorem \ref{Theorem 4.10} below).
\end{rem}

\begin{lem}
\label{Lemma 3.8} Let $d\in\mathbb{N}$ s.t. $d\geq4$. Let $M=\cup_{i=1}^{k}M_{i}$
be a union of subspaces of $\mathbb{C}^{d}$ with $\sum_{i=1}^{k}M_{i}=\mathbb{C}^{d}$
s.t. for every $i\in\left\{ 1,\ldots,k\right\} $, $\dim M_{i}=2$
and $M_{i}\nsubseteq\sum_{j\neq i}M_{j}$. Suppose that $\cap_{i=1}^{k}M_{i}=\left\{ 0\right\} $.
If for some $s\in\left\{ 1,\ldots,k-1\right\} $, $\cap_{i=1}^{s}M_{i}\neq\left\{ 0\right\} $
and $\cap_{i=s+1}^{k}M_{i}\neq\left\{ 0\right\} $ then
\[
\left(\sum_{i=1}^{s}M_{i}\right)\cap\left(\sum_{i=s+1}^{k}M_{i}\right)=\left\{ 0\right\} \cdot
\]
\end{lem}

\begin{proof}
Since $\cap_{i=1}^{s}M_{i}\neq\left\{ 0\right\} $ and $\cap_{i=s+1}^{k}M_{i}\neq\left\{ 0\right\} $,
it follows that 
\[
\dim\sum_{i=1}^{s}M_{i}=s+1
\]
 and 
\[
\dim\sum_{i=s+1}^{k}M_{i}=k-s+1\cdot
\]
Therefore, if $\left(\sum_{i=1}^{s}M_{i}\right)\cap\left(\sum_{i=s+1}^{k}M_{i}\right)\neq\left\{ 0\right\} $
then, 
\[
d=\dim\sum_{i=1}^{k}M_{i}\leq\dim\left(\sum_{i=1}^{s}M_{i}\right)+\dim\left(\sum_{i=s+1}^{k}M_{i}\right)-1=\left(s+1\right)+\left(k-s+1\right)-1=k+1\cdot
\]
 That is, $k\geq d-1$. Since $\cap_{i=1}^{k}M_{i}=\left\{ 0\right\} $,
it follows, by Corollary \ref{Corollary 3.16}, that there are $i,j\in\left\{ 1,\ldots,k\right\} $
s.t. $i\neq j$ and $M_{i}\cap M_{j}\neq\left\{ 0\right\} $. This
implies that $k\leq d-2$. This contradicts that $k\geq d-1$. Therefore,
\[
\left(\sum_{i=1}^{s}M_{i}\right)\cap\left(\sum_{i=s+1}^{k}M_{i}\right)=\left\{ 0\right\} \cdot
\]
\end{proof}
\begin{thm}
\label{Theorem 4.10}Let $d\in\mathbb{N}$ s.t. $d\geq2$ and let
$M=\cup_{i=1}^{k}M_{i}$ be a union of subspaces of $\mathbb{C}^{d}$
with $\sum_{i=1}^{k}M_{i}=\mathbb{C}^{d}$ s.t. for every $i\in\left\{ 1,\ldots,k\right\} $,
$\dim M_{i}=2$ and $M_{i}\nsubseteq\sum_{j\neq i}M_{j}$. Then, $M$
is a tractable union of linear subspaces of $\mathbb{C}^{d}$.
\end{thm}

\begin{proof}
In one hand, if $\cap_{i=1}^{k}M_{i}\neq\left\{ 0\right\} $ then
set $E:=\cap_{i=1}^{k}M_{i}$ and note that, in this case, for every
$i,j\in\left\{ 1,\ldots,k\right\} $, if $i\neq j$ then $M_{i}\cap M_{j}=E$.
On the other hand, if $\cap_{i=1}^{k}M_{i}=\left\{ 0\right\} $ then
the assertion follows from Lemma \ref{Lemma 3.8} and Theroem \ref{Theorem 4.9}
below. 
\end{proof}
\begin{thm}
\label{Theorem 4.9}Let $d,k\in\mathbb{N}$ s.t. $d,k\geq2$. Let
$M=\cup_{i=1}^{k}M_{i}$ be a union of subspaces of $\mathbb{C}^{d}$
s.t. $\sum_{i=1}^{k}M_{i}=\mathbb{C}^{d}$. Assume that for every
$i\in\left\{ 1,\ldots,k\right\} $, $M_{i}\nsubseteq\sum_{j\neq i}M_{j}$
and that $\cap_{r=1}^{k}M_{i}=\left\{ 0\right\} $. Then, $M=\cup_{i=1}^{m}N_{i}$,
for some $m\in\mathbb{N}$, s.t. for each $i\in\left\{ 1,\ldots,m\right\} $,
\end{thm}

\begin{enumerate}
\item $N_{i}=\cup_{r=1}^{k_{i}}M_{i_{r}}$ is a sub-union of $M$
\item $\cap_{r=1}^{k_{i}}M_{i_{r}}\neq\left\{ 0\right\} \cdot$
\item $\dim\left(\textrm{span}N_{i}\right)\leq d-1$. 
\item For every $j\in\left\{ 1,\ldots,m\right\} $, if $j\neq i$ then $\left(\cap_{r=1}^{k_{i}}M_{i_{r}}\right)\cap\left(\cap_{r=1}^{k_{j}}M_{j_{r}}\right)=\left\{ 0\right\} \cdot$
\end{enumerate}
\begin{proof}
We will prove the assertion by induction on $d\geq2$: If $d=2$ and
$k\geq2$ then for every $i\in\left\{ 1,\ldots,k\right\} $, $\dim M_{i}=1$.
Therefore, in this case, for every $i,j\in\left\{ 1,\ldots,k\right\} $,
if $i\neq j$ then $M_{i}\cap M_{j}=\left\{ 0\right\} \cdot$ Now,
assume that the assertion holds for all $d=2,\ldots,n$, for some
$n\geq2$. We will prove that the assertion holds for $d=n+1$: Let
$M=\cup_{i=1}^{k}M_{i}$ be a union of subspaces of $\mathbb{C}^{n+1}$
s.t. $\sum_{i=1}^{k}M_{i}=\mathbb{C}^{n+1}$. Assume that for every
$i\in\left\{ 1,\ldots,k\right\} $, $M_{i}\nsubseteq\sum_{j\neq i}M_{j}$
and that $\cap_{i=1}^{k}M_{i}=\left\{ 0\right\} $. For convenience,
assume that $\cap_{i=1}^{s}M_{i}\neq\left\{ 0\right\} $, where 
\[
s:=\max\left\{ t\in\mathbb{N}:\left\{ i_{1},\ldots,i_{t}\right\} \subseteq\left\{ 1,\ldots,k\right\} \text{s.t. }i_{1}<i_{2}<\ldots<i_{t}\text{ and }\cap_{l=1}^{t}M_{i_{l}}\neq\left\{ 0\right\} \right\} \cdot
\]
 Set $E_{1}:=$ $\cap_{i=1}^{s}M_{i}$ and $N_{1}:=\cup_{i=1}^{s}M_{i}$.
Since $E_{1}\neq\left\{ 0\right\} $, it follows that $\dim\left(\textrm{span}N_{1}\right)\leq d$.
Note that $\dim\sum_{i=s+1}^{k}M_{i}\leq n$. Therefore, if $\cap_{i=s+1}^{k}M_{i}\neq\left\{ 0\right\} $
then the proof is finished. Else, if $\cap_{i=s+1}^{k}M_{i}=\left\{ 0\right\} $
then, by the induction assumption, $\cup_{i=s+1}^{k}M_{i}=\cup_{i=2}^{m}N_{i}$,
for some $m\in\mathbb{N}$, s.t. for each $i\in\left\{ 2,\ldots,m\right\} $,
\begin{enumerate}
\item $N_{i}=\cup_{r=1}^{k_{i}}M_{i_{r}}$ is a sub-union of $\cup_{i=s+1}^{k+1}M_{i}$.
\item $\cap_{r=1}^{k_{i}}M_{i_{r}}\neq\left\{ 0\right\} \cdot$
\item $\dim\left(\textrm{span}N_{i}\right)\leq n$. 
\item For every $j\in\left\{ 2,\ldots,m\right\} $, if $j\neq i$ then $\left(\cap_{r=1}^{k_{i}}M_{i_{r}}\right)\cap\left(\cap_{r=1}^{k_{j}}M_{j_{r}}\right)=\left\{ 0\right\} \cdot$
\end{enumerate}
Therefore, combining the above and the definition of $s$ yields that
the assertion holds for $d=n+1$.
\end{proof}
\newpage{}

\section{Stability of RKHSs on Tractable Homogeneous Varieties Under Linear
Deformations \label{sec:Linear-deformations-of}}

Let $V$ and $W$ be two biholomorphic homogeneous varieties in $B_{d}$.
Recall that, in this case, there exists an invertible linear map 
\[
A:\mathbb{C^{\mathit{d}}}\longrightarrow\mathbb{C^{\mathit{d}}}
\]
 that maps $V$ isometrically onto $W$. In this case, $A$ induces
a RKHS isomorphism $T_{A}:H_{V}\longrightarrow H_{W}$ such that 
\[
T_{A}\left(k_{x}\right)=k_{Ax}\quad\text{for all \ensuremath{x\in V}}\cdot
\]
By Theorem \ref{Theorem 36}, it follows that for any element $f\in H_{V}$,
there exists a unique $\oplus_{n=0}^{\infty}\xi_{n}\in\oplus_{n=0}^{\infty}V^{n}$
such that 
\[
f=\sum_{n=0}^{\infty}\left\langle \cdot,\xi_{n}\right\rangle \cdot
\]
By the definition of $T_{A}$, it follows that for every $x\in V$,
\[
T_{A}\left(k_{x}\right)=k_{Ax}=k_{x}\circ A^{*}\cdot
\]
 Since $H_{V}=\overline{\text{span}\left\{ k_{x}:x\in V\right\} }$
and $T_{A}$ is bounded on $H_{V}$, it follows that for every $f=\sum_{n=0}^{\infty}\left\langle \cdot,\xi_{n}\right\rangle \in H_{V}$,
\[
T_{A}\left(f\right)=f\circ A^{*}\cdot
\]
In fact, for every $y\in V$, we have
\[
T_{A}\left(f\right)\left(y\right)=f\left(A^{*}y\right)=\sum_{n=0}^{\infty}\left\langle \left(A^{*}y\right)^{n},\xi_{n}\right\rangle =\sum_{n=0}^{\infty}\left\langle y^{n},A^{\otimes n}\xi_{n}\right\rangle \cdot
\]
 Therefore, 
\begin{equation}
T_{A}\left(f\right)=T_{A}\left(\sum_{n=0}^{\infty}\left\langle \cdot,\xi_{n}\right\rangle \right)=\sum_{n=0}^{\infty}\left\langle \cdot,A^{\otimes n}\xi_{n}\right\rangle \cdot\label{eq:The induced RKHS isomorphism}
\end{equation}
Note that from \ref{eq:The induced RKHS isomorphism} we deduce, in
particular, that for every $n\in\mathbb{N}$ and $\xi_{n}\in V^{n}$,
\[
T_{A}\left(\left\langle \cdot,\xi_{n}\right\rangle \right)=\left\langle \cdot,A^{\otimes n}\xi_{n}\right\rangle \cdot
\]
 It follows that 
\[
\left\Vert A^{\otimes n}\mid_{V^{n}}\right\Vert \leq\left\Vert T_{A}\right\Vert \cdot
\]
In particular, we obtain the following theorem:
\begin{thm}
Let $V$ be a homogeneous variety in $B_{d}$. Assume that $A:\mathbb{C^{\mathit{d}}}\longrightarrow\mathbb{C^{\mathit{d}}}$
is an invertible linear map which acts isometrically on $V$ and set
$W:=AV$. Let 
\[
T_{A}:H_{V}\longrightarrow H_{W}
\]
be the RKHS isomorphism that is detemined by
\[
T_{A}\left(k_{x}\right)=k_{Ax}\quad\text{for all \ensuremath{x\in V}}\cdot
\]
 Then, for every $n\in\mathbb{N}$, $\left\Vert A^{\otimes n}\mid_{V^{n}}\right\Vert \leq\left\Vert T_{A}\right\Vert \cdot$
In particular,
\[
\left\Vert A\mid_{\text{span}V}\right\Vert _{OP}\leq\left\Vert T_{A}\right\Vert \cdot
\]
\end{thm}

In this section we mainly show that, if $V$ is a tractable homogeneous
variety (See Definition \ref{Definition 1.5}) and $A$ is almost
a unitary, then $T_{A}$ is almost a unitary in the following sense
(see Theorem \ref{Theorem 5.14} below):

If $\left(A_{n}:\mathbb{C^{\mathit{d}}}\longrightarrow\mathbb{C^{\mathit{d}}}\right)_{n\in\mathbb{N}}$
is a sequence of invertible linear transformations which are isometric
on $V$, then 
\[
\mathit{\parallel\mathit{A_{n}}\parallel_{OP},\parallel A_{n}^{-1}\parallel_{OP}\overset{n\rightarrow\infty}{\longrightarrow}1}
\]
 implies 
\[
\left\Vert T_{A_{n}}\right\Vert \left\Vert T_{A_{n}}^{-1}\right\Vert \stackrel{n\rightarrow\infty}{\longrightarrow}1\cdot
\]
 Where for every $n\in\mathbb{N}$, $T_{A_{n}}:H_{V}\longrightarrow H_{A_{n}V}$
is the RKHS isomorphism induced by $A_{n}$. This establishes the
implication $\left(1\right)\Longrightarrow\left(2\right)$ in Theorem
\ref{Theorem 1.39}.
\begin{prop}
Let $V$ and $W$ be two homogeneuous varieties in $B_{d}$ s.t. $V\perp W$,
i.e. $\left\langle x,y\right\rangle =0$ for all $x\in V$and $y\in W$.
Let $A:\mathbb{C}^{d}\longrightarrow\mathbb{C}^{d}$ be an invertible
linear transformation, which is isometric on $V$and $W$ s.t. $AV\perp AW$.
Then, 
\[
\left\Vert \left(T_{A}\right)\mid_{H_{\left(V\cup W\right)}}\right\Vert =\max\left\{ \left\Vert \left(T_{A}\right)\mid_{H_{V}}\right\Vert ,\left\Vert \left(T_{A}\right)\mid_{H_{W}}\right\Vert \right\} \cdot
\]
\end{prop}

\begin{proof}
Set $H_{1}:=H_{V}$ , $H_{2}:=H_{W}\ominus H_{\left\{ 0\right\} }$.
Then, $H_{1}\perp H_{2}$ and $T_{A}\left(H_{1}\right)\perp T_{A}\left(H_{2}\right)$.
By Proposition \ref{Proposition 310}, It follows that 
\[
H_{V\cup W}=H_{V}+H_{W}=H_{1}\oplus H_{2}.
\]
 Let $g_{1}\in H_{1}$ and $g_{2}\in H_{2}$ . Then, 
\begin{eqnarray*}
\left\Vert T_{A}\left(g_{1}+g_{2}\right)\right\Vert ^{2} & = & \left\Vert T_{A}\left(g_{1}\right)\right\Vert ^{2}+\left\Vert T_{A}\left(g_{2}\right)\right\Vert ^{2}\\
 & \leq & \left(\left\Vert g_{1}+g_{2}\right\Vert ^{2}\right)\max\left\{ \left\Vert \left(T_{A}\right)_{\mid H_{V}}\right\Vert ^{2},\left\Vert \left(T_{A}\right)_{\mid H_{W}}\right\Vert ^{2}\right\} \cdot
\end{eqnarray*}
\end{proof}
\begin{prop}
\label{Proposition 5.2}Let $V$ and $W$ be finite unions of linear
subspaces of $\mathbb{C}^{\mathit{d}}$. Then, 
\[
\mathit{H}_{B_{d}\cap\left(V+W\right)}=\overline{\textrm{span}\left\{ \left\langle \cdot,x\right\rangle ^{m}\left\langle \cdot,y\right\rangle ^{n}:x\in V,y\in W,m,n\in\mathbb{N}_{0}\right\} }\cdot
\]
\end{prop}

\begin{proof}
Recall that $\mathcal{\mathit{H}}_{B_{d}\cap\left(V+W\right)}=\sum_{n=0}^{\infty}\textrm{span}\left\{ \left\langle \cdot,x+y\right\rangle ^{n}:x\in V,y\in W\right\} \cdot$
Now, let $x\in V$, $y\in W$ and $n,m\in\mathbb{N_{\mathit{0}}}$.
Then, 
\[
\left\langle \cdot,x+e^{-it}y\right\rangle ^{n+m}=\sum_{k=0}^{n+m}e^{itk}\binom{n+m}{k}\left\langle \cdot,x\right\rangle ^{n+m-k}\left\langle \cdot,y\right\rangle ^{k},\text{for all }t\in\left[0,2\pi\right]\cdot
\]
 It follows that
\[
\int_{0}^{2\pi}\left\langle \cdot,x+e^{-it}y\right\rangle ^{n+m}e^{-int}dt=2\pi\binom{n+m}{n}\left\langle \cdot,x\right\rangle ^{m}\left\langle \cdot,y\right\rangle ^{n}\cdot
\]
 Therefore, 
\[
\left\langle \cdot,x\right\rangle ^{m}\left\langle \cdot,y\right\rangle ^{n}\in\textrm{span}\left\{ \left\langle \cdot,x+y\right\rangle ^{n+m}:x\in V,y\in W\right\} \cdot
\]
\end{proof}
\begin{prop}
\label{Proposition 5.3}Let $V$ and $W$ be finite unions of linear
subspaces of $\mathbb{C}^{\mathit{d}}$. Let $A:\mathbb{C}^{d}\longrightarrow\mathbb{C}^{d}$
be an invertible linear transformation, which is isometric on $V+W$.
For any $m\in\mathbb{N}_{0}$, let 
\[
H\left(m\right):=\left[H_{V\cap B_{d}}\cdot\textrm{span}\left\{ \left\langle \cdot,y\right\rangle ^{m}:y\in W\right\} \right]\cdot
\]
Then, for every $m\in\mathbb{N}_{0}$,
\end{prop}

\[
\left\Vert \left(T_{A}\right)\mid_{H\left(m\right)}\right\Vert \leq\left\Vert \left(T_{A}\right)\mid_{H_{V\cap B_{d}}}\right\Vert \left\Vert \left(T_{A}\right)\mid_{H_{W\cap B_{d}}}\right\Vert 
\]
and
\[
T_{A}\left(H\left(m\right)\right)=\left[H_{AV\cap B_{d}}\cdot\textrm{span}\left\{ \left\langle \cdot,Ay\right\rangle ^{m}:y\in W\right\} \right]\cdot
\]

\begin{proof}
Let $m\in\mathbb{N}_{0}$, $f\in H_{B_{d}\cap V}$ and $g=\left\langle \cdot,y\right\rangle ^{m}$,
where $y\in W$. Then, 
\[
J\left(fg\right)=P_{s}\left(J\left(f\right)\otimes J\left(g\right)\right),
\]
where $P_{s}$ is the orthogonal projection of  $\mathcal{F}\left(\mathbb{C}^{d}\right)$
onto $\mathcal{F}_{s}\left(\mathbb{C}^{d}\right)$ and $J:\mathcal{H}_{d}^{2}\longrightarrow F_{s}\left(\mathbb{C}^{d}\right)$
is the anti-unitary operator defined in \ref{Equation 2.1.1}. Moreover,
we have 
\[
T_{A}\left(fg\right)=J^{-1}\circ\left(\widetilde{A}\otimes A^{\otimes m}\right)\circ J\left(fg\right),
\]
where $\widetilde{A}=\oplus_{m=0}^{\infty}A^{\otimes m}$. Therefore,
\[
J\left(H\left(m\right)\right)=P_{s}\left(J\left(H_{V}\right)\otimes W^{m}\right)
\]
 and for every $h\in H\left(m\right)$,
\[
T_{A}\left(h\right)=J^{-1}\circ\left(\widetilde{A}\otimes A^{\otimes m}\right)\circ J\left(h\right)\cdot
\]
 Therefore, 
\begin{eqnarray*}
\left\Vert \left(T_{A}\right)\mid_{H\left(m\right)}\right\Vert  & = & \left\Vert \left(\widetilde{A}\otimes A^{\otimes m}\right)\mid_{J\left(H\left(m\right)\right)}\right\Vert \\
 & \leq & \left\Vert \left(\widetilde{A}\otimes A^{\otimes m}\right)\mid_{J\left(H_{V}\right)\otimes W^{m}}\right\Vert \\
 & = & \left\Vert \widetilde{A}\mid_{J\left(H_{V}\right)}\right\Vert \left\Vert \left(A^{\otimes m}\right)\mid_{W^{m}}\right\Vert \\
 & = & \left\Vert \left(T_{A}\right)\mid_{H_{V}}\right\Vert \left\Vert \left(T_{A}\right)\mid_{J^{-1}\left(W^{m}\right)}\right\Vert \\
 & \leq & \left\Vert \left(T_{A}\right)\mid_{H_{V}}\right\Vert \left\Vert \left(T_{A}\right)\mid_{H_{W}}\right\Vert \cdot
\end{eqnarray*}
 Finally, note that 
\begin{eqnarray*}
T_{A}\left(fg\right) & = & J^{-1}\circ\left(\widetilde{A}\otimes A^{\otimes m}\right)\circ P_{s}\left(J\left(f\right)\otimes J\left(g\right)\right)\\
 & = & J^{-1}\circ P_{s}\circ\left(\widetilde{A}\otimes A^{\otimes m}\right)\circ\left(J\left(f\right)\otimes J\left(g\right)\right)\\
 & = & T_{A}\left(f\right)T_{A}\left(g\right)\cdot
\end{eqnarray*}
 It follows that $T_{A}\left(H\left(m\right)\right)=\left[H_{AV\cap B_{d}}\cdot\textrm{span}\left\{ \left\langle \cdot,Ay\right\rangle ^{m}:y\in W\right\} \right]$. 
\end{proof}
\begin{cor}
\label{Corollary 5.4}Let $V$ and $W$ be finite unions of linear
subspaces of $\mathbb{C}^{\mathit{d}}$, s.t. $V\perp W$. Let $A:\mathbb{C}^{d}\longrightarrow\mathbb{C}^{d}$
be an invertible linear transformation, which is isometric on $V$and
$W$, s.t. $AV\perp AW$. Then, 
\[
\left\Vert \left(T_{A}\right)\mid_{H_{\left(V+W\right)\cap B_{d}}}\right\Vert \leq\left\Vert \left(T_{A}\right)\mid_{H_{V\cap B_{d}}}\right\Vert \left\Vert \left(T_{A}\right)_{\mid H_{W\cap B_{d}}}\right\Vert \cdot
\]
\end{cor}

\begin{proof}
Since $V\perp W$, it follows that 
\[
\left[H_{V\cap B_{d}}\cdot\textrm{span}\left\{ \left\langle \cdot,y\right\rangle ^{n}:y\in W\right\} \right]
\]
 and 
\[
\left[H_{V\cap B_{d}}\cdot\textrm{span}\left\{ \left\langle \cdot,y\right\rangle ^{m}:y\in W\right\} \right]
\]
 are orthogonal for $m\neq n$. Now, by Proposition \ref{Proposition 5.2},
it follows that 
\[
\mathcal{\mathit{H}}_{B_{d}\cap\left(V+W\right)}=\sum_{n=0}^{\infty}\left[H_{V\cap B_{d}}\cdot\textrm{span}\left\{ \left\langle \cdot,y\right\rangle ^{n}:y\in W\right\} \right]\cdot
\]
Now, the assertion follows from Proposition \ref{Proposition 5.3}. 
\end{proof}
\begin{cor}
(See \cite[Lemma 7.12]{davidson2011isomorphism}). \label{Corollary 5.6}Let
$V$ be a finite union of linear subspaces of $\mathbb{C}^{\mathit{d}}$
and let $E$ be a subspace of $\mathbb{C}^{\mathit{d}}$ s.t. $V\perp E$.
Let $A:\mathbb{C}^{d}\longrightarrow\mathbb{C}^{d}$ be an invertible
linear transformation, which is isometric on $V$and $E$ s.t. $AV\perp AE$.
Then, 
\[
\left\Vert \left(T_{A}\right)\mid_{H_{\left(V+E\right)\cap B_{d}}}\right\Vert =\left\Vert \left(T_{A}\right)\mid_{H_{V\cap B_{d}}}\right\Vert \cdot
\]
\end{cor}

\begin{proof}
This follows from Corollary \ref{Corollary 5.4}. Note that 
\[
\left\Vert \left(T_{A}\right)\mid_{H_{V\cap B_{d}}}\right\Vert \leq\left\Vert \left(T_{A}\right)\mid_{H_{\left(V+E\right)\cap B_{d}}}\right\Vert \leq\left\Vert \left(T_{A}\right)\mid_{H_{V\cap B_{d}}}\right\Vert \left\Vert \left(T_{A}\right)\mid_{E\cap B_{d}}\right\Vert 
\]
 and that $\left\Vert \left(T_{A}\right)\mid_{H_{E\cap B_{d}}}\right\Vert =1$.
\end{proof}
\begin{claim}
\label{Claim 61}Let $A:\mathbb{C^{\mathit{d}}}\longrightarrow\mathbb{C^{\mathit{d}}}$
be an invertible linear map. Then,
\end{claim}

\begin{enumerate}
\item $\left\Vert A^{*}A-I\right\Vert =\max\left\{ \left|\left\Vert A\right\Vert ^{2}-1\right|,\left|1-\frac{1}{\left\Vert A^{-1}\right\Vert ^{2}}\right|\right\} $.
\item If $\left\Vert \mathit{A}\right\Vert \geq1$ and $\left\Vert A^{-1}\right\Vert \geq1$
then, 
\begin{gather*}
\max\left\{ \left\Vert \mathit{A}\right\Vert ^{2}-1,\left\Vert A^{-1}\right\Vert ^{2}-1\right\} \leq\max\left\{ \left\Vert A^{*}A-I\right\Vert ,\left\Vert \left(A^{-1}\right)^{*}A^{-1}-I\right\Vert \right\} \leq\left\Vert A\right\Vert ^{2}\left\Vert A^{-1}\right\Vert ^{2}-1.
\end{gather*}
\end{enumerate}
\begin{proof}
Let $A:\mathbb{C^{\mathit{d}}}\longrightarrow\mathbb{C^{\mathit{d}}}$
be an invertible linear map.
\begin{enumerate}
\item Note that $\sigma_{B\left(\mathbb{C^{\mathit{d}}}\right)}\left(A^{*}A-I\right)=\sigma_{B\left(\mathbb{C^{\mathit{d}}}\right)}\left(A^{*}A\right)-1\subseteq\mathbb{R}.$
Therefore, 
\[
\left\Vert A^{*}A-I\right\Vert =\max\left\{ \left|\lambda-1\right|:\lambda\in\sigma_{B\left(\mathbb{C^{\mathit{d}}}\right)}\left(A^{*}A\right)\right\} .
\]
Since, 
\[
\max\left\{ \lambda:\lambda\in\sigma_{B\left(\mathbb{C^{\mathit{d}}}\right)}\left(A^{*}A\right)\right\} =\left\Vert A\right\Vert ^{2}
\]
and 
\[
\min\left\{ \lambda:\lambda\in\sigma_{B\left(\mathbb{C^{\mathit{d}}}\right)}\left(A^{*}A\right)\right\} =\frac{1}{\left\Vert A^{-1}\right\Vert ^{2}},
\]
 It follows that $\left\Vert A^{*}A-I\right\Vert =\max\left\{ \left|\left\Vert A\right\Vert ^{2}-1\right|,\left|1-\frac{1}{\left\Vert A^{-1}\right\Vert ^{2}}\right|\right\} $.
\item This follows from $\left(1\right)$. 
\end{enumerate}
\end{proof}
For a closed subspace $S$ of a Hilbert space $\mathcal{H}$, let
$P_{S}$ denote the orthogonal projection of $\mathcal{H}$ onto $S$.
\begin{claim}
\label{Claim 5.8}Let $\mathit{V}$ and $\mathit{W}$ be linear subspaces
of $\mathbb{C^{\mathit{d}}}$. Then,
\end{claim}

\[
V\cap W\neq\left\{ 0\right\} \Longleftrightarrow\left\Vert P_{V}P_{W}\right\Vert =1
\]

\begin{proof}
First, note that $\left\Vert P_{V}P_{W}\right\Vert \leq1$. Second,
if $V\cap W\neq\left\{ 0\right\} $, then choose $x\in V\cap W$ with
$\left\Vert x\right\Vert =1$. Then, $\left\Vert P_{V}P_{W}x\right\Vert =\left\Vert x\right\Vert =1$.
Hence, $\left\Vert P_{V}P_{W}\right\Vert =1$. Conversely, if $\left\Vert P_{V}P_{W}\right\Vert =1$
then there exists $x\in\mathbb{C^{\mathit{d}}}$ with $\left\Vert x\right\Vert =1$,
s.t. $\left\Vert P_{V}P_{W}x\right\Vert =1$. It follows that $\left\Vert P_{W}x\right\Vert =1$.
Using the fact that
\[
\left\Vert P_{W}x\right\Vert ^{2}=\left\Vert P_{V}P_{W}x\right\Vert ^{2}+\left\Vert \left(I-P_{V}\right)P_{W}x\right\Vert ^{2}
\]
 yields $\left\Vert \left(I-P_{V}\right)P_{W}x\right\Vert =0$. Therefore,
$P_{W}x\in\mathit{V}$. Again, using the fact that 
\[
\left\Vert x\right\Vert ^{2}=\left\Vert P_{W}x\right\Vert ^{2}+\left\Vert \left(I-P_{W}\right)x\right\Vert ^{2}
\]
 yields $\left\Vert \left(I-P_{W}\right)x\right\Vert =0$. Therefore,
$P_{W}x=x\in\mathit{W}\cap V$. Hence, $V\cap W\neq\left\{ 0\right\} $.
\end{proof}
\begin{lem}
\label{Lemma 5.10}Let $\mathit{V}$ and $\mathit{W}$ be linear subspaces
of $\mathbb{C^{\mathit{d}}}$. Let $A:\mathbb{C^{\mathit{d}}}\longrightarrow\mathbb{C^{\mathit{d}}}$
be an invertible linear map. Then,
\end{lem}

\begin{enumerate}
\item $\left\Vert \mathit{P_{AV}P_{AW}}\right\Vert \leq\left\Vert A^{-1}\right\Vert ^{2}\left(\left\Vert A^{*}A-I\right\Vert +\left\Vert P_{V}P_{W}\right\Vert \right)$.
\item $\left\Vert \mathit{P_{V}P_{W}}\right\Vert \leq\left\Vert \mathit{A}\right\Vert ^{2}\left(\left\Vert \left(A^{-1}\right)^{*}A^{-1}-\mathit{I}\right\Vert +\left\Vert P_{AV}P_{AW}\right\Vert \right)$.
\item If $\mathit{A}$ is a unitary then $\left\Vert \mathit{P_{AV}P_{AW}}\right\Vert =\left\Vert P_{V}P_{W}\right\Vert $.
\end{enumerate}
\begin{proof}
let $A:\mathbb{C^{\mathit{d}}}\longrightarrow\mathbb{C^{\mathit{d}}}$
be an invertible linear map.
\begin{enumerate}
\item Note that 
\[
P_{AV}=AP_{V}A^{-1}P_{AV}
\]
 and 
\[
P_{AW}=P_{AW}\left(A^{-1}\right)^{*}P_{W}A^{*}\cdot
\]
It follows that 
\begin{eqnarray*}
\left\Vert P_{AV}P_{AW}\right\Vert =\left\Vert P_{AW}P_{AV}\right\Vert  & = & \left\Vert P_{AW}\left(A^{-1}\right)^{*}P_{W}A^{*}AP_{V}A^{-1}P_{AV}\right\Vert \\
 & \leq & \left\Vert A^{-1}\right\Vert ^{2}\left\Vert P_{W}A^{*}AP_{V}\right\Vert \\
 & = & \left\Vert A^{-1}\right\Vert ^{2}\left\Vert P_{V}A^{*}AP_{W}\right\Vert \\
 & \leq & \left\Vert A^{-1}\right\Vert ^{2}\left(\left\Vert A^{*}A-I\right\Vert +\left\Vert P_{V}P_{W}\right\Vert \right).
\end{eqnarray*}
 
\item Note that $V=A^{-1}\left(AV\right)$ and $W=A^{-1}\left(AW\right)$.
Therefore, by $\left(3\right)$ we have 
\[
\left\Vert \mathit{P_{V}P_{W}}\right\Vert \leq\left\Vert \mathit{A}\right\Vert ^{2}\left(\left\Vert \left(A^{-1}\right)^{*}A^{-1}-I\right\Vert +\left\Vert P_{AV}P_{AW}\right\Vert \right).
\]
 
\item Suppose that $\mathit{A}$ is a unitary. Then, $\left(3\right)$ yields
$\left\Vert \mathit{P_{AV}P_{AW}}\right\Vert \leq\left\Vert P_{V}P_{W}\right\Vert $,
and $\left(4\right)$ yields 
\[
\left\Vert \mathit{P_{V}P_{W}}\right\Vert \leq\left\Vert P_{AV}P_{AW}\right\Vert \cdot
\]
\end{enumerate}
\end{proof}
\begin{thm}
\label{Theorem 5.10}Let $k\in\mathbb{N}$ s.t. $k\geq2$. Suppose
that for all $i\in\left\{ 1,\ldots,k\right\} $, $M_{i}$ is a subspace
of $\mathbb{C^{\mathit{d}}}$. Let $\left(A_{n}\right)_{n\in\mathbb{N}_{0}}$
be a sequence of invertible linear transformations, $A_{n}:\mathbb{C^{\mathit{d}}}\longrightarrow\mathbb{C^{\mathit{d}}}$,
s.t. $A_{0}={\bf Id}$ and $\mathit{\parallel\mathit{A_{n}}\parallel,\parallel A_{n}^{-1}\parallel\overset{n\rightarrow\infty}{\longrightarrow}1}$.
For every $n\in\mathbb{N}_{0}$, set 
\[
c_{n}:=\max\left\{ \left\Vert P_{A_{n}M_{i}}P_{A_{n}M_{j}}\right\Vert :i,j\in\left\{ 1,\ldots,k\right\} ,i\neq j\right\} .
\]
 Then, 
\[
c_{n}\overset{n\rightarrow\infty}{\longrightarrow}c_{0}\cdot
\]
\begin{proof}
Note that for every $n\in\mathbb{N}_{0}$, $0\leq c_{n}\leq1$. Let
$n\in\mathbb{N}$. Then, by Lemma \ref{Lemma 5.10},
\[
c_{n}\leq\left\Vert \mathit{A_{n}^{-1}}\right\Vert ^{2}\left(\left\Vert A_{n}^{*}A_{n}-I\right\Vert +c_{0}\right)
\]
 and 
\[
c_{0}\leq\left\Vert \mathit{A}_{n}\right\Vert ^{2}\left(\left\Vert \left(A_{n}^{-1}\right)^{*}A_{n}^{-1}-\mathit{I}\right\Vert +c_{n}\right).
\]
Since, by Claim \ref{Claim 61}, also $\left\Vert A_{n}^{*}A_{n}-I\right\Vert \underset{n\rightarrow\infty}{\longrightarrow}0$
and $\left\Vert \left(A_{n}^{-1}\right)^{*}A_{n}^{-1}-I\right\Vert \underset{n\rightarrow\infty}{\longrightarrow}0$,
it follows that 
\[
c_{0}\leq\liminf_{n}c_{n}\leq\limsup_{n}c_{n}\leq c_{0}\cdot
\]
\end{proof}
\end{thm}

\begin{lem}
\label{Lemma 5.11}(See \cite[Lemma 7.10]{davidson2011isomorphism}).
Let $k\in\mathbb{N}$ s.t. $k\geq2$ and let $M=\cup_{i=1}^{k}M_{i}$
be a union of subspaces of $\mathbb{C}^{d}$ s.t. $\sum_{i=1}^{k}M_{i}=\mathbb{C}^{d}$.
Assume that for every $i,j\in\left\{ 1,\ldots,k\right\} $, 
\[
i\neq j\Longrightarrow M_{i}\cap M_{j}=\left\{ 0\right\} \cdot
\]
 Set 
\[
c:=\max\left\{ \left\Vert P_{M_{i}}P_{M_{j}}\right\Vert :i,j\in\left\{ 1,\ldots,k\right\} ,i\neq j\right\} .
\]
Let $n\in\mathbb{N}$ and $\mathit{x=\text{\ensuremath{\sum_{i=1}^{k}x_{i}\in M^{n}=\sum_{i=1}^{k}M_{i}^{n}}}}$,
where for every $i\in\left\{ 1,\ldots,k\right\} $, $x_{i}\in M_{i}^{n}$.
Then, 
\[
\left(1-c^{n}k\right)\sum_{i=1}^{k}\left\Vert x_{i}\right\Vert ^{2}\leq\left\Vert x\right\Vert ^{2}\leq\left(1+c^{n}k\right)\sum_{i=1}^{k}\left\Vert x_{i}\right\Vert ^{2}\cdot
\]
 In particular, if $c^{n}k<1$ then 
\[
\sum_{i=1}^{k}\left\Vert x_{i}\right\Vert ^{2}\leq\frac{\left\Vert x\right\Vert ^{2}}{1-kc^{n}}\cdot
\]
\end{lem}

\begin{proof}
First, note that, by Claim \ref{Claim 5.8}, $c<1$. Now, let $n\in\mathbb{N}$
and $\mathit{x=\text{\ensuremath{\sum_{i=1}^{k}x_{i}\in M^{n}=\sum_{i=1}^{k}M_{i}^{n}}}}$,
where for every $i\in\left\{ 1,\ldots,k\right\} $, $x_{i}\in M_{i}^{n}$.
Then, 
\begin{eqnarray*}
\left|\left\Vert x\right\Vert ^{2}-\sum_{i=1}^{k}\left\Vert x_{i}\right\Vert ^{2}\right| & = & \left|\sum_{i\neq j}\left\langle x_{i},x_{j}\right\rangle \right|\leq\sum_{i\neq j}\left|\left\langle x_{i},x_{j}\right\rangle \right|\\
 & = & \sum_{i\neq j}\left|\left\langle P_{M_{j}^{n}}P_{M_{i}^{n}}x_{i},x_{j}\right\rangle \right|\\
 & \leq & \sum_{i\neq j}^{k}\left\Vert P_{M_{j}^{n}}P_{M_{i}^{n}}\right\Vert \left\Vert x_{i}\right\Vert \left\Vert x_{j}\right\Vert \\
 & \leq & \sum_{i\neq j}\left\Vert P_{M_{j}}^{\otimes n}P_{M_{i}}^{\otimes n}\right\Vert \left\Vert x_{i}\right\Vert \left\Vert x_{j}\right\Vert \\
 & = & \sum_{i\neq j}\left\Vert P_{M_{i}}P_{M_{j}}\right\Vert ^{n}\left\Vert x_{i}\right\Vert \left\Vert x_{j}\right\Vert \\
 & \leq & c^{n}\left(\sum_{i=1}^{k}\left\Vert x_{i}\right\Vert \right)^{2}\\
 & \leq & kc^{n}\sum_{i=1}^{k}\left\Vert x_{i}\right\Vert ^{2}\cdot
\end{eqnarray*}
Therefore, 
\[
\left(1-c^{n}k\right)\sum_{i=1}^{k}\left\Vert x_{i}\right\Vert ^{2}\leq\left\Vert x\right\Vert ^{2}\leq\left(1+c^{n}k\right)\sum_{i=1}^{k}\left\Vert x_{i}\right\Vert ^{2}\cdot
\]
Now, assume that $c^{n}k<1$. Then, it follows that $\sum_{i=1}^{k}\left\Vert x_{i}\right\Vert ^{2}\leq\frac{\left\Vert x\right\Vert ^{2}}{1-kc^{n}}\cdot$
\end{proof}
\begin{prop}
\label{Proposition 5.12}Let $d,k\in\mathbb{N}$. Let $N=\cup_{i=1}^{k}\cup_{r=1}^{k_{i}}N_{i,r}$
be a finite union of finite unions of subspaces of $\mathbb{C}^{d}$
s.t $\textrm{span}\left(N\right)=\mathbb{C}^{d}$. For every $i\in\left\{ 1,\ldots,k\right\} $,
set $N_{i}:=\cup_{r=1}^{k_{i}}N_{i,r}$ and suppose that for every
$i,j\in\left\{ 1,\ldots,k\right\} $, 
\[
i\neq j\Longrightarrow\textrm{span}N_{i}\cap\textrm{span}N_{j}=\left\{ 0\right\} \cdot
\]
 Assume that $\left(A_{m}:\mathbb{C}^{d}\longrightarrow\mathbb{C}^{d}\right)_{m\in\mathbb{N}}$
is a sequence of invertible linear transformations, which are isometric
on $N$. Set $V:=N\cap B_{d}$ and for every $m\in\mathbb{N}$, $W_{m}:=B_{d}\cap A_{m}N$.
Further, for every $m\in\mathbb{N}$, let $T_{A_{m}}:H_{V}\rightarrow H_{W_{m}}$
be the RKHS isomorphism which induced by $A_{m}$. Then, there exist
$s\in\mathbb{R}$, $0\leq s<1$, and $N\left(s,k\right)\in\mathbb{N}$
s.t. for every $m\in\mathbb{N}$ and $N\geq N\left(s,k\right)$,
\begin{enumerate}
\item $\left\Vert T_{A_{m}}\right\Vert \leq max\left\{ \left\Vert A_{m}\right\Vert ^{N},\sqrt{\frac{1+s^{N}k}{1-s^{N}k}}\underset{1\leq i\leq k}{\max}\left\Vert T_{A_{m}}\mid_{H_{N_{i}\cap B_{d}}}\right\Vert \right\} \cdot$
\item $\left\Vert T_{A_{m}}^{-1}\right\Vert \leq max\left\{ \left\Vert A_{m}^{-1}\right\Vert ^{N},\sqrt{\frac{1+s^{N}k}{1-s^{N}k}}\underset{1\leq i\leq k}{\max}\left\Vert T_{A_{m}}^{-1}\mid_{H_{A_{m}N_{i}\cap B_{d}}}\right\Vert \right\} \cdot$
\end{enumerate}
\end{prop}

\begin{proof}
We will prove assertion $\left(1\right)$. The proof for $\left(2\right)$
is similar. Set 
\[
c_{0}:=\max\left\{ \left\Vert P_{\textrm{span}N_{i}}P_{\textrm{span}N_{j}}\right\Vert :i,j\in\left\{ 1,\ldots,k\right\} ,i\neq j\right\} 
\]
 and for every $m\in\mathbb{N}$, set
\[
c_{m}:=\max\left\{ \left\Vert P_{A\left(\textrm{span}N_{i}\right)}P_{A\left(\textrm{span}N_{j}\right)}\right\Vert :i,j\in\left\{ 1,\ldots,k\right\} ,i\neq j\right\} \cdot
\]
 Let 
\[
s:=\max\left\{ c_{m}:m\in\mathbb{N\cup\mathrm{\left\{ 0\right\} }}\right\} 
\]
 and 
\[
N\left(s,k\right):=\min\left\{ n\in\mathbb{N}:s^{n}k<1\right\} 
\]
(note that by by Claim \ref{Claim 5.8} and Theorem \ref{Theorem 5.10},
it follows that $s<1$). Let $m\in\mathbb{N}$ and $n\in\mathbb{N}$
s.t. $n\geq N\left(s,k\right)$. Then, for every $\mathit{x=\text{\ensuremath{\sum_{i=1}^{k}x_{i}\in N^{n}=\sum_{i=1}^{k}N_{i}^{n}}}}$,
where for every $i\in\left\{ 1,\ldots,k\right\} $, $x_{i}\in N_{i}^{n}$,
\begin{eqnarray*}
\left|\left\Vert A_{m}^{\otimes n}x\right\Vert ^{2}-\sum_{i=1}^{k}\left\Vert A_{m}^{\otimes n}x_{i}\right\Vert ^{2}\right| & = & \left|\sum_{i\neq j}\left\langle A_{m}^{\otimes n}x_{i},A_{m}^{\otimes n}x_{j}\right\rangle \right|\\
 & \leq & \sum_{i\neq j}\left|\left\langle A_{m}^{\otimes n}x_{i},A_{m}^{\otimes n}x_{j}\right\rangle \right|\\
 & = & \sum_{i\neq j}\left|\left\langle P_{A_{m}^{\otimes n}N_{j}^{n}}P_{A_{m}^{\otimes n}N_{i}^{n}}A_{m}^{\otimes n}x_{i},A_{m}^{\otimes n}x_{j}\right\rangle \right|\\
 & \leq & \sum_{i\neq j}^{k}\left\Vert P_{A_{m}^{\otimes n}N_{j}^{n}}P_{A_{m}^{\otimes n}N_{i}^{n}}\right\Vert \left\Vert A_{m}^{\otimes n}x_{i}\right\Vert \left\Vert A_{m}^{\otimes n}x_{j}\right\Vert \\
 & \leq & \sum_{i\neq j}\left\Vert P_{A_{m}N_{j}}^{\otimes n}P_{A_{m}N_{i}}^{\otimes n}\right\Vert \left\Vert A_{m}^{\otimes n}x_{i}\right\Vert \left\Vert A_{m}^{\otimes n}x_{j}\right\Vert \\
 & \leq & \sum_{i\neq j}\left\Vert P_{A_{m}N_{i}}P_{A_{m}N_{j}}\right\Vert ^{n}\left\Vert A_{m}^{\otimes n}x_{i}\right\Vert \left\Vert A_{m}^{\otimes n}x_{j}\right\Vert \\
 & \leq & \sum_{i\neq j}\left\Vert P_{A\left(\textrm{span}N_{i}\right)}P_{A\left(\textrm{span}N_{j}\right)}\right\Vert ^{n}\left\Vert A_{m}^{\otimes n}x_{i}\right\Vert \left\Vert A_{m}^{\otimes n}x_{j}\right\Vert \\
 & \leq & \left(\underset{i\neq j}{\max}\left\Vert P_{A_{m}\left(\textrm{span}N_{i}\right)}P_{A_{m}\left(\textrm{span}N_{j}\right)}\right\Vert \right)^{n}\left(\sum_{i=1}^{k}\left\Vert A_{m}^{\otimes n}x_{i}\right\Vert \right)^{2}\\
 & \leq & kc_{m}^{n}\sum_{i=1}^{k}\left\Vert A_{m}^{\otimes n}x_{i}\right\Vert ^{2}\\
 & \leq & ks^{n}\sum_{i=1}^{k}\left\Vert A_{m}^{\otimes n}x_{i}\right\Vert ^{2}\cdot
\end{eqnarray*}
Therefore, 
\[
\left|\left\Vert A_{m}^{\otimes n}x\right\Vert ^{2}-\sum_{i=1}^{k}\left\Vert A_{m}^{\otimes n}x_{i}\right\Vert ^{2}\right|\leq ks^{n}\sum_{i=1}^{k}\left\Vert A_{m}^{\otimes n}x_{i}\right\Vert ^{2}\cdot
\]
 It follows that 
\[
\left\Vert A_{m}^{\otimes n}x\right\Vert ^{2}\leq\left(1+ks^{n}\right)\sum_{i=1}^{k}\left\Vert A_{m}^{\otimes n}x_{i}\right\Vert ^{2}\cdot
\]
 Therefore, 
\begin{equation}
\left\Vert A_{m}^{\otimes n}x\right\Vert ^{2}\leq\left(1+ks^{n}\right)\underset{i\in\left\{ 1,\ldots,k\right\} }{\max}\left\Vert T_{A_{m}}\mid_{H_{N_{i}\cap B_{d}}}\right\Vert ^{2}\sum_{i=1}^{k}\left\Vert x_{i}\right\Vert ^{2}\cdot\label{eq: 5.2}
\end{equation}
 By Lemma \ref{Lemma 5.11} and the definition of $s$, we deduce
that
\begin{equation}
\sum_{i=1}^{k}\left\Vert x_{i}\right\Vert ^{2}\leq\frac{\left\Vert x\right\Vert ^{2}}{1-ks^{n}}\cdot\label{eq: 5.3}
\end{equation}
 From Inequalities \ref{eq: 5.2} and \ref{eq: 5.3} it follows that
\[
\left\Vert A_{m}^{\otimes n}x\right\Vert ^{2}\leq\frac{1+ks^{n}}{1-ks^{n}}\underset{i\in\left\{ 1,\ldots,k\right\} }{\max}\left\Vert T_{A_{m}}\mid_{H_{N_{i}\cap B_{d}}}\right\Vert ^{2}\left\Vert x\right\Vert ^{2}\cdot
\]
Recall that $T_{A_{m}}:H_{V}\rightarrow H_{W_{m}}$ is the RKHS isomorphism,
which is defined by 
\[
T_{A_{m}}\left(\sum_{n=0}^{\infty}\left\langle \cdot,\xi_{n}\right\rangle \right)=\sum_{n=0}^{\infty}\left\langle \cdot,A_{m}^{\otimes n}\left(\xi_{n}\right)\right\rangle ,\quad\oplus_{n=0}^{\infty}\xi_{n}\in\oplus_{n=0}^{\infty}N^{n}.
\]
Let $\oplus_{n=0}^{\infty}\xi_{n}\in\oplus_{n=0}^{\infty}N^{n}$.
Then, for every $\mathit{\mathit{N}\geq N\left(s,k\right)}$,
\begin{flalign*}
\left\Vert T_{A_{m}}\left(\sum_{n=0}^{\infty}\left\langle \cdot,\xi_{n}\right\rangle \right)\right\Vert ^{2} & =\left\Vert \sum_{n=0}^{\infty}\left\langle \cdot,A_{m}^{\otimes n}\left(\xi_{n}\right)\right\rangle \right\Vert ^{2}\\
 & \leq\sum_{n=0}^{N}\left\Vert A_{m}\right\Vert ^{2n}\left\Vert \xi_{n}\right\Vert ^{2}+\sum_{n=N+1}^{\infty}\frac{1+s^{n}k}{1-s^{n}k}\underset{i\in\left\{ 1,\ldots,k\right\} }{\max}\left\Vert T_{A_{m}}\mid_{N_{i}\cap B_{d}}\right\Vert ^{2}\left\Vert \xi_{n}\right\Vert ^{2}\\
 & \leq\max\left\{ \left\Vert A_{m}\right\Vert ^{2N},\frac{1+s^{N}k}{1-s^{N}k}\underset{}{\underset{i\in\left\{ 1,\ldots,k\right\} }{\max}\left\Vert T_{A_{m}}\mid_{N_{i}\cap B_{d}}\right\Vert ^{2}}\right\} \left\Vert \sum_{n=0}^{\infty}\left\langle \cdot,\xi_{n}\right\rangle \right\Vert ^{2}\cdot
\end{flalign*}
 Therefore, we obtain that for every $m\in\mathbb{N}$ and $N\geq N\left(s,k\right)$,
\[
\left\Vert T_{A_{m}}\right\Vert \leq\max\left\{ \left\Vert A_{m}\right\Vert ^{N},\sqrt{\frac{1+s^{N}k}{1-s^{N}k}}\underset{}{\underset{i\in\left\{ 1,\ldots,k\right\} }{\max}\left\Vert T_{A_{m}}\mid_{N_{i}\cap B_{d}}\right\Vert }\right\} \cdot
\]
\end{proof}
We are now ready to prove the Implication $\left(1\right)\Longrightarrow\left(2\right)$
in Theorem \ref{Theorem 1.39} (see Theorem \ref{Theorem 5.14} below).
For convenience, we will first state the following claim:
\begin{claim}
\label{Claim: 5.13} Let $d\in\mathbb{N}$ and $V\subseteq B_{d}$
be a tractable homogeneous variety s.t. $\textrm{span}\left(V\right)=\mathbb{C}^{d}$.
Suppose that $V=V_{1}\cup\cdot\cdot\cdot\cup V_{k}$, for some $k\in\mathbb{N}$,
is the decomposition of $V$ into irreducible components. Set $M:=\cup_{i=1}^{k}\textrm{span}\left(V_{i}\right)$.
Then, $V$ satisfies one of the following conditions:
\end{claim}

\begin{enumerate}
\item $M=\mathbb{C}^{d}$.
\item There exist $k\in\mathbb{N}$ s.t. $k\geq2$ and a finite union $N=\cup_{i=1}^{k}\cup_{r=1}^{k_{i}}N_{i,r}$
of finite unions of subspaces of $\mathbb{C}^{d}$ s.t. the following
hold:
\begin{enumerate}
\item $V\subseteq N$.
\item For every $i\in\left\{ 1,\ldots,k\right\} $, $N_{i}:=\cup_{r=1}^{k_{i}}N_{i,r}$
is a tractable union of subspaces of $\textrm{span}N_{i}\cdot$
\item For every $i,j\in\left\{ 1,\ldots,k\right\} $, if $i\neq j$ then
$\left(\textrm{span}N_{i}\right)\cap\left(\textrm{span}N_{j}\right)=\left\{ 0\right\} \cdot$
\item For every $i\in\left\{ 1,\ldots,k\right\} $, $\dim\left(\textrm{span}N_{i}\right)\leq d-1$. 
\item For every invertible linear  transformation $A:\mathbb{C}^{d}\longrightarrow\mathbb{C}^{d}$,
if $A$ is isometric on $V$ then $A$ is isometric on $N$.
\end{enumerate}
\item $V\subseteq E\oplus N$, where $N\subseteq\mathbb{C}^{d}$ is a tractable
union of subspaces of $\textrm{span}N$ and $\left\{ 0\right\} \neq E\subseteq\mathbb{C}^{d}$
is a subspace, which is orthogonal to $\textrm{span}N$ s.t. for every
invertible linear  transformation $A:\mathbb{C}^{d}\longrightarrow\mathbb{C}^{d}$,
if $A$ is isometric on $V$ then $A$ is isometric on $E\oplus N$. 
\end{enumerate}
\begin{proof}
Note that $V\subseteq M$. Therefore, by \cite[Proposition 7.6]{davidson2011isomorphism}),
it follows that for every invertible linear  transformation $A:\mathbb{C}^{d}\longrightarrow\mathbb{C}^{d}$,
$A$ is isometric on $V$ if and only if $A$ is isometric on $M$.
Now, the proof follows from Definitions \ref{Definition 1.4} , \ref{Definition 1.5}
.
\end{proof}
\begin{thm}
\label{Theorem 5.14}Let $d\in\mathbb{N}$. Let $V$ be a tractable
homogeneous variety in $B_{d}$ s.t. $\textrm{span}\left(V\right)=\mathbb{C}^{d}$.
Suppose that for every $n\in\mathbb{N}$, $A_{n}:\mathbb{C}^{d}\longrightarrow\mathbb{C}^{d}$
is an invertible linear transformation which is isometric on $V$and
set $W_{n}:=A_{n}V$. For every $n\in\mathbb{N}$, let $T_{A_{n}}:H_{V}\rightarrow H_{W_{n}}$
be the RKHS isomorphism determined by 
\[
T_{A_{n}}\left(k_{x}\right)=k_{A_{n}x},\;x\in V\cdot
\]
 If 
\[
\mathit{\left\Vert \mathit{A_{n}}\right\Vert {}_{OP},\left\Vert A_{n}^{-1}\right\Vert {}_{OP}\overset{n\rightarrow\infty}{\longrightarrow}1}
\]
 then 
\[
\left\Vert T_{A_{n}}\right\Vert \left\Vert T_{A_{n}}^{-1}\right\Vert \stackrel{n\rightarrow\infty}{\longrightarrow}1\cdot
\]
\end{thm}

\begin{proof}
We will prove the assertion by induction on $d$: If $d=1$ then $V=B_{d}$.
Therefore, for every $n\in\mathbb{N}$, $A_{n}$ is a unitary. It
follows that for every $n\in\mathbb{N}$, $T_{A_{n}}:H_{V}\rightarrow H_{W_{n}}$
is a unitary. So, if $d=1$ then for every $n\in\mathbb{N}$, $\left\Vert T_{A_{n}}\right\Vert =\left\Vert T_{A_{n}}^{-1}\right\Vert =1$.
Hence, $\left\Vert T_{A_{n}}\right\Vert \left\Vert T_{A_{n}}^{-1}\right\Vert \stackrel{n\rightarrow\infty}{\longrightarrow}1$.
Now, let $d_{0}\in\mathbb{N}$ and assume that the assertion holds
for $d=1,\ldots,d_{0}$. We will show that the assertion holds for
$d=d_{0}+1$: Let $V$ be a tractable homogeneous variety in $B_{d}$
with $\textrm{span}\left(V\right)=\mathbb{C}^{d_{0}+1}$. Suppose
that for every $n\in\mathbb{N}$, $A_{n}:\mathbb{C}^{d_{0}+1}\longrightarrow\mathbb{C}^{d_{0}+1}$
is an invertible linear transformation which is isometric on $V$
and set $W_{n}:=A_{n}V$. Assume that $\mathit{\left\Vert \mathit{A_{n}}\right\Vert {}_{OP},\left\Vert A_{n}^{-1}\right\Vert {}_{OP}\overset{n\rightarrow\infty}{\longrightarrow}1}$.
We shall show what is required by considering the $3$ cases for $V$,
which are listed in Claim \ref{Claim: 5.13} :
\begin{itemize}
\item $M=\mathbb{C}^{d}$, i.e., there is $i\in\left\{ 1,\ldots,k\right\} $
s.t. $M_{i}=\mathbb{C}^{d}$ (here $M_{i}$ and $M$ are defined as
in the statement of Claim \ref{Claim: 5.13} ): In this case, for
every $n\in\mathbb{N}$, $A_{n}$ is a unitary (see\cite[Proposition 7.6]{davidson2011isomorphism}).
Therefore, for every $n\in\mathbb{N}$, $T_{A_{n}}$ has an extension
to a unitary $\widetilde{T_{A_{n}}}:\mathcal{H}_{d}^{2}\rightarrow\mathcal{H}_{d}^{2}$
. Since $V\subseteq B_{d}$ and for every $n\in\mathbb{N}$, $W_{n}\subseteq B_{d}$,
it follows that for every $n\in\mathbb{N}$,
\[
1\leq\left\Vert T_{A_{n}}\right\Vert \leq\left\Vert \widetilde{T_{A_{n}}}\right\Vert =1
\]
 and 
\[
1\leq\left\Vert T_{A_{n}}^{-1}\right\Vert \leq\left\Vert \widetilde{T_{A_{n}}}^{-1}\right\Vert =1\cdot
\]
 Therefore, if $V$ is irreducible then $\left\Vert T_{A_{n}}\right\Vert \left\Vert T_{A_{n}}^{-1}\right\Vert \stackrel{n\rightarrow\infty}{\longrightarrow}1$.
\item $V\subseteq N$, where $N=\cup_{i=1}^{k}\cup_{r=1}^{k_{i}}N_{i,r}$
is a union of unions of subspaces of $\mathbb{C}^{d}$ as in condition
$2$ of Claim \ref{Claim: 5.13}. In this case, for every $n\in\mathbb{N}$,
$T_{A_{n}}$ has an extension to a RKHS isomorphism $\widetilde{T_{A_{n}}}:H_{N\cap B_{d}}\rightarrow H_{A_{n}N\cap B_{d}}$.
By Proposition \ref{Proposition 5.12}, there exist $s\in\mathbb{R}$,
$0\leq s<1$, and $N\left(s,k\right)\in\mathbb{N}$ s.t. for every
$n\in\mathbb{N}$ and $N\geq N\left(s,k\right)$, 
\[
\left\Vert \widetilde{T_{A_{n}}}\right\Vert \leq max\left\{ \left\Vert A_{n}\right\Vert ^{N},\sqrt{\frac{1+s^{N}k}{1-s^{N}k}}\underset{1\leq i\leq k}{\max}\left\Vert \widetilde{T_{A_{n}}}\mid_{H_{N_{i}\cap B_{d}}}\right\Vert \right\} 
\]
 and 

\[
\left\Vert \widetilde{T_{A_{n}}}^{-1}\right\Vert \leq max\left\{ \left\Vert A_{n}^{-1}\right\Vert ^{N},\sqrt{\frac{1+s^{N}k}{1-s^{N}k}}\underset{1\leq i\leq k}{\max}\left\Vert \widetilde{T_{A_{n}}}^{-1}\mid_{H_{A_{n}N_{i}\cap B_{d}}}\right\Vert \right\} \cdot
\]
 It follows that 
\[
\limsup_{n}\left\Vert \widetilde{T_{A_{n}}}\right\Vert \leq\sqrt{\frac{1+s^{N}k}{1-s^{N}k}}\limsup_{n}\underset{1\leq i\leq k}{\max}\left\Vert \widetilde{T_{A_{n}}}\mid_{H_{N_{i}\cap B_{d}}}\right\Vert 
\]
 and 

\[
\limsup_{n}\left\Vert \widetilde{T_{A_{n}}}^{-1}\right\Vert \leq\sqrt{\frac{1+s^{N}k}{1-s^{N}k}}\limsup_{n}\underset{1\leq i\leq k}{\max}\left\Vert \widetilde{T_{A_{n}}}^{-1}\mid_{H_{A_{n}N_{i}\cap B_{d}}}\right\Vert \cdot
\]
 By the induction hypothesis, it follows that 
\[
\limsup_{n}\left\Vert \widetilde{T_{A_{n}}}\right\Vert \leq\sqrt{\frac{1+s^{N}k}{1-s^{N}k}}
\]
 and 

\[
\limsup_{n}\left\Vert \widetilde{T_{A_{n}}}^{-1}\right\Vert \leq\sqrt{\frac{1+s^{N}k}{1-s^{N}k}}\cdot
\]
 Since $0\leq s<1$, it follows that 
\[
\lim_{n}\left\Vert \widetilde{T_{A_{n}}}\right\Vert =1
\]
 and 
\[
\lim_{n}\left\Vert \widetilde{T_{A_{n}}}^{-1}\right\Vert =1\cdot
\]
 Since $V\subseteq N\cap B_{d}$ and for every $n\in\mathbb{N}$,
$W_{n}\subseteq A_{n}N\cap B_{d}$ , it follows that for every $n\in\mathbb{N}$,
\[
1\leq\left\Vert T_{A_{n}}\right\Vert \leq\left\Vert \widetilde{T_{A_{n}}}\right\Vert \underset{n\rightarrow\infty}{\longrightarrow}1
\]
 and 
\[
1\leq\left\Vert T_{A_{n}}^{-1}\right\Vert \leq\left\Vert \widetilde{T_{A_{n}}}^{-1}\right\Vert \underset{n\rightarrow\infty}{\longrightarrow}1\cdot
\]

\item $V$ satisfies condition $3$ in Claim \ref{Claim: 5.13} , i.e.,
$V\subseteq E\oplus N$, where $N\subseteq\mathbb{C}^{d}$ is a tractable
union of subspaces of $\textrm{span}N$ and $\left\{ 0\right\} \neq E\subseteq\mathbb{C}^{d}$
as in condition $3$ of Claim \ref{Claim: 5.13} . In this case, for
every $n\in\mathbb{N}$, $T_{A_{n}}$ has an extension to a RKHS isomorphism
\[
\widetilde{T_{A_{n}}}:H_{\left(E\oplus N\right)\cap B_{d}}\rightarrow H_{A_{n}\left(N\oplus E\right)\cap B_{d}}\cdot
\]
By Corollary \ref{Corollary 5.6} and the induction hypothesis, it
follows that for every $n\in\mathbb{N}$, 
\[
1\leq\left\Vert T_{A_{n}}\right\Vert \leq\left\Vert \widetilde{T_{A_{n}}}\right\Vert =\left\Vert \widetilde{T_{A_{n}}}\mid_{H_{N\cap B_{d}}}\right\Vert \underset{n\rightarrow\infty}{\longrightarrow}1
\]
 and 
\[
1\leq\left\Vert T_{A_{n}}^{-1}\right\Vert \leq\left\Vert \widetilde{T_{A_{n}}}^{-1}\right\Vert =\left\Vert \widetilde{T_{A_{n}}}^{-1}\mid_{H_{A_{n}N\cap B_{d}}}\right\Vert \underset{n\rightarrow\infty}{\longrightarrow}1\cdot
\]
\end{itemize}
\end{proof}
At the end of this Section, it may be worth mentioning that, at present,
we believe that Theorem \ref{Theorem 5.14} remains true even without
the tractability assumption.

\section{Vector Valued Multipliers and Their Multiplier Norms}

\subsection{Useful Equalities for Vector Valued Multipliers}
\begin{prop}
\label{Proposition 71}Let $F=\left(f_{1},\ldots,f_{d}\right)$ and
$G=\left(g_{1},\ldots,g_{d}\right)$ be vector valued multipliers
on $B_{d}$. Let $V\subseteq B_{d}$ be a variety, and let $U:\mathbb{C^{\mathit{d}}}\longrightarrow\mathbb{C^{\mathit{d}}}$
be is a unitary. Then,
\end{prop}

\begin{enumerate}
\item for every $x,y\in B_{d}$, $\left\langle M_{G}M_{F}^{*}k_{x},k_{y}\right\rangle =\left\langle G\left(y\right),F\left(x\right)\right\rangle \left\langle k_{x},k_{y}\right\rangle $.
\item for every $x,y\in V$, $\left\langle M_{F}M_{F}^{*}k_{x},k_{y}\right\rangle =\left\langle M_{F\mid_{V}}M_{F\mid_{V}}^{*}k_{x},k_{y}\right\rangle $.
\item for every $x,y\in V$, 
\[
\left\langle \left(M_{F\mid_{V}}M_{F\mid_{V}}^{*}-M_{G\mid_{V}}M_{G\mid_{V}}^{\ast}\right)k_{x},k_{y}\right\rangle =\left(\left\langle F\left(y\right),F\left(x\right)\right\rangle -\left\langle G\left(y\right),G\left(x\right)\right\rangle \right)\left\langle k_{x},k_{y}\right\rangle \cdot
\]
\item $\sup_{\mathit{x\in V}}\left|\left\Vert F\left(x\right)\right\Vert ^{2}-\left\Vert G\left(x\right)\right\Vert ^{2}\right|\leq\left\Vert M_{F\mid_{V}}M_{F\mid_{V}}^{*}-M_{G\mid_{V}}M_{G\mid_{V}}^{*}\right\Vert $.
\item $M_{UF}M_{UG}^{*}=M_{F}M_{G}^{\ast}$.
\end{enumerate}
\begin{proof}
Let , $x,y\in B_{d}$.Then,
\begin{enumerate}
\item For every $i\in\left\{ 1,\ldots,d\right\} $, we have 
\[
\left\langle M_{g_{i}}M_{f_{i}}^{*}k_{x},k_{y}\right\rangle =\left\langle \overline{f_{i}\left(x\right)}k_{x},\overline{g_{i}\left(y\right)}k_{y}\right\rangle =g_{i}\left(y\right)\overline{f_{i}\left(x\right)}\left\langle k_{x},k_{y}\right\rangle \cdot
\]
Therefore, 
\begin{flalign*}
\left\langle M_{G}M_{F}^{*}k_{x},k_{y}\right\rangle  & =\left\langle k_{x},k_{y}\right\rangle \sum_{i=1}^{d}g_{i}\left(y\right)\overline{f_{i}\left(x\right)}\\
 & =\left\langle k_{x},k_{y}\right\rangle \sum_{i=1}^{d}\left\langle G\left(y\right),e_{i}\right\rangle \overline{\left\langle F\left(x\right),e_{i}\right\rangle }\\
 & =\left\langle k_{x},k_{y}\right\rangle \left\langle G\left(y\right),F\left(x\right)\right\rangle .
\end{flalign*}
 Where $\left\{ e_{i}\right\} _{i=1}^{d}$ is the standard basis of
$\mathbb{C^{\mathit{d}}}$.
\item Let $x,y\in V$. Note that 
\begin{eqnarray*}
\left\langle M_{F}M_{F}^{\ast}k_{x},k_{y}\right\rangle  & = & \left\langle \left(M_{F}\right)\mid_{\oplus_{i=1}^{d}H_{V}}M_{F}^{*}k_{x},k_{y}\right\rangle \\
 & = & \left\langle P_{H_{V}}\left(M_{F}\right)\mid_{\oplus_{i=1}^{d}H_{V}}M_{F}^{\ast}P_{H_{V}}k_{x},k_{y}\right\rangle =\left\langle M_{F\mid_{V}}M_{F\mid_{V}}^{*}k_{x},k_{y}\right\rangle \cdot
\end{eqnarray*}
\item The assertion follows by $\left(2\right)$
\item The assertion follows by $\left(3\right)$.
\item By $\left(1\right)$ we have that for every $x,y\in B_{d}$,
\[
\left\langle M_{UF}M_{UG}^{*}k_{x},k_{y}\right\rangle =\left\langle UFy,UGx\right\rangle \left\langle k_{x},k_{y}\right\rangle =\left\langle Fy,Gx\right\rangle \left\langle k_{x},k_{y}\right\rangle =\left\langle M_{F}M_{G}^{*}k_{x},k_{y}\right\rangle .
\]
 Therefore, $M_{UF}M_{UG}^{*}=M_{F}M_{G}^{*}$.
\end{enumerate}
\end{proof}

\subsection{Linear transformations as Vector Valued Multipliers}
\begin{notation}
Let $\left\{ e_{i}\right\} _{i=1}^{d}$ denote the standard basis
of $\mathbb{C^{\mathit{d}}}$. For any linear map $A:\mathbb{C^{\mathit{d}}}\longrightarrow\mathbb{C^{\mathit{d}}}$
and every $i\in\left\{ 1,\ldots,d\right\} $, let $A_{i}:\mathbb{C^{\mathit{d}}}\longrightarrow\mathbb{C}$
denote the linear functional, which is given by 
\[
A_{i}\left(x\right)=\left\langle Ax,e_{i}\right\rangle \quad x\in\mathbb{C^{\mathit{d}}}.
\]
 Then, 
\[
A=\left(A_{1},\ldots,A_{d}\right)\cdot
\]
\end{notation}

Note that $A_{i}$ is either a homogeneous polynomial of degree $1$,
or the zero ploynomial. Therefore, $A_{i}\in\mathcal{M}_{d}$. Hence,
\[
A=\left(A_{1},\ldots,A_{d}\right):B_{d}\longrightarrow\mathbb{C^{\mathit{d}}}
\]
 is a vector valued multiplier on $B_{d}$. By definition, for every
$i\in\left\{ 1,\ldots,d\right\} $, the $\mathcal{H_{\mathit{d}}^{\mathit{2}}}$-norm
of $A_{i}$ is $\left\Vert A^{\ast}\left(e_{i}\right)\right\Vert $,
(see \cite[Section 1]{arveson1998subalgebras}). That is, 
\[
\left\Vert A_{i}\right\Vert {}_{\mathcal{H_{\mathit{d}}^{\mathit{2}}}}=\left\Vert A^{\ast}\left(e_{i}\right)\right\Vert \cdot
\]
Moreover, making use of the Riesz theorem yields 
\[
\left\Vert A_{i}\right\Vert _{OP}=\left\Vert A^{*}\left(e_{i}\right)\right\Vert \cdot
\]
By Proposition \ref{Proposition 74} below, we have 
\[
\left\Vert A_{i}\right\Vert _{\mathcal{M}_{d}}=\left\Vert A_{i}\right\Vert _{\mathcal{H}_{d}^{2}}\cdot
\]
Hence, we obtain the following Corollary.
\begin{cor}
\label{Corollary 73}Let $A:\mathbb{C^{\mathit{d}}}\longrightarrow\mathbb{C^{\mathit{d}}}$
be a linear map. Then,
\end{cor}

\begin{enumerate}
\item $\left\Vert A_{i}\right\Vert _{\mathcal{M}_{d}}=\left\Vert A_{i}\right\Vert {}_{\mathcal{H_{\mathit{d}}^{\mathit{2}}}}=\left\Vert A^{\ast}\left(e_{i}\right)\right\Vert =\left\Vert A_{i}\right\Vert _{OP}=\sqrt{\sum_{j=1}^{d}\left|\left\langle Ae_{j},e_{i}\right\rangle \right|^{2}}$
.
\item $\left\Vert A\right\Vert _{\oplus_{i=1}^{d}\mathcal{H_{\mathit{d}}^{\mathit{2}}}}=\sqrt{\sum_{i,j=1}^{d}\left|\left\langle Ae_{j},e_{i}\right\rangle \right|^{2}}$
.
\end{enumerate}
\begin{prop}
\label{Proposition 74}Let $n\in\mathbb{N}_{0}$. Suppose that $f$
is a homogeneous polynomial of degee $n$ given by
\[
f\left(x\right)=\langle x^{n},\xi\rangle,\quad\textrm{for all }x\in B_{d},
\]
for some $\xi\in\left(\mathbb{C}^{d}\right)^{n}$. Then, $f\in\mathcal{M}_{d}$
and $\left\Vert f\right\Vert _{\mathcal{M}_{d}}=\left\Vert f\right\Vert _{\mathcal{H}_{d}^{2}}$.
\end{prop}

\begin{proof}
Let $g=\sum_{k=0}^{\infty}\left\langle \cdot,\xi_{n}\right\rangle \in\mathcal{H}_{d}^{2}$
and let $m\in\mathbb{N}_{0}$. Then, 
\[
f\cdot g_{m}=\left\langle \cdot,\xi\right\rangle _{\left(\mathbb{C}^{d}\right)^{n}}\cdot\left\langle \cdot,\xi_{m}\right\rangle _{\left(\mathbb{C}^{d}\right)^{m}}=\left\langle \cdot,P_{\left(\mathbb{C}^{d}\right)^{n+m}}\left(\xi\otimes\xi_{m}\right)\right\rangle _{\left(\mathbb{C}^{d}\right)^{n+m}},
\]
where $P_{\left(\mathbb{C}^{d}\right)^{n+m}}$ is the orthogonal projection
of $\left(\mathbb{C}^{d}\right)^{\otimes\left(n+m\right)}$ onto $\left(\mathbb{C}^{d}\right)^{n+m}$.
Therefore, $fg_{m}\in\mathcal{H}_{d}^{2}$ and
\[
\left\Vert f\cdot g_{m}\right\Vert _{\mathcal{H}_{d}^{2}}\leq\left\Vert f\right\Vert _{\mathcal{H}_{d}^{2}}\left\Vert g_{m}\right\Vert _{\mathcal{H}_{d}^{2}}\cdot
\]
It follows that, 
\[
\sum_{m=0}^{\infty}\left\Vert f\cdot g_{m}\right\Vert _{\mathcal{H}_{d}^{2}}^{2}\leq\left\Vert f\right\Vert _{\mathcal{H}_{d}^{2}}^{2}\left\Vert g\right\Vert _{\mathcal{H}_{d}^{2}}^{2}<\infty\cdot
\]
Therefore, 
\[
\sum_{m=0}^{\infty}f\cdot g_{m}\in\mathcal{H}_{d}^{2}\cdot
\]
 Therefore, we obtain that for every $x\in B_{d}$,
\[
f\left(x\right)g\left(x\right)=\sum_{m=0}^{\infty}f\left(x\right)g_{m}\left(x\right)=\sum_{m=0}^{\infty}\left\langle x^{m},P_{\left(\mathbb{C}^{d}\right)^{n+m}}\left(\xi\otimes\xi_{m}\right)\right\rangle _{\left(\mathbb{C}^{d}\right)^{n+m}}
\]
and that 
\[
fg=\sum_{m=0}^{\infty}fg_{m}\in\mathcal{H}_{d}^{2}
\]
with

\[
\left\Vert fg\right\Vert _{\mathcal{H}_{d}^{2}}=\sqrt{\sum_{m=0}^{\infty}\left\Vert f\cdot g_{m}\right\Vert _{\mathcal{H}_{d}^{2}}^{2}}\leq\left\Vert f\right\Vert _{\mathcal{H}_{d}^{2}}\left\Vert g\right\Vert _{\mathcal{H}_{d}^{2}}\cdot
\]
Hence, $\left\Vert f\right\Vert _{\mathcal{M}_{d}}\leq\left\Vert f\right\Vert _{\mathcal{H}_{d}^{2}}$.
Now, since always we have $\left\Vert f\right\Vert _{\mathcal{M}_{d}}\geq\left\Vert f\right\Vert _{\mathcal{H}_{d}^{2}}$,
it follows that $\left\Vert f\right\Vert _{\mathcal{M}_{d}}=\left\Vert f\right\Vert _{\mathcal{H}_{d}^{2}}$.
\end{proof}
\begin{prop}
\label{Proposition 75}Let $A,B:\mathbb{C^{\mathit{d}}}\longrightarrow\mathbb{C^{\mathit{d}}}$
be linear transformations. Then, 
\[
M_{A}M_{B}^{*}=M_{B^{*}A}M_{{\bf Id}}^{\ast}=M_{{\bf Id}}M_{A^{*}B}^{*}\cdot
\]
\end{prop}

\begin{proof}
By Proposition \ref{Proposition 71}, we have that for every $x,y\in B_{d}$,
\begin{eqnarray*}
\left\langle M_{B^{*}A}M_{{\bf Id}}^{*}k_{x},k_{y}\right\rangle  & = & \left\langle B^{*}Ay,{\bf Id}x\right\rangle \left\langle k_{x},k_{y}\right\rangle \\
 & = & \left\langle Ay,Bx\right\rangle \left\langle k_{x},k_{y}\right\rangle \\
 & = & \left\langle M_{A}M_{B}^{*}k_{x},k_{y}\right\rangle .
\end{eqnarray*}
 Therefore, $M_{B^{*}A}M_{{\bf Id}}^{*}=M_{A}M_{B}^{*}$. The proof
for $M_{A}M_{B}^{*}=M_{{\bf Id}}M_{A^{*}B}^{*}$ is the same.
\end{proof}
\begin{prop}
Let $A:\mathbb{C}^{d}\longrightarrow\mathbb{C}^{d}$ be a linear map.
Let $V\subseteq B_{d}$ be a variety. Then, 
\[
M_{A\mid_{V}}M_{A\mid_{V}}^{*}=M_{{\bf Id\mid_{\mathit{V}}}}M_{{\bf Id\mid_{\mathit{V}}}}^{*}
\]
 if and only if $A$ is an isometry on $\textrm{\text{span}}\left(V\right)$.
\end{prop}

\begin{proof}
By Proposition \ref{Proposition 71}, for every $x,y\in\mathit{V}$,
\[
\left\langle \left(M_{A\mid_{V}}M_{A\mid_{V}}^{*}-M_{{\bf Id\mid_{\mathit{V}}}}M_{{\bf Id\mid_{\mathit{V}}}}^{\ast}\right)k_{x},k_{y}\right\rangle =\left(\left\langle A\left(y\right),A\left(x\right)\right\rangle -\left\langle y,x\right\rangle \right)\left\langle k_{x},k_{y}\right\rangle \cdot
\]
 It follows that 
\[
M_{A\mid_{V}}M_{A\mid_{V}}^{*}=M_{{\bf Id\mid_{\mathit{V}}}}M_{{\bf Id\mid_{\mathit{V}}}}^{*}
\]
 if and only if for every $x,y\in V$, 
\[
\left\langle A\left(y\right),A\left(x\right)\right\rangle =\left\langle y,x\right\rangle \cdot
\]
Now, use the fact that for every $x,y\in\textrm{\text{span}}\left(V\right)$,
\[
\langle A\left(y\right),A\left(x\right)\rangle=\left\langle y,x\right\rangle 
\]
 if and only if for every $x,y\in V$, 
\[
\left\langle A\left(y\right),A\left(x\right)\right\rangle =\left\langle y,x\right\rangle \cdot
\]
\end{proof}

\subsection{The Multiplier Norms of Linear transformations on Homogeneous Varieties\label{sec:Linear-maps-as}}

In this section, the main result we present is that the multiplier
norm of a linear map on a homogeneous variety is equal to its operator
norm on the subspace, of $\mathbb{C}^{d}$, spanned by the variety.
\begin{thm}
\label{Theorem 77}Let $A:\mathbb{C^{\mathit{d}}}\longrightarrow\mathbb{C^{\mathit{d}}}$
be a linear map. Then, $\left\Vert M_{A}\right\Vert =\left\Vert A\right\Vert _{OP}\cdot$
\end{thm}

\begin{proof}
Let $A:\mathbb{C^{\mathit{d}}}\longrightarrow\mathbb{C^{\mathit{d}}}$
be a linear map. Let $x\in B_{d}$, and let $\rho_{x}:\mathcal{M_{\mathit{d}}\longrightarrow\mathbb{C}}$
be the point evaluation linear functional at $x$. Then,
\[
\left\Vert A\left(x\right)\right\Vert =\left\Vert \left(A_{1}\left(x\right),\ldots,A_{d}\left(x\right)\right)\right\Vert =\left\Vert \left(\rho_{x}\circ A_{1},\ldots,\rho_{x}\circ A_{d}\right)\right\Vert .
\]
Since $\rho_{x}$ is completely contractive, it follows that 
\[
\left\Vert A\left(x\right)\right\Vert \leq\left\Vert \left(A_{1},\ldots A_{d}\right)\right\Vert _{Mult\left(B_{d}\right)}:=\parallel M_{A}\parallel.
\]
Therefore,
\[
\left\Vert A\right\Vert _{OP}=\sup_{x\in B_{d}}\left\Vert A\left(x\right)\right\Vert \leq\left\Vert M_{A}\right\Vert .
\]
It remains to show that $\left\Vert M_{A}\right\Vert \leq\left\Vert A\right\Vert {}_{OP}$:
First, assume that $A:\mathbb{C^{\mathit{d}}}\longrightarrow\mathbb{C^{\mathit{d}}}$
is a unitarily diagonalizable. In this case, there exists an unitary
$U:\mathbb{C^{\mathit{d}}}\longrightarrow\mathbb{C^{\mathit{d}}}$,
and $\lambda=\left(\lambda_{1},\ldots,\lambda_{d}\right)\in\mathbb{C^{\mathit{d}}}$
s.t. 
\[
UAU^{*}e_{i}=\lambda_{i}e_{i},\;\textrm{for all }i\in\left\{ 1,\ldots,d\right\} .
\]
Let $x\in B_{d}$. Then, 
\begin{eqnarray*}
\left(UA\right)_{i}\left(x\right) & = & \left\langle UAx,e_{i}\right\rangle =\sum_{j=1}^{d}\left\langle x,U^{*}e_{j}\right\rangle \left\langle UAU^{*}e_{j},e_{i}\right\rangle \cdot\\
 & = & \lambda_{i}\langle Ux,e_{i}\rangle=\lambda_{i}U_{i}\left(x\right).
\end{eqnarray*}
That is, $\left(UA\right)_{i}=\lambda_{i}U_{i}$. Now, let $\left(f_{1},\ldots,f_{d}\right)\in\oplus_{i=1}^{d}\mathcal{H_{\mathit{d}}^{\mathit{2}}}$.
Then, 
\begin{align*}
\left\Vert M_{UA}\left(f_{1},\ldots,f_{d}\right)\right\Vert ^{2} & =\left\Vert \sum_{i=1}^{d}\lambda_{i}U_{i}f_{i}\right\Vert ^{2}\\
 & =\left\Vert M_{U}\left(\lambda_{1}f_{1},\ldots,\lambda_{d}f_{d}\right)\right\Vert ^{2}\\
 & \leq\left\Vert \left(\lambda_{1}f_{1},\dots,\lambda_{d}f_{d}\right)\right\Vert ^{2}\\
 & =\sum_{i=1}^{d}\mid\lambda_{i}\mid^{2}\parallel f_{i}\parallel^{2}\\
 & \leq\left(\max_{i\in\left\{ 1,\ldots d\right\} }\left\{ \mid\lambda_{i}\mid\right\} \right)^{2}\left\Vert \left(f_{1},\ldots,f_{d}\right)\right\Vert ^{2}\\
 & =\left\Vert A\right\Vert _{OP}^{2}\parallel\left\Vert f_{1},\ldots,f_{d}\right\Vert ^{2}\cdot
\end{align*}
Hence, 
\[
\left\Vert M_{UA}\right\Vert \leq\left\Vert A\right\Vert _{OP}\cdot
\]
 By Proposition \ref{Proposition 71}, we have $\left\Vert M_{UA}\right\Vert =\left\Vert M_{A}\right\Vert $.
Therefore, 
\[
\left\Vert M_{A}\right\Vert \leq\left\Vert A\right\Vert _{OP}.
\]
 Finally, let $A:\mathbb{C^{\mathit{d}}}\longrightarrow\mathbb{C^{\mathit{d}}}$
be any linear map. Since $A^{*}A$ is a normal linear map, we have
\[
\left\Vert M_{A^{*}A}\right\Vert =\left\Vert A^{*}A\right\Vert _{OP}
\]
And, by Proposition \ref{Proposition 75}, we have
\[
\left\Vert M_{A}M_{A}^{*}\right\Vert =\left\Vert M_{A^{*}A}M_{{\bf Id}}^{*}\right\Vert .
\]
It follows that, 
\[
\left\Vert M_{A}\right\Vert ^{2}=\left\Vert M_{A}M_{A}^{*}\right\Vert \leq\left\Vert M_{A^{*}A}\right\Vert \left\Vert M_{{\bf Id}}^{*}\right\Vert =\left\Vert M_{A^{*}A}\right\Vert =\left\Vert A\right\Vert _{OP}^{2}.
\]
 Therefore, $\left\Vert M_{A}\right\Vert \leq\left\Vert A\right\Vert _{OP}.$
Hence, $\left\Vert M_{A}\right\Vert =\left\Vert A\right\Vert _{OP}$.
\end{proof}
\begin{lem}
\label{Lemma 5.6}Let $V\subseteq B_{d}$ be a homogeneous variety.
Let $A:\mathbb{C}^{d}\longrightarrow\mathbb{C}^{d}$ be a linear map.
Let $F=\left(f_{1},\ldots,f_{d}\right)$, where for every $i=1,\ldots,d$,
$f_{i}\in\mathcal{H}_{d}^{2}$ . For every $n\in\mathbb{N}_{0}$,
let $F_{n}$ be the $n$-th homogeneous component of $F$. If $F\mid_{V}=A\mid_{V}$,
then 
\[
A\mid_{B_{d}\cap\textrm{span}\left(V\right)}=\left(F_{1}\right)\mid_{B_{d}\cap\textrm{span}\left(V\right)}\cdot
\]
\end{lem}

\begin{proof}
Note that $F=\sum_{n=1}^{\infty}F_{n}$, because $F\left(0\right)=0$.
By assumption, for every $x\in V$, we have
\[
A\left(x\right)-\sum_{n=1}^{\infty}F_{n}\left(x\right)=0\cdot
\]
 Since $V$ is homogeneous, it follows, by Theorem \ref{Theorem 36},
that $A\mid_{V}=\left(F_{1}\right)\mid_{V}$ and for every $2\leq n\in\mathbb{N}$,
$\left(F_{n}\right)\mid_{V}=0$. Since $F_{1}$ is uniquely extened
to a linear map on $\mathbb{C}^{d}$, we obtain 
\[
A\mid_{B_{d}\cap\textrm{span}\left(V\right)}=\left(F_{1}\right)\mid_{B_{d}\cap\textrm{span}\left(V\right)}\cdot
\]
\end{proof}
\begin{thm}
\label{Theorem 6.9}Let $V\subseteq B_{d}$ be a homogeneous variety.
Let $A:\mathbb{C}^{d}\longrightarrow\mathbb{C}^{d}$ be a linear map.
Then,
\[
\left\Vert A\mid_{V}\right\Vert _{Mult\left(V\right)}=\left\Vert A\mid_{\textrm{span}\left(V\right)}\right\Vert _{OP}\cdot
\]
\end{thm}

\begin{proof}
By Proposition \ref{Proposition 1.14}, there exists a vector valued
multiplier $F=\sum_{n=1}^{\infty}F_{n}$ on $B_{d}$, s.t. $F\mid_{V}=A\mid_{V}$
with $\left\Vert F\right\Vert _{\mathit{Mult}\left(B_{d}\right)}=\left\Vert A\mid_{V}\right\Vert _{Mult\left(V\right)}$.
Since $F\mid_{V}=A\mid_{V}$, it follows, by Lemma \ref{Lemma 5.6},
that
\[
A\mid_{B_{d}\cap\textrm{span}\left(V\right)}=\left(F_{1}\right)\mid_{B_{d}\cap\textrm{span}\left(V\right)}\cdot
\]
 By \cite[Lemma 6.1]{davidson2011isomorphism}, we have $\left\Vert F_{1}\right\Vert _{\mathit{Mult\left(B_{d}\right)}}\leq\left\Vert F\right\Vert _{Mult\left(B_{d}\right)}$.
Therefore, 
\begin{eqnarray*}
\left\Vert A\mid_{V}\right\Vert _{Mult\left(V\right)} & \leq & \left\Vert A\mid_{B_{d}\cap\textrm{span}\left(V\right)}\right\Vert _{Mult\left(B_{d}\cap\textrm{span}\left(V\right)\right)}\\
 & \leq & \left\Vert F_{1}\right\Vert _{\mathit{Mult\left(B_{d}\right)}}\\
 & \leq & \left\Vert F\right\Vert _{Mult\left(B_{d}\right)}\\
 & = & \left\Vert A\mid_{V}\right\Vert _{Mult\left(V\right)}\cdot
\end{eqnarray*}
 Hence, 
\[
\left\Vert A\mid_{V}\right\Vert _{Mult\left(V\right)}=\left\Vert A\mid_{B_{d}\cap\textrm{span}\left(V\right)}\right\Vert _{Mult\left(B_{d}\cap\textrm{span}\left(V\right)\right)}\cdot
\]
By Theorem \ref{Theorem 77}, 
\[
\left\Vert A\mid_{\textrm{span}\left(V\right)}\right\Vert _{OP}=\left\Vert A\mid_{B_{d}\cap\textrm{span}\left(V\right)}\right\Vert _{Mult\left(B_{d}\cap\textrm{span}\left(V\right)\right)}\cdot
\]
 Therefore, 
\[
\left\Vert A\mid_{V}\right\Vert _{Mult\left(V\right)}=\left\Vert A\mid_{\textrm{span}\left(V\right)}\right\Vert _{OP}\cdot
\]
\end{proof}
\begin{thm}
Let $V,W\subseteq B_{d}$ be homogeneous varieties. Suppose that $F=\left(f_{1},\ldots,f_{d}\right)$
is a biholomorphism of $V$onto $W$, with $F\left(0\right)=0$, where
for every $i=1,\ldots,d$, $f_{i}\in\mathcal{H}_{d}^{2}$. Then, 
\[
F\mid_{V}=\left(F_{1}\right)\mid_{V}\cdot
\]
\end{thm}

\begin{proof}
Since $F\left(0\right)=0$ and $V,W$ are homogeneous, it follows,
by Theorem \ref{Theorem 1.34}, that there exists a linear map $A:\mathbb{C}^{d}\longrightarrow\mathbb{C}^{d}$
s.t. that $F\mid_{V}=A\mid_{V}$. Now, the assertion follows from
Lemma \ref{Lemma 5.6}.
\end{proof}

\subsection{Automorphisms of $B_{d}$ as Vector Valued Multipliers and the Induced
Projections on the Drury-Arveson space}
\begin{thm}
\label{Theorem 713}Let $F=\left(f_{1},\ldots,f_{d}\right)$ be a
vector valued multiplier on $B_{d}$. Suppose that $F\left(a\right)=0$
for some $a\in B_{d}$. Then, 
\[
M_{F}M_{F}^{*}+P_{H_{\left\{ a\right\} }}={\bf Id}
\]
 if and only if 
\[
F\in{\bf Aut\left(\mathit{B_{d}}\right)},
\]
Where $P_{H_{\left\{ a\right\} }}$ is the orthogonal projection of
$\mathcal{H}_{d}^{2}$ onto $H_{\left\{ a\right\} }$; i.e., for every
$f\in\mathcal{H}_{d}^{2}$, $P_{H_{\left\{ a\right\} }}\left(f\right)=\left\langle f,\frac{k_{a}}{\left\Vert k_{a}\right\Vert }\right\rangle \frac{k_{a}}{\left\Vert k_{a}\right\Vert }$.
\end{thm}

\begin{proof}
Let $x,y\in B_{d}$. Then, 
\[
\left\langle P_{H_{\left\{ a\right\} }}k_{x},k_{y}\right\rangle =\frac{1}{\left\Vert k_{a}\right\Vert ^{2}}\left\langle k_{x},k_{a}\right\rangle \left\langle k_{a},k_{y}\right\rangle =\overline{\delta_{a}\left(x\right)}\delta_{a}\left(y\right),
\]
where for every $x\in B_{d}$, $\delta_{a}\left(x\right):=\frac{k_{a}\left(x\right)}{\left\Vert k_{a}\right\Vert }$.
Now, suppose that $F\in{\bf Aut\left(\mathit{\mathit{B_{d}}}\right)}$.
Then, by Proposition \ref{Theorem 1.31}, it follows that 
\[
\left\langle P_{H_{\left\{ a\right\} }}k_{x},k_{y}\right\rangle =\frac{\left\langle k_{x},k_{y}\right\rangle }{\left\langle k_{F\left(x\right)},k_{F\left(y\right)}\right\rangle }=\left(1-\left\langle F\left(y\right),F\left(x\right)\right\rangle \right)\left\langle k_{x},k_{y}\right\rangle \cdot
\]
By Proposition \ref{Proposition 71}, we have $\left\langle M_{F}M_{F}^{*}k_{x},k_{y}\right\rangle =\left\langle F\left(y\right),F\left(x\right)\right\rangle \left\langle k_{x},k_{y}\right\rangle $.
Therefore, 
\[
\left\langle \left(M_{F}M_{F}^{*}+P_{H_{\left\{ a\right\} }}\right)k_{x},k_{y}\right\rangle =\left\langle k_{x},k_{y}\right\rangle ,\,x,y\in B_{d}\cdot
\]
 Hence, 
\[
M_{F}M_{F}^{*}+P_{H_{\left\{ a\right\} }}={\bf Id}\cdot
\]
 Conversely, suppose that 
\[
M_{F}M_{F}^{*}+P_{H_{\left\{ a\right\} }}={\bf Id}\cdot
\]
 Let $\Phi_{a}$ be the elementary automorphism of $B_{d}$ that interchanges
$a\text{\  and \  0}$. Then, by the first direction above, we have
\[
M_{\Phi_{a}}M_{\Phi_{a}}^{*}={\bf Id}-P_{H_{\left\{ a\right\} }}\cdot
\]
Therefore,
\[
M_{F}M_{F}^{*}=M_{\Phi_{a}}M_{\Phi_{a}}^{*},
\]
It follows that 
\[
\left\langle M_{F}M_{F}^{*}k_{x},k_{y}\right\rangle =\left\langle M_{\Phi_{a}}M_{\Phi_{a}}^{*}k_{x},k_{y}\right\rangle ,\,x,y\in B_{d}\cdot
\]
Hence, 
\[
\left\langle F\left(y\right),F\left(x\right)\right\rangle \left\langle k_{x},k_{y}\right\rangle =\left\langle \Phi_{a}\left(y\right),\Phi_{a}\left(x\right)\right\rangle \left\langle k_{x},k_{y}\right\rangle ,\,x,y\in B_{d}\cdot
\]
 Therefore, 
\[
\left\langle F\left(y\right),F\left(x\right)\right\rangle =\left\langle \Phi_{a}\left(y\right),\Phi_{a}\left(x\right)\right\rangle ,\,x,y\in B_{d}\cdot
\]
 Now, since $\Phi_{a}\circ\Phi_{a}={\bf Id}$, it follows that 
\[
\left\langle F\circ\Phi_{a}\left(y\right),F\circ\Phi_{a}\left(x\right)\right\rangle =\left\langle y,x\right\rangle ,\,x,y\in B_{d}\cdot
\]
 Therefore, there exists a unitary $U:\mathbb{C}^{d}\longrightarrow\mathbb{C}^{d}$
s.t. $F\circ\Phi_{a}=U$ on $B_{d}$. Hence, $F=U\circ\Phi_{a}$.
That is, $F\in{\bf Aut\left(\mathit{B_{d}}\right)}$.
\end{proof}
\begin{cor}
Let $F\in{\bf Aut\left(\mathit{B_{d}}\right)}$. Then, $\left\Vert F\right\Vert _{Mult\left(B_{d}\right)}=1\cdot$
\end{cor}

\begin{proof}
Let $a=F^{-1}\left(0\right)$. Then, by Corollary \ref{Theorem 713},
$M_{F}M_{F}^{*}$ is a projection.
\end{proof}

\subsection{The Multiplier Norm of The $d$ Shift}
\begin{thm}
\label{Theorem 5.200}( \cite[Theorem 8.2.2]{rudin2008function}).
Let $\mathit{F:B_{d}\longrightarrow B_{d}}$ be holomorphic with $F\left(0\right)=0.$Then,
$F$ and the linear operator $D_{F}\left(0\right)$ fix the same points
of $B_{d}$.
\end{thm}

\begin{thm}
\label{Theorem 06.15}Let $V\subseteq B_{d}$ be a variety that contains
at least two points. Assume that $0\in V$. Then, 
\end{thm}

\begin{enumerate}
\item for every vector valued multiplier $F:B_{d}\longrightarrow\mathbb{C}^{d}$
such that ${\bf \mathit{Z}\mid_{\mathit{V}}}=F\mid_{V}$ with $\left\Vert {\bf \mathit{Z}\mid_{\mathit{V}}}\right\Vert _{Mult\left(V\right)}=\left\Vert F\right\Vert _{Mult\left(B_{d}\right)}$,
we have 
\[
F\mid_{B_{d}\cap\textrm{span}V}={\bf \mathit{Z}}\mid_{B_{d}\cap\textrm{span}V}\cdot
\]
\item $\left\Vert {\bf \mathit{Z}\mid_{\mathit{V}}}\right\Vert _{Mult\left(V\right)}=1\cdot$
\end{enumerate}
\begin{proof}
Let $V\subseteq B_{d}$ be a variety such that $\left|V\right|\geq2$
and  $0\in V$. 
\begin{enumerate}
\item Let a vector valued multiplier $F:B_{d}\longrightarrow\mathbb{C}^{d}$
such that ${\bf \mathit{Z}\mid_{\mathit{V}}}=F\mid_{V}$ with 
\[
\left\Vert {\bf \mathit{Z}\mid_{\mathit{V}}}\right\Vert _{Mult\left(V\right)}=\left\Vert F\right\Vert _{Mult\left(B_{d}\right)}\cdot
\]
It follows that 
\[
\left\Vert F\right\Vert _{\mathit{\infty}}\leq\left\Vert F\right\Vert _{Mult\left(B_{d}\right)}=\left\Vert {\bf \mathit{Z\mathit{\mid_{V}}}}\right\Vert _{Mult\left(V\right)}\leq\left\Vert \mathit{Z}\right\Vert _{Mult\left(B_{d}\right)}=1\cdot
\]
Therefore, $\left\Vert F\right\Vert _{\infty}\leq1$. Now, since $V$
contains at least two points and ${\bf \mathit{Z}\mid_{\mathit{V}}}=F\mid_{V}$
, $F$ is not a constant function. Therefore, by the maximum principle,
$F$ maps $B_{d}$ into $B_{d}$. Note that $F:B_{d}\longrightarrow B_{d}$
is a holomorphic function with $F\left(0\right)=0.$ Therefore, by
Theorem \ref{Theorem 5.200}, $F$ and the linear operator $D_{F}\left(0\right)$
fix the same points of $B_{d}$. In particular, $D_{F}\left(0\right)\mid_{V}=F\mid_{V}={\bf \mathit{Z}\mathit{\mid_{V}}}$.
Since $D_{F}\left(0\right)$ is linear, it follows that $D_{F}\left(0\right)\mid_{\mathit{B_{d}\cap\textrm{span}}V}={\bf \mathit{Z}\mathit{\mid_{B_{d}\cap\textrm{span}V}}}$.
Again, by Theorem \ref{Theorem 5.200}, it follows that 
\[
F\mid_{B_{d}\cap\textrm{span}V}={\bf \mathit{Z}}\mid_{B_{d}\cap\textrm{span}V}\cdot
\]
\item By Proposition \ref{Proposition 1.14}, there exists a vector valued
multiplier $F:B_{d}\longrightarrow\mathbb{C}^{d}$ such that ${\bf \mathit{Z}\mid_{\mathit{V}}}=F\mid_{V}$
with $\left\Vert {\bf \mathit{Z}\mid_{\mathit{V}}}\right\Vert _{Mult\left(V\right)}=\left\Vert F\right\Vert _{Mult\left(B_{d}\right)}$.
By $\left(1\right)$, we have $F\mid_{B_{d}\cap\textrm{span}V}={\bf \mathit{Z}}\mid_{B_{d}\cap\textrm{span}V}$.
Therefore,
\begin{eqnarray*}
\left\Vert F\right\Vert _{Mult\left(B_{d}\right)} & = & \left\Vert {\bf \mathit{Z}\mid_{\mathit{V}}}\right\Vert _{Mult\left(V\right)}\\
 & \leq & \left\Vert {\bf \mathit{Z}\mathit{\mid_{B_{d}\cap\textrm{span}V}}}\right\Vert _{Mult\left(B_{d}\cap\textrm{span}V\right)}\\
 & = & \left\Vert F\mid_{B_{d}\cap\textrm{spanV}}\right\Vert _{Mult\left(B_{d}\cap\textrm{span}V\right)}\\
 & \leq & \left\Vert F\right\Vert _{Mult\left(B_{d}\right)}\cdot
\end{eqnarray*}
Hence, $\left\Vert {\bf \mathit{Z}\mid_{\mathit{V}}}\right\Vert _{Mult\left(V\right)}=\left\Vert {\bf \mathit{Z}\mathit{\mid_{B_{d}\cap\textrm{span}V}}}\right\Vert _{Mult\left(B_{d}\cap\textrm{span}V\right)}=1$.
\end{enumerate}
\end{proof}
\begin{thm}
\label{Thm 07.15}Let $V\subseteq B_{d}$ be a variety that contains
at least two points . Then, 
\[
\left\Vert {\bf \mathit{Z}\mid_{\mathit{V}}}\right\Vert _{Mult\left(V\right)}=1\cdot
\]
\end{thm}

\begin{proof}
Let $a\in V$ and let $\Phi_{a}$ be the elementary automorphism of
$B_{d}$ such that $\Phi_{a}\left(a\right)=0$. Therefore, $0\in\Phi_{a}\left(V\right)$.
By Proposition \ref{Proposition 1.14}, there exists a vector valued
multiplier $F$ on $B_{d}$ such that 
\[
{\bf \mathit{Z}\mid_{\mathit{V}}}=F\mid_{V}
\]
 with $\left\Vert {\bf \mathit{Z}\mid_{\mathit{V}}}\right\Vert _{Mult\left(V\right)}=\left\Vert F\right\Vert _{Mult\left(B_{d}\right)}$.
Therefore, 
\[
\left\Vert F\right\Vert _{Mult\left(B_{d}\right)}=\left\Vert Z\mid_{V}\right\Vert \leq\left\Vert Z\right\Vert _{Mult\left(B_{d}\right)}=1\cdot
\]
 It follows that $F$ maps $B_{d}$ into $B_{d}$. Moreover, 
\[
\Phi_{a}\circ F\circ\Phi_{a}\mid_{\Phi_{a}\left(V\right)}=Z\mid_{\Phi_{a}\left(V\right)}\cdot
\]
 Since $0\in\Phi_{a}\left(V\right)$, and $\Phi_{a}\circ F\circ\Phi_{a}\left(0\right)=0,$
by Theorem \ref{Theorem 5.200}, it follows that 
\[
\Phi_{a}\circ F\circ\Phi_{a}\mid_{B_{d}\cap\textrm{span}\Phi_{a}\left(V\right)}=Z\mid_{B_{d}\cap\textrm{span}\Phi_{a}\left(V\right)}\cdot
\]
 Therefore, 
\[
F\circ\Phi_{a}\mid_{B_{d}\cap\textrm{span}\Phi_{a}\left(V\right)}=\Phi_{a}\mid_{B_{d}\cap\textrm{span}\Phi_{a}\left(V\right)}\cdot
\]
By \cite[Proposition 2.4.2]{rudin2008function}, automorphisms of
the ball map affine sets to affine sets. It follows that 

\[
\left\Vert F\circ\Phi_{a}\right\Vert _{Mult\left(B_{d}\right)}\geq\underset{x\in B_{d}\cap\textrm{span}\Phi_{a}\left(V\right)}{\sup}\left\Vert \Phi_{a}\left(x\right)\right\Vert =1\cdot
\]
 Since $\Phi_{a}$ is an automorphism, it follows that the map 
\begin{eqnarray*}
M_{d}\left(\mathcal{M}_{d}\right) & \to & M_{d}\left(\mathcal{M}_{d}\right)\\
F & \to & F\circ\varPhi_{a}
\end{eqnarray*}
is an isometry. Therefore, 
\[
\left\Vert F\right\Vert _{Mult\left(B_{d}\right)}=\left\Vert F\circ\varPhi_{a}\right\Vert _{Mult\left(B_{d}\right)}\geq1\cdot
\]
But, $\left\Vert F\right\Vert _{Mult\left(B_{d}\right)}\leq1$. Therefore,
$\left\Vert F\right\Vert _{Mult\left(B_{d}\right)}=1\cdot$
\end{proof}

\section{Multiplier Algebra Isomorphisms and the Induced Multiplier Biholomorphisms\label{sec:Norm-Multiplier-algebra-Isomorphisms}}
\begin{claim}
\label{claim: 8.1}Let $V,W$ be two varieties in $B_{d}$. Suppose
that $\varphi:\mathcal{M_{\mathit{V}}}\rightarrow\mathcal{M_{\mathit{W}}}$
is a multiplier algebra isomorphism. Let $F=\left(f_{1},\ldots,f_{d}\right)$
be a vector valued multiplier on $V$. Then, $F\circ G_{\varphi}$
is a vector valued multiplier on $W$ with 
\[
\left\Vert F\circ G_{\varphi}\right\Vert _{Mult\left(W\right)}\leq\left\Vert \varphi^{\left(d\right)}\right\Vert \left\Vert F\right\Vert _{Mult\left(V\right)}\cdot
\]
\end{claim}

\begin{proof}
Note that 
\[
F\circ G_{\varphi}=\left(f_{1}\circ G_{\varphi},\ldots,f_{d}\circ G_{\varphi}\right)=\left(\varphi\left(f_{1}\right),\ldots,\varphi\left(f_{d}\right)\right)\cdot
\]
Therefore, $F\circ G_{\varphi}$ is a vector valued multiplier on
$W$ and 
\[
\left\Vert F\circ G_{\varphi}\right\Vert _{Mult\left(W\right)}\leq\left\Vert \varphi^{\left(d\right)}\right\Vert \left\Vert F\right\Vert _{Mult\left(V\right)}\cdot
\]
\end{proof}
From claim \ref{claim: 8.1} we obtain the following propostion:
\begin{prop}
\label{Proposition 8.2}Let $V,W$ be two varieties in $B_{d}$. Suppose
that $\varphi,\psi:\mathcal{M_{\mathit{V}}}\rightarrow\mathcal{M_{\mathit{W}}}$
are multiplier algebra isomorphisms. Then, $G_{\psi^{-1}}\circ G_{\varphi}$
is a multiplier biholomorphism of $W$ onto $W$, and
\[
\left\Vert G_{\psi^{-1}}\circ G_{\varphi}\right\Vert _{Mult\left(W\right)}\leq\left\Vert \varphi^{\left(d\right)}\right\Vert \left\Vert G_{\psi^{-1}}\right\Vert _{Mult\left(V\right)}\cdot
\]
 In fact, $\varphi\circ\psi^{-1}:\mathcal{M_{\mathit{W}}}\rightarrow\mathcal{M_{\mathit{W}}}$
is an algebra isomorphism, which is given by 
\[
\varphi\circ\psi^{-1}\left(g\right)=g\circ G_{\psi^{-1}}\circ G_{\varphi},\quad g\in\mathcal{M}_{W}\cdot
\]
\end{prop}

\begin{thm}
\label{Theorem 7.3}Let $V$ and $W$ be two varieties in $B_{d}$.
Assume that $V$ contains at least two points, and that there exists
a multiplier algebra isomorphism 
\[
\varphi:\mathcal{M_{\mathit{V}}}\rightarrow\mathcal{M_{\mathit{W}}}\cdot
\]
Then,
\end{thm}

\[
\frac{1}{\left\Vert \left(\varphi^{-1}\right)^{\left(d\right)}\right\Vert }\leq\left\Vert G_{\varphi}\right\Vert _{Mult\left(W\right)}\leq\left\Vert \varphi^{\left(d\right)}\right\Vert \cdot
\]

\begin{proof}
Follows from Theorem \ref{Thm 07.15} and Proposition \ref{Proposition 8.2}. 
\end{proof}
\begin{cor}
\label{Corollary 7.4}Let $V$and $W$ be two varieties in $B_{d}$.
Suppose that there exists a RKHS isomorphism $T:H_{V}\longrightarrow H_{W}$
determined by
\[
T\left(k_{x}\right)=\lambda\left(x\right)k_{F(x)},\quad x\in V.
\]
Then, $T$ induces a multiplier algebra isomorphism $\varphi:\mathcal{M_{\mathit{V}}}\rightarrow\mathcal{M_{\mathit{W}}}$,
which is given by 
\[
\varphi\left(f\right)=f\circ F^{-1},\quad f\in\mathcal{M}_{V}\cdot
\]
Moreover, if $V$contains at least two points, then
\end{cor}

\[
\frac{1}{\left\Vert \varphi^{-1}\right\Vert _{cb}}\leq\frac{1}{\left\Vert \left(\varphi^{-1}\right)^{\left(d\right)}\right\Vert }\leq\left\Vert F^{-1}\right\Vert _{Mult\left(W\right)}\leq\left\Vert \varphi^{\left(d\right)}\right\Vert \leq\left\Vert \varphi\right\Vert _{cb}\leq\left\Vert T\right\Vert \left\Vert T^{-1}\right\Vert \cdot
\]

\begin{proof}
The assertion follows by Proposition \ref{Proposition 1.29}and Theorem
\ref{Theorem 7.3}.
\end{proof}
\begin{thm}
\label{Theorem 7.5}If $F:V\longrightarrow W$ is a multiplier biholomorphism
of $V$onto $W$, with $\left|V\right|\geq2$.\textbf{ }Then, the
following are equivalent.
\end{thm}

\begin{enumerate}
\item $F$ is a restriction of an automorphism of $B_{d}$ to $V$.
\item $\left\Vert F\right\Vert _{Mult\left(V\right)}=1$ and $\left\Vert F^{-1}\right\Vert _{Mult\left(W\right)}=1$.
\end{enumerate}
\begin{proof}
If $F:V\longrightarrow W$ is a restriction of an automorphism $B_{d}$,
then by \cite[Proposition 4.1]{davidson2015operator} the map $\varphi:\mathcal{M_{\mathit{V}}}\rightarrow\mathcal{M_{\mathit{W}}}$,
which is given by 
\[
\varphi\left(f\right)=f\circ F^{-1},\quad f\in\mathcal{M}_{V},
\]
 is a unital completely isometric multiplier algebra isomorphism.
That is, $\left\Vert \varphi\right\Vert _{cb}=1$ and $\left\Vert \varphi^{-1}\right\Vert _{cb}=1$.
Therefore, by Theorem \ref{Theorem 7.3} it follows that $\left\Vert F\right\Vert _{Mult\left(V\right)}=1$
and $\left\Vert F^{-1}\right\Vert _{Mult\left(W\right)}=1$. For the
converse, see the proof of \cite[Lemma 4.4+Theorem 4.5]{davidson2015operator}.
\end{proof}
For $\lambda\in\mathbb{T}$, let $U_{\lambda}:\mathbb{C^{\mathit{d}}}\longrightarrow\mathbb{C^{\mathit{d}}}$
be the unitary defined by $U_{\lambda}z=\lambda z$ for all $z\in\mathbb{C}^{d}$.
\begin{prop}
\label{Proposition 7.6}Let $V$and $W$ be two homogeneous varieties
in $B_{d}$. Suppose that $\varphi:\mathcal{M_{\mathit{V}}}\rightarrow\mathcal{M_{\mathit{W}}}$
is an algebra isomorphism. Then, there are $\lambda,\mu\in\mathbb{T}$
such that the biholomorphism
\[
G_{\varphi^{-1}}\circ U_{\lambda}\circ G_{\varphi}\circ U_{\mu}\circ G_{\varphi^{-1}}:V\to W
\]
fixes the origin. Moreover, there is a linear map $A:\mathbb{C^{\mathit{d}}}\longrightarrow\mathbb{C^{\mathit{d}}}$
such that 
\[
A\mid_{V}=G_{\varphi^{-1}}\circ U_{\lambda}\circ G_{\varphi}\circ U_{\mu}\circ G_{\varphi^{-1}}\cdot
\]
\end{prop}

\begin{proof}
This follows by a standard application of the \textquotedbl disc trick\textquotedbl{}
(See \cite[Proposition 4.7]{davidson2011isomorphism} and \cite[Lemma 9.6]{hartz2017isomorphism}).
The moreover part follows from \cite[Theorem 7.4]{davidson2011isomorphism}.
\end{proof}
\begin{thm}
\label{Theorem 7.7}Let $V$and $W$ be two homogeneous varieties
in $B_{d}$. Suppose that $\varphi:\mathcal{M_{\mathit{V}}}\rightarrow\mathcal{M_{\mathit{W}}}$
is an algebra isomorphism. Then, there is a linear map $A:\mathbb{C^{\mathit{d}}}\longrightarrow\mathbb{C^{\mathit{d}}}$
such that $A\left(V\right)=W$ with
\[
\left\Vert A\right\Vert _{Mult\left(V\right)}\leq\left\Vert \left(\varphi^{-1}\right)^{\left(d\right)}\right\Vert ^{2}\left\Vert \varphi^{\left(d\right)}\right\Vert \cdot
\]
\end{thm}

\begin{proof}
By Proposition, it follows that there is an invertible linear map
$A:\mathbb{C^{\mathit{d}}}\longrightarrow\mathbb{C^{\mathit{d}}}$
with $A\left(V\right)=W$, and $\lambda,\mu\in\mathbb{T}$ such that
\[
A\mid_{V}=G_{\varphi^{-1}}\circ U_{\lambda}\circ G_{\varphi}\circ U_{\mu}\circ G_{\varphi^{-1}},
\]
where $U_{\lambda},U_{\mu}:\mathbb{C^{\mathit{d}}}\longrightarrow\mathbb{C^{\mathit{d}}}$
are the unitaries defined above. Applying Claim \ref{claim: 8.1}
yields that
\begin{align*}
\left\Vert A\right\Vert _{Mult\left(V\right)} & \leq\left\Vert \left(\varphi^{-1}\right)^{\left(d\right)}\right\Vert \left\Vert G_{\varphi^{-1}}\circ U_{\lambda}\circ G_{\varphi}\circ U_{\mu}\right\Vert _{Mult\left(W\right)}\\
 & =\left\Vert \left(\varphi^{-1}\right)^{\left(d\right)}\right\Vert \left\Vert G_{\varphi^{-1}}\circ U_{\lambda}\circ G_{\varphi}\right\Vert _{Mult\left(W\right)}\\
 & \leq\left\Vert \left(\varphi^{-1}\right)^{\left(d\right)}\right\Vert \left\Vert \varphi^{\left(d\right)}\right\Vert \left\Vert G_{\varphi^{-1}}\circ U_{\lambda}\right\Vert _{Mult\left(V\right)}\\
 & =\left\Vert \left(\varphi^{-1}\right)^{\left(d\right)}\right\Vert \left\Vert \varphi^{\left(d\right)}\right\Vert \left\Vert G_{\varphi^{-1}}\right\Vert _{Mult\left(V\right)}\\
 & \leq\left\Vert \left(\varphi^{-1}\right)^{\left(d\right)}\right\Vert ^{2}\left\Vert \varphi^{\left(d\right)}\right\Vert \cdot
\end{align*}
\end{proof}
\section*{Acknowledgements}
This paper is based on the research thesis entitled "Deformations of Reproducing Kernel Hilbert Spaces on Homogeneous Varieties," submitted to the Senate of the Technion-Israel Institute of Technology in partial fulfillment of the requirements for the degree of Doctor of Philosophy. I gratefully acknowledge the financial support provided by the Technion during my Ph.D. studies. I am deeply thankful to my supervisor, Professor Orr Moshe Shalit, for his guidance, support, and insightful advice throughout the course of this research. I would also like to express my sincere thanks and gratitude to my parents for their deep love and support.

\bibliographystyle{plain}

\end{document}